%% file: paper_1_arxiv.tex
\documentclass[a4paper, 10pt, DIV15, headinclude=false, footinclude=false, leqno]{scrartcl}

\usepackage[T1]{fontenc}
\usepackage{lmodern}
\usepackage[utf8]{inputenc}
\usepackage[english]{babel}
\usepackage{csquotes}
\usepackage{url}
\usepackage{amsmath,amssymb,amsthm}
\usepackage{mathtools}
\usepackage{xfrac}
\usepackage{sidecap} 
\usepackage[
   algo2e,
   lined,
   boxed,
  commentsnumbered,
  linesnumbered,
  ruled
 ]{algorithm2e}
\usepackage{hyperref}

\usepackage{pgfplots}

\usepackage{pgfkeys}
    \newenvironment{customlegend}[1][]{%
        \begingroup
        \csname pgfplots@init@cleared@structures\endcsname
        \pgfplotsset{#1}%
    }{%
        \csname pgfplots@createlegend\endcsname
        \endgroup
    }%
    \def\addlegendimage{\csname pgfplots@addlegendimage\endcsname}

\pgfplotsset{compat=1.14}
\usepgfplotslibrary{groupplots}
\usepgfplotslibrary{dateplot}
\usepgfplotslibrary{units}
\usetikzlibrary{spy,backgrounds}
\usepackage{pgfplotstable}

\usepackage{subcaption} 
\usepackage{enumerate}
\usepackage{todonotes}
\usepackage{color}
\usepackage{setspace}
\definecolor{sixclassRdYlBu1}{rgb}{0.84,0.19,0.15}
\definecolor{sixclassRdYlBu2}{rgb}{0.99,0.55,0.35}
\definecolor{sixclassRdYlBu3}{rgb}{1.0,0.88,0.56}
\definecolor{sixclassRdYlBu4}{rgb}{0.88,0.95,0.97}
\definecolor{sixclassRdYlBu5}{rgb}{0.57,0.75,0.86}
\definecolor{sixclassRdYlBu6}{rgb}{0.27,0.46,0.71}

\DeclarePairedDelimiterX\skal[2]{(}{)}{#1\,,\,#2}

\theoremstyle{plain}
\newtheorem{lemma}{Lemma}[section]
\newtheorem{theorem}[lemma]{Theorem}
\newtheorem{remark}[lemma]{Remark}

\newtheorem{definition}[lemma]{Definition}
\newtheorem{proposition}[lemma]{Proposition}
\newtheorem{assumption}{Assumption}

\newcommand{\hS}[1]{\hspace{#1pt}}

\newcommand{\R}{\mathbb{R}}
\newcommand{\N}{\mathbb{N}}

\newcommand{\J}{\mathcal{J}}
\newcommand{\Jhat}{\hat{\mathcal{J}}}
\newcommand{\Jnoncor}{\hat{J}}
\newcommand{\cJhatn}{{{\Jhat_\red}}}

\newcommand{\HH}{\mathcal{H}}

\newcommand{\HHhat}{\hat{\mathcal{H}}}

\newcommand{\pr}{\textnormal{pr}}
\newcommand{\du}{\textnormal{du}}
\newcommand{\dred}[1]{\tilde d_{#1}}
\newcommand{\gradred}[1]{\widetilde\nabla_{#1}}
\newcommand{\noncorgrad}{\widetilde\nabla}
\newcommand{\cont}[1]{\gamma_{#1}}

\newcommand{\bformd}{a_{\mu}}
\newcommand{\lformd}{l_{\mu}}
\newcommand{\kformd}{k_{\mu}}
\newcommand{\jformd}{j_{\mu}}
\newcommand{\resd}{r_{\mu}}

\DeclareMathOperator{\integralend}{d}
\newcommand{\intend}[1]{\integralend \hS{-1.25} #1}
\newcommand{\dx}{\intend{x}}
\newcommand{\ds}{\intend{s}}

\newcommand{\Params}{\mathcal{P}}

\newcommand{\red}{{r}}

\newcommand{\Proj}{\text{P}}

\makeatletter
\newcommand{\labeltext}[2]{%
  \@bsphack
  \csname phantomsection\endcsname 
  \def\@currentlabel{#1}{\label{#2}}%
  \@esphack
}
\makeatother

\allowdisplaybreaks

\DeclareOldFontCommand{\rm}{\normalfont\rmfamily}{\mathrm}
\begin{document}

\title{A non-conforming dual approach for adaptive Trust-Region Reduced Basis approximation of PDE-constrained optimization\footnote{{\textbf{Funding:}} The authors acknowledge funding by the Deutsche Forschungsgemeinschaft for the project {\em Localized Reduced Basis Methods for PDE-constrained Parameter Optimization}
under contracts OH 98/11-1; SCHI 1493/1-1; VO 1658/6-1. Tim Keil, Mario Ohlberger and Felix Schindler
acknowledge funding by the Deutsche Forschungsgemeinschaft under Germany’s Excellence Strategy EXC 2044 390685587, Mathematics M\"unster: Dynamics -- Geometry -- Structure.
}
}
\author{Tim Keil$^\dag$, Luca Mechelli$^\ddag$, Mario Ohlberger$^\dag$,\\Felix Schindler\thanks{Mathematics Münster, Westfälische Wilhelms-Universität Münster, Einsteinstr. 62, D-48149 M\"unster, \url{{tim.keil,mario.ohlberger,felix.schindler,}@uni-muenster.de}}, Stefan Volkwein\thanks{Department of Mathematics and Statistics, University of Konstanz, D-78457 Konstanz, \url{{luca.mechelli,stefan.volkwein}@uni-konstanz.de}}}

\maketitle
\begin{abstract}
In this contribution we propose and rigorously analyze new variants of adaptive Trust-Region methods for parameter optimization with PDE constraints and bilateral parameter constraints.
The approach employs successively enriched Reduced Basis surrogate models that are constructed during the outer optimization loop and used as model function for the Trust-Region method. Each Trust-Region sub-problem is solved with the projected BFGS method. Moreover, we propose a non-conforming dual (NCD) approach to improve the standard RB approximation of the optimality system.
Rigorous improved a posteriori error bounds are derived and used to prove convergence of the resulting NCD-corrected adaptive Trust-Region Reduced Basis algorithm.
Numerical experiments demonstrate that this approach enables to reduce the computational demand for large scale or multi-scale PDE constrained optimization problems significantly.
\par\vskip\baselineskip\noindent
\par\vskip\baselineskip\noindent
\end{abstract}
%



\section*{Introduction}
\label{sec:introduction}

We are concerned with the development and rigorous analysis of novel efficient model order reduction methods for parameter optimization constrained by coercive variational state equations using the {\em first optimize,  then discretize} approach.
The methods are based on successive enrichment of the underlying reduced order models within the framework of Trust-Region optimization. 
Optimization problems constrained by partial differential equations (PDEs) arise in many fields of application 
in engineering and across all sciences. 
Examples of such problems include optimal (material) design or optimal control of processes and inverse problems, 
where parameters of a PDE model are unknown and need to be estimated from measurements. 
The numerical solution of such problems is very challenging as the underlying PDEs have to be solved repeatedly 
within outer optimization algorithms and the dimension of the parameters that need to be optimized might be 
very high or even infinite dimensional. 
PDE constrained optimization problems have been of interest for many decades. Classically, the underlying 
PDE (forward problem) is approximated by a high dimensional full order model (FOM) that results from discretization, e.g.,  
by the Finite Element or Finite Volume method.
Hence, the complexity of the optimization problem directly depends on 
the numbers of degrees of freedom (DOF) of the FOM. 
Mesh adaptivity has been advised to minimize the number of DOFs; see, e.g., 
\cite{BKR2000,MR2556843,Clever2012,MR2905006,MR1887737,MR2892950} and the references therein. 

\textbf{Model order reduction for PDE constrained optimization and optimal control.} 
A more recent  
approach is the usage of model order reduction (MOR) methods in 
order to replace the FOM by a surrogate reduced order model (ROM) of possibly very low dimension. 
MOR is a very active research field that has seen 
tremendous development in recent years, both from a theoretical and application  
point of view. For an introduction and overview 
we refer to the monographs and collections
\cite{BCOW2017,BOP+2017,HRS2016,QMN2016}. 
A particular promising model reduction approach for parameterized partial differential equations (pPDEs)
is the Reduced Basis (RB) method that relies on the approximation of the solution manifold of pPDEs
by low dimensional linear approximation spaces that are spanned from suitably selected particular solutions, 
called snapshots. 
A posteriori error estimation for solutions of the ROM with respect to the FOM
is the basis for efficient Greedy algorithms to select the snapshots in a quasi-optimal way
\cite{BCD+2011,Haasdonk2013}. Alternatively, construction of reduced bases using proper orthogonal decomposition 
(POD) may be used \cite{GV17}. 
The construction of a reduced basis and the respective projected ROM is generally called the offline phase, 
whereas evaluating the ROM is called online phase.

There exists a large amount of literature using such reduced order surrogate models for optimization methods. 
A posteriori error estimates for reduced approximation of
linear-quadratic optimization problems and parametrized
optimal control problems with control constraints were studied, e.g., in \cite{Dede2012,GK2011,KTV13,NRMQ2013,OP2007}.
In \cite{dihl15} an RB approach is proposed which also enables
an estimation on the actual error on the control variable and not only on the gradient of the output functional.   
Certified Reduced Basis methods for parametrized elliptic
optimal control problems with distributed controls were studied in \cite{KTGV2018}.
With the help of an a posteriori error estimator, ROMs
can be constructed with respect to a desired accuracy but also with respect to a local area in the parameter set \cite{EPR2010,HDO2011}.
For very high dimensional parameter sets, simultaneous parameter and state reduction has been advised \cite{HO2014,HO2015,LWG2010}.
However, constructing a reduced order surrogate for a prohibitively expensive
forward problem can also take a significant amount of computational resources. To remedy this, it is
beneficial to use optimization methods that optimize on a local level of the control variable, assuming
the surrogate only to be accurate enough in the respective parameter region.
Hence, we require an approach which goes beyond the classical offline/online decomposition.
Recently, RB methods have been advised with a progressive construction of ROMs \cite{BMV2020,GHH2016,ZF2015}. 
Also localized RB methods that are based on efficient localized a posteriori error control and online enrichment \cite{BEOR2017,OS2015} 
overcome traditional offline/online splitting and are thus particularly well suited for applications in optimization or inverse 
problems \cite{OSS2018,OS2017}.

\textbf{Trust-Region reduced order models for second-order methods.}
Trust-Region (TR) approaches are a class of optimization methods that are advantageous for the usage
of locally accurate surrogate models. The main idea is to solve optimization sub-problems only in a local
area of the parameter set which resolves the burden of constructing a global RB space.
The problem that might occur is the fact that during this minimization one usually moves away from the original parameters on which 
the reduced order model was built, and the quality of the reduced model cannot be guaranteed anymore. 
For that reason, a priori and a posteriori error analysis are required to ensure accurate reduced order 
approximations for the optimization problem; cf.~\cite{GV17,HV08,KTV13}.
In \cite{AFS00,SS13} a TR approach was proposed to control the quality of the (POD) reduced 
order model, referred to as TR-POD, a meanwhile well-established method in applications; cf.~\cite{BC08,CAN12}.

TR methods ensure global convergence for locally convergent methods.
In each iteration of the TR algorithm the nonlinear objective is replaced by a model function which can be optimized with much less effort; cf.~\cite{CGT00,NW06}.
One suitable choice for the model is a reduced order discretization of the objective (e.g., by utilizing a second-order Taylor approximation).
To ensure convergence to stationary points the accuracy of the model function and of its gradient have to be monitored.
In \cite{RoggTV17} a posteriori error bounds are utilized to monitor the approximation quality of the gradient.
We also refer to \cite{GGMRV17}, where the authors utilize basis update strategies to improve the reduced order approximation scheme with respect to the optimization goal.
The TR strategy can be combined with second-order methods for nonlinear optimization: with the Newton method to solve the reduced problem and with the SQP method for the all-at-once approach; cf.~\cite{HV06}. 

Constraints on the control and the metric for the Trust-Region radius can affect the convergence of the method.
For an error-aware TR method, the TR radius is directly characterized by the a posteriori error estimator for the 
cost functional of the surrogate model. Thus, the offline phase of the RB method can completely be omitted since the RB model can 
be adaptively enriched during the outer optimization loop. With this procedure the surrogate model eventually 
will have a high accuracy around the optimum of the optimization problem, ignoring the accuracy of the part 
which the outer (and inner) optimization loop does not approach at all. Error aware TR-RB methods can be utilized in many
different ways. 
One possible TR-RB approach has been extensively studied
in \cite{QGVW2017} for linear parametric elliptic equations, which ensures convergence of the nonlocal TR-RB. 
Note that the experiments in \cite{QGVW2017} are for up to six dimensional parameter sets without inequality constraints.
In \cite{YM2013}, the TR framework is combined with an efficient RB error bound for defining the Trust-Region in the design optimization of vibrating structures using frequency domain formulations.

\textbf{Main results.} 
In this contribution we present several significant advances for adaptive 
Trust-Region Reduced Basis optimization methods for parameterized partial differential equations:
\begin{itemize}
\item For the model function in the TR-RB approach, we follow a non-conforming dual (NCD) approach by choosing as model function the Lagrangian associated to the optimization problem. This permits more accurate results in terms of approximation of the optimal solution;
\item we provide efficiently computable a posteriori error estimates for all reduced quantities for different choices of the cost functional and its (approximate) gradient;
\item we rigorously prove the convergence of the TR-RB method with bilateral inequality constraints on the 
parameters; 
\item we devise several new adaptive enrichment strategies for the progressive construction of the 
Reduced Basis spaces;
\item we demonstrate in numerical experiments that our new TR-RB methods outperform existing model reduction approaches for large scale optimization problems in well defined benchmark problems. 
\end{itemize}

\textbf{Organization of the article.} 
In Section~\ref{sec:problem} we introduce the PDE constrained optimization problem and state first- and second-order optimality conditions. These serve as a basis for the full order discretization 
derived in Section~\ref{sec:mor}. 
Moreover, in Section~\ref{sec:mor} we introduce different strategies of model order reduction for the full order model and 
derive rigorous a posteriori error estimates for all equations, functionals, and gradient information. 
Section~\ref{sec:TRRB_and_adaptiveenrichment} is devoted to the derivation of Trust-Region -- Reduced Basis methods and the presentation of the convergence analysis of the adaptive TR-RB algorithm. 
In addition, we discuss in detail several variants of new TR-RB algorithms that differ in their respective reduced gradient information as well as in the enrichment strategies for the 
construction of the corresponding reduced models.  
All variants are thoroughly analyzed numerically in Section~\ref{sec:num_experiments}, where we consider three well defined benchmark problems. We also compare with selected state of the art optimization methods from the literature. 

\section{Problem formulation}
\label{sec:problem}

Given $\mu_\mathsf{a},\mu_\mathsf{b}\in\mathbb{R}^P$ with $P \in \N$ we consider the compact and convex admissible parameter set 
\[ 
\Params:= \left\{\mu\in\mathbb{R}^P\,|\,\mu_\mathsf{a} \leq \mu \leq \mu_\mathsf{b} \right\} \subset \R^P,
\] 
where $\leq$ is understood component-wise. Let $V$ be a real-valued Hilbert space with
inner product $(\cdot \,,\cdot)$ and induced norm $\|\cdot\|$. We are interested in efficiently approximating
PDE-constrained parameter optimization problems with the quadratic continuous cost functional
\[
\J: V \times \Params \to \R, \quad  (u,\mu) \mapsto \J(u, \mu) = \Theta(\mu) + j_\mu(u) + k_\mu(u, u),
\]
where $\Theta: \Params \to \mathbb{R}$ denotes a parameter functional and, for each $\mu \in \Params$, $j_\mu \in V'$ is a parameter-dependent continuous linear functional and $k_\mu: V \times V \to \R$ a continuous symmetric bilinear form. To be more precise, we consider the following constrained minimization problem:
{\color{white}
	\begin{equation}
	\tag{P}
	\label{P}
	\end{equation}
}\vspace{-45pt} 
\begin{subequations}\begin{equation}
	\min_{(u,\mu)\in V\times\Params} \J(u, \mu), 
	\tag{P.a}\label{P.argmin} \end{equation} 
	subject to $(u,\mu)$ satisfying the \emph{state -- or primal -- equation}
	\begin{align}
	a_\mu(u, v) = l_\mu(v) && \text{for all } v \in V,
	\tag{P.b}\label{P.state}
	\end{align}\end{subequations}%
\setcounter{equation}{0}
where, for each $\mu \in \Params$, $a_\mu: V \times V \to \R$ denotes a continuous and coercive symmetric bilinear form and $l_\mu \in V'$ denotes a continuous linear functional. For given $u \in V$, $\mu \in \Params$, we introduce the primal residual $r_\mu^\pr(u) \in V'$ associated with \eqref{P.state} by
\begin{align}
r_\mu^\pr(u)[v] := l_\mu(v) - a_\mu(u, v) &&\text{for all }v \in V.
\label{eq:primal_residual}
\end{align}
The primal residual plays a crucial role for a posteriori error analysis and for sensitivities of solution maps. 
\begin{remark}
	The Lagrange functional for \eqref{P} is given by $\mathcal{L}(u,\mu,p) = \J(u,\mu) + r_\mu^\pr(u)[p]$ for $(u,\mu)\in V\times\Params$ and for $p\in V$.
\end{remark}
To apply RB methods efficiently, we require the parametrization of the problem to be separable from $V$ throughout the work. This separability is a standard assumption for RB methods and can be circumvented by using empirical interpolation techniques \cite{BMNP2004,CS2010,DHO2012}. 

\begin{assumption}[Parameter-separability]
	\label{asmpt:parameter_separable}
	We assume $a_\mu$, $l_\mu$, $j_\mu$, $k_\mu$ to be parameter separable with $\Xi^a, \Xi^l, \Xi^j, \Xi^k \in \N$ non-parametric components $a_\xi: V \times V \to \R$ for $1 \leq \xi \leq \Xi^a$, $l_\xi \in V'$ for $1 \leq \xi \leq \Xi^l$, $j_\xi \in V'$ for $1 \leq \xi \leq \Xi^j$ and $k_\xi: V \times V \to \R$ for $1 \leq \xi \leq \Xi^k$, and respective parameter functionals $\theta_\xi^a, \theta_\xi^l, \theta_\xi^j, \theta_\xi^k: \Params \to \mathbb{R}$, such that
	\begin{align*}
	a_\mu(u, v) &= \sum_{\xi = 1}^{\Xi^a} \theta_\xi^a(\mu)\, a_\xi(u, v)
	, &&&&& l_\mu(v) &= \sum_{\xi = 1}^{\Xi^l} \theta_\xi^l(\mu)\, l_\xi(v),
	\end{align*}
	and analogously for $j_\mu$ and $k_\mu$.
\end{assumption}
Due to Assumption~\ref{asmpt:parameter_separable}, all quantities which linearly depend on $a_\mu$, $l_\mu$, $j_\mu$ and $k_\mu$ (such as $\J$ or the primal residual) are also separable w.r.t.~the parameter.
Since we will use a Lagrangian ansatz for an explicit computation of derivatives, we require some notation that we use throughout this paper. 

\subsection{A note on differentiability}
\label{sec:differentiability}

If $\J: V \times \Params \to \R$ is Fr\'echet differentiable w.r.t.~each $\mu \in \Params$, for each $u \in V$ and each $\mu \in \Params$ there exists a bounded linear functional $\partial_\mu \J(u, \mu) \in \mathbb{R}^P$, such that the Fr\'echet derivative of $\J$ w.r.t.~its second argument in the direction of $\nu \in \mathbb{R}^P$ is given by $\partial_\mu \J(u, \mu) \cdot \nu$ (noting that the dual space of $\mathbb{R}^P$ is itself). We refer to $\partial_\mu \J(u, \mu)$ as the derivative w.r.t.~$\mu$. In addition, for $u \in V$, $\mu \in \Params$ we denote the \emph{partial derivative} of $\J(u, \mu)$ w.r.t.~the $i$-th component of $\mu$ by $d_{\mu_i}\J(u, \mu)$ for $1 \leq i \leq P$. Note that $d_{\mu_i}\J(u, \mu) = \partial_\mu \J(u, \mu)\cdot e_i $, where $e_i \in \R^P$ denotes the $i$-th canonical unit vector. Furthermore, we denote the \emph{gradient} w.r.t.~its second argument -- the vector of components $d_{\mu_i}\J(u, \mu)$ -- by the operator $\nabla_\mu \J: V \times \Params \to \R^P$. Similarly, if $\J$ is Fr\'echet differentiable w.r.t.~each $u \in V$, for each $u \in V$ and each $\mu \in \Params$ there exists a bounded linear functional $\partial_u \J(u, \mu)\in V'$, such that the Fr\'echet derivative of $\J$ w.r.t.~its first argument in any direction $v \in V$ is given by $\partial_u \J(u, \mu)[v]$. We refer to $\partial_u \J(u, \mu)$ simply as the derivative w.r.t.~$u$. 
If $\J$ is twice Fr\'echet differentiable w.r.t.~each $\mu \in \Params$, we denote its \emph{hessian} w.r.t.~its second argument by the operator $\HH_\mu \J: V \times \Params \to \R^{P \times P}$. 

We treat $a$, $l$, $j$ and $k$ in a similar manner, although, for notational compactness, we indicate their parameter-dependency differently.
For instance, interpreting the bilinear form $a$ as a map $a: V \times V \times \Params \to \R$, $\left(u, v, \mu\right) \mapsto a_\mu(u, v)$, we denote the Fr\'echet derivatives of $a$ w.r.t.~the first, second and third argument of said map in the direction of $w\in V$, $\nu \in \mathbb{R}^P$ by $\partial_u a_\mu(u, v)[w] \in \R$, $\partial_v a_\mu(u, v)[w] \in \R$ and $\partial_\mu a_\mu(u, v)\cdot \nu \in \R$, respectively.
Similarly, interpreting the linear functional $l$ as a map $l: V \times \Params \to \R$, $(v, \mu) \mapsto l_\mu(v)$, we denote the Fr\'echet derivatives of $l$ w.r.t.~the first and second argument of said map in the direction of $w \in V$, $\nu \in \mathbb{R}^P$ by $\partial_v l_\mu(v)[w] \in \R$ and $\partial_\mu l_\mu(v) \cdot \nu \in \R$, respectively. We omit the word Fr\'{e}chet when referring to the derivatives of $\J$, $a$, $l$, $j$ and $k$, in order to simplify the notation, unless it is strictly necessary to specify it.

We apply this notation for Fr\'echet  and partial derivatives for functionals and bilinear forms throughout this manuscript.
Note that we denote the derivatives w.r.t.~the symbol of the argument in the original definition of the functional or bilinear form, not w.r.t.~the symbol of the actual argument, i.e.~we use $\partial_u \J(u_\mu, \mu)$ for the derivative w.r.t.~the first argument, not $\partial_{u_\mu} \J(u_\mu, \mu)$ or $\partial_v a_\mu(u, p)$ for the derivative w.r.t.~the second argument, not $\partial_p a_\mu(u, p)$. Note also that, due to Assumption~\ref{asmpt:parameter_separable}, we can exchange the order of differentiation w.r.t.~$V$ and $\mathbb{R}^P$, i.e.~$\partial_u\big(\partial_\mu \J(u, \mu)\cdot\nu\big)[w] = \partial_\mu\big(\partial_u \J(u, \mu)[w]\big)\cdot\nu$. 

\begin{assumption}[Differentiability of $a$, $l$ and $\J$]
	\label{asmpt:differentiability}
	We assume $a_{\mu}$, $l_{\mu}$ and $\J$ to be twice Fr\'{e}chet differentiable w.r.t.~$\mu$. This obviously requires that all parameter-dependent coefficient functions in Assumption~{\rm{\ref{asmpt:parameter_separable}}} are twice differentiable as well.
\end{assumption}

\begin{remark}[Derivatives w.r.t.~$V$]
	\label{prop:gateaux_wrt_V}
	Due to the (bi-)linearity of $a$, $l$, $j$ and $k$, we can immediately compute their derivatives w.r.t.~arguments in $V$. For $u, v \in V$, $\mu \in \Params$, the derivatives of $a$, $l$ and $\J$ w.r.t.~arguments in $V$ in the direction of $w \in V$ are given, respectively, by
	\[
	\partial_u a_\mu(u, v)[w] = a_\mu(w, v),\hspace{1em}\partial_v a_\mu(u, v)[w] = a_\mu(u, w), \hspace{1em} \partial_v l_\mu(v)[w] = l_\mu(w), \hspace{1em}\partial_u \J(u, \mu)[w] = j_\mu(w) + 2 k_\mu(w, u).
	\]
\end{remark}
We compute the partial derivatives of $a$ and $l$ w.r.t.~the parameter by means of their separable decomposition.

\begin{remark}[Derivatives w.r.t.~$\Params$]
	\label{rem:gateaux_wrt_P}
	For $\mu \in \Params$, $u, v, \in V$ the derivatives of $a$ and $l$ w.r.t.~$\mu$ in the direction of $\nu \in \mathbb{R}^P$ are given by
	\begin{align}
	\partial_\mu a_\mu(u, v) \cdot \nu &= \sum_{\xi = 1}^{\Xi^a} \big(\partial_\mu \theta_\xi^a(\mu) \cdot \nu\big)\, a_\xi(u, v)&&\text{and}&&&
	\partial_\mu l_\mu(v) \cdot \nu &= \sum_{\xi = 1}^{\Xi^l} \big(\partial_\mu \theta_\xi^l(\mu) \cdot \nu\big)\, l_\xi(v),
	\notag
	\end{align}
	respectively, if $u, v$ do not depend on $\mu$.
	We also introduce the following shorthand notation for the derivative of functionals and bilinear forms w.r.t.~the parameter in the direction of $\nu \in \mathbb{R}^P$, e.g.~for $\mu \in \Params$ we introduce
	\begin{align}
	\partial_\mu l_\mu &\cdot \nu \in V'&&& v \mapsto \big(\partial_\mu l_\mu \cdot \nu\big)(v) &:= \partial_\mu l_\mu(v)\cdot \nu&&\text{and}
	\notag\\
	\partial_\mu a_\mu &\cdot \nu \in V \times V \to \R&&& u, v \mapsto \big(\partial_\mu a_\mu \cdot \nu\big)(u, v) &:= \partial_\mu a_\mu(u, v)\cdot \nu,
	\notag
	\end{align}
	and note that $\partial_\mu l_\mu$ and $\partial_\mu a_\mu$ are continuous and separable w.r.t.~the parameter, owing to Assumption~{\rm\ref{asmpt:parameter_separable}}.
\end{remark}

The bilinear form $a_\mu(\cdot\,,\cdot)$ is continuous and coercive for all $\mu\in\Params$. Thus we can define the bounded solution map $\mathcal{S}:\Params \to V$, $\mu \mapsto u_\mu:= \mathcal S(\mu)$, where $u_\mu$ is the unique solution to \eqref{P.state} for a given $\mu$. The Fr\'echet derivatives of $\mathcal{S}$ are a common tool for RB methods and optimization, e.g., for constructing Taylor RB spaces that consist of the primal solution as well as their sensitivities (see \cite{HAA2017}) or for deriving optimality conditions for \eqref{P} (see \cite{HPUU2009}).
\begin{proposition}[Fr\'{e}chet derivative of the solution map]
	\label{prop:solution_dmu_eta}
	Considering the solution map $\mathcal S:\Params \to V$, $\mu \mapsto u_\mu$, its Fr\'{e}chet derivative $d_{\nu} u_\mu \in V$ w.r.t.~a direction $\nu\in\mathbb{R}^P$ is the unique solution of
	\begin{align}\label{eq:primal_sens}
	a_\mu(d_{\nu}u_\mu, v) = \partial_\mu r_\mu^\pr(u_\mu)[v] \cdot \nu &&\text{for all } v \in V.
	\end{align}
\end{proposition}
\begin{proof}
	We refer to \cite{HPUU2009} for the proof of this result.
\end{proof}

\subsection{Optimal solution and optimality conditions}
\label{sec:first_order_optimality_conditions}
In this section, we discuss the existence of an optimal solution for problem \eqref{P}. Then, we characterize a locally optimal solution through first- and second-order optimality conditions. Throughout the paper, a bar indicates optimality.
\begin{theorem}[Existence of an optimal solution]
	\label{Thm:existence_optimalsolution}
	Problem \eqref{P} admits an optimal solution pair $(\bar u, \bar \mu)\in V\times\Params$, where $\bar u:= u_{\bar \mu}$ is the solution of \eqref{P.state} for the parameter $\bar\mu$.
\end{theorem}
\begin{proof}
	Note that the quantities involved in problem \eqref{P} satisfies Assumption~1.44 in \cite{HPUU2009}. Thus the existence follows from \cite[Theorem~1.45]{HPUU2009}.
\end{proof}
Let us introduce the reduced cost functional $\Jhat: \Params\mapsto \mathbb{R},\,\mu\mapsto \Jhat(\mu) := \J(u_\mu, \mu)= \J( \mathcal S(\mu),\mu)$. Then problem \eqref{P} is equivalent to the so-called reduced problem
\begin{align}
\min_{\mu \in \Params} \Jhat(\mu).
\tag{$\hat{\textnormal{P}}$}\label{Phat}
\end{align}
\begin{remark} \hfill
	\begin{enumerate}
		\item Since $r^\pr_\mu(u_\mu)[p] = 0$ for any $p\in V$, it follows that $\Jhat(\mu) = \mathcal L(u_\mu,\mu,p)$ for any $p\in V$.
		\item The cost functional $\Jhat$ is in general non-convex, thus the existence of a unique minimum for $\Jhat$ (and thus of $\J$) can not be guaranteed.
                \item Let $(\bar u, \bar \mu) \in V \times \Params$ be a local optimal solution to \eqref{P} with $\bar u := u_{\bar \mu}$ the solution of the primal equation \eqref{P.state} for the parameter $\bar\mu$. Then the following constraint qualification holds true: For any $f\in V'$ there exists a pair $(u,\mu) \in V\times \R^P$ solving
		\begin{align*}
                a_{\bar \mu}(u, v) - \partial_\mu r_{\bar \mu}^\pr(\bar u)[v] \cdot \mu = f(v) &&\text{for all } v \in V.
		\end{align*}
		\item Theorem~{\rm{\ref{Thm:existence_optimalsolution}}} does not provide any solution method.
	\end{enumerate}
	
\end{remark}
One can derive first-order necessary optimality conditions in order to compute candidates for a local optimal solution of \eqref{P}. We refer to \cite[Cor. 1.3]{HPUU2009} for a proof of the following result:
\begin{proposition}[First-order necessary optimality conditions for \eqref{P}]
	\label{prop:first_order_opt_cond}
	Let $(\bar u, \bar \mu) \in V \times \Params$ be a local optimal solution to \eqref{P}. Moreover, let Assumption~{\rm{\ref{asmpt:differentiability}}} hold true. Then there exists a unique Lagrange multiplier $\bar p\in V$ such that the following first-order necessary optimality conditions hold:
	\begin{subequations}
		\label{eq:optimality_conditions}
		\begin{align}
		r_{\bar \mu}^\pr(\bar u)[v] &= 0 &&\text{for all } v \in V,
		\label{eq:optimality_conditions:u}\\
		\partial_u \J(\bar u,\bar \mu)[v] - a_{\bar\mu}(v,\bar p) &= 0 &&\text{for all } v \in V,
		\label{eq:optimality_conditions:p}\\
		(\nabla_\mu \J(\bar u,\bar \mu)+\nabla_{\mu} r^\pr_{\bar\mu}(\bar u)[\bar p]) \cdot (\nu-\bar \mu) &\geq 0 &&\text{for all } \nu \in \Params. 
		\label{eq:optimality_conditions:mu}
		\end{align}	
	\end{subequations}
\end{proposition}
Note that \eqref{eq:optimality_conditions:u} resembles the state equation \eqref{P.state}. From \eqref{eq:optimality_conditions:p} we deduce the \emph{adjoint -- or dual -- equation} with unique solution $p_{\mu} \in V$ for a fixed $\mu \in \Params$, i.e.
\begin{align}
a_\mu(v, p_\mu) = \partial_u \J(u_\mu, \mu)[v]
= j_\mu(v) + 2 k_\mu(v, u_\mu)&&\text{for all } v \in V,
\label{eq:dual_solution}
\end{align}
given the solution $u_\mu \in V$ to the state equation \eqref{P.state}. From \eqref{eq:optimality_conditions:p} we observe that the variable $\bar p$ of the optimal triple solves the dual equation \eqref{eq:dual_solution} for $\bar \mu$. Similarly to the primal solution, we can consider the dual solution map $\mathcal A:\Params \to V$, $\mu \mapsto \mathcal A(\mu) := p_\mu$, where $p_\mu$ is the solution of \eqref{eq:dual_solution} for the parameter $\mu$. In particular, $\bar p = p_{\bar \mu}$. For given $u, p \in V$, we also introduce the dual residual $r_\mu^\du(u, p) \in V'$ associated with \eqref{eq:dual_solution} by
\begin{align}
r_\mu^\du(u, p)[v] := j_\mu(v) + 2k_\mu(v, u) - a_\mu(v, p)&&\text{for all }v \in V.
\label{eq:dual_residual}
\end{align}
In addition, from the dual equation \eqref{eq:dual_solution}, we obtain the following formulation for the dual sensitivities.
\begin{proposition}[Fr\'echet derivative of the dual solution map]
	\label{prop:dual_solution_dmu_eta}
	Considering the dual solution map $\mathcal A:\Params \to V$, $\mu \mapsto p_\mu$, its directional derivative $d_{\eta} p_\mu\in V$ w.r.t.~a direction $\eta \in \Params$ is the solution of
	\begin{equation} \label{eq:dual_sens}
	\begin{split}
	a_\mu(q, d_{\eta} p_\mu) &= -\partial_\mu a_\mu(q, p_\mu)\cdot \eta + d_\mu \partial_u \J(u_\mu, \mu)[q] \cdot \eta
	= \partial_\mu r_\mu^\du(u_\mu, p_\mu)[q] \cdot \eta + 2 k_\mu(q, d_{\eta}u_\mu)
	\end{split}
	\end{equation}
	for all $q \in V$, where the latter equality holds for quadratic $\J$ as in \eqref{P.argmin}.
\end{proposition}
\begin{proof}
	Note that $\mathcal A$ is well defined because the bilinear form $a_{\mu}(\cdot \, , \cdot)$ is continuous and coercive. For a proof of the other claims we refer to \cite{HPUU2009}, for instance.
\end{proof}
Furthermore, we can compute first-order derivatives of $\Jhat$.
\begin{proposition}[Gradient of $\Jhat$]
	\label{prop:grad_Jhat}
	For given $\mu\in\Params$, the gradient of $\Jhat$, $\nabla_\mu\Jhat: \Params \to \R^P$, is given by
	\begin{align*}
	\nabla_{\mu}\Jhat(\mu) &= \nabla_{\mu}\J(u_{\mu}, \mu)+\nabla_{\mu}r_\mu^\pr(u_{\mu})[p_{\mu}]= \nabla_\mu \Theta(\mu) + \nabla_\mu j_\mu(u_\mu)
	+  \nabla_\mu k_\mu(u_\mu, u_\mu) + \nabla_\mu l_\mu(p_\mu) - \nabla_\mu a_\mu(u_\mu, p_\mu).
	\end{align*}
\end{proposition}
\begin{proof}
	This follows from \eqref{eq:primal_residual}, \eqref{eq:primal_sens}, \eqref{eq:dual_solution} and \eqref{P.argmin}, cf.~\cite{HPUU2009}.
\end{proof}
\begin{remark}
	The proof of Proposition~{\rm{\ref{prop:grad_Jhat}}} relies on the fact that both $u_\mu$ and $p_\mu$ belong to the same space $V$; cf.~\emph{\cite{HPUU2009}}. In particular, for any $\mu\in\Params$, we have $\nabla_\mu \Jhat(\mu) = \nabla_\mu \mathcal L(u_\mu,\mu,p_\mu)$.
\end{remark}
For $\bar\mu$ satisfying the first-order necessary optimality conditions, we have that $\bar \mu$ is a stationary point of the cost functional $\Jhat$. Thus, $\bar \mu$ can be either a local minimum, a saddle point or a local maximum of the cost functional $\Jhat$ (and obviously the same relationship occurs between $(\bar u,\bar \mu)$ and $\J$). We thus consider second-order sufficient optimality conditions in order to characterize local minima of the functional $\Jhat$, requiring its hessian.
\begin{proposition}[Hessian of $\Jhat$]
	\label{prop:hessian_Jhat}
	The hessian of $\Jhat$, $\HHhat_\mu := \HH_\mu \Jhat: \Params \to \R^{P \times P}$, is determined by its application to a direction $\nu\in \mathbb{R}^P$, given by
	\begin{align}
	\HHhat_\mu(\mu) \cdot \nu
	= \nabla_\mu\Big(&\partial_u \J(u_\mu, \mu)[d_{\nu}u_\mu] + l_\mu(d_{\nu}p_\mu) - a_\mu(d_{\nu}u_\mu, p_\mu) - a_\mu(u_\mu, d_{\nu}p_\mu)
	\notag\\
	&+\big(\partial_\mu \J(u_\mu, \mu)+ \partial_\mu l_\mu(p_\mu)- \partial_\mu a_\mu(u_\mu, p_\mu)\big)\cdot \nu\Big),
	\notag
	\end{align}
	where $u_\mu, p_\mu \in V$ denote the primal and dual solutions, respectively. 
	For a quadratic $\J$ as in \eqref{P.argmin} the above formula simplifies to
	\begin{align}
	\HHhat_\mu(\mu) \cdot \nu
	= \nabla_\mu\Big(&j_\mu(d_{\nu}u_\mu) + 2k_\mu(d_{\nu}u_\mu, u_\mu) + l_\mu(d_{\nu}p_\mu) - a_\mu(d_{\nu}u_\mu, p_\mu) - a_\mu(u_\mu, d_{\nu}p_\mu)
	\notag\\
	&+\big(\partial_\mu \J(u_\mu, \mu)+ \partial_\mu l_\mu(p_\mu)- \partial_\mu a_\mu(u_\mu, p_\mu)\big)\cdot \nu\Big).
	\notag
	\end{align}
\end{proposition}
\begin{proof}
	See, e.g., \cite{HPUU2009} for the first part. The second one follows from Remark~\ref{prop:gateaux_wrt_V}.
\end{proof}
\begin{proposition}[Second-order sufficient optimality conditions]\label{prop:second_order}
	Let Assumption~{\rm{\ref{asmpt:differentiability}}} hold true. Suppose that $\bar \mu\in \Params$ satisfies the first-order necessary optimality conditions \eqref{eq:optimality_conditions}. If $\HHhat_\mu(\bar \mu)$ is positive definite on the \emph{critical cone} $\mathcal C(\bar\mu)$ at $\bar\mu\in\Params$, i.e., if $\nu \cdot(\HHhat_\mu(\bar \mu)\cdot \nu) > 0$ for all $\nu\in\mathcal C(\bar\mu)\setminus\{0\}$, with
	\begin{align*}
	\mathcal C(\bar\mu):= \big\{\nu\in\mathbb{R}^P\,\big|\, \exists \mu\in\Params,\,c_1>0: \nu = c_1(\mu-\bar \mu),\, \nabla_\mu \Jhat(\bar\mu)\cdot \nu = 0  \big\},
	\end{align*}
	then $\bar \mu$ is a strict local minimum of \eqref{Phat}.
\end{proposition} 
\begin{proof}
	For this result we refer to \cite{CasTr15,Nocedal06}, for instance.
\end{proof}

\section{High dimensional discretization and model order reduction}
\label{sec:mor}

We first discretize the optimization problem \eqref{P} as well as the corresponding optimality conditions 
using the classical Ritz-Galerkin projection into a possibly high dimensional approximation space $V_h \subset V$, such as conforming Finite Elements. 
Note that we restrict ourselves to a conforming approximation for simplicity and that we do not further specify the choice of $V_h$, as neither impacts the analysis below.
Based on this idea, we then derive different ways for the ROM using the Reduced Basis method with possibly different reduced primal and dual 
state spaces.
Thus, the resulting ROM optimality system will in general not be equivalent to a Ritz-Galerkin projection of the FOM one onto a reduce space $V_\red \subset V_h$. For this reason, we will introduce a non-conforming dual-corrected (NCD-corrected) approach; cf.~Section~\ref{sec:ncd_approach}.

\subsection{FOM for the optimality system} 
\label{sec:problem_fom}
For the discretization of the optimization problem we assume that a finite-dimensional subspace $V_h \subset V$ is given and obtain the FOM for the optimality system of \eqref{P} by Ritz-Galerkin projection of equations \eqref{eq:optimality_conditions} onto $V_h$. 
In particular, we have for each $\mu \in \Params$ the solution $u_{h, \mu} \in V_h$ of the \emph{discrete primal equation}
\begin{align}
\bformd(u_{h, \mu}, v_h) = \lformd(v_h) &&\text{for all } v_h \in V_h,
\label{eq:state_h}
\end{align}
and hence $\resd^\pr(u_{h, \mu})[v_h] = 0$ for all $v_h \in V_h$, $\mu \in \Params$.
We also have for each $\mu \in \Params$ the solution $p_{h, \mu} \in V_h$ of the \emph{discrete dual equation}
\begin{align}
\bformd(v_h, p_{h, \mu}) = \partial_u \J(u_{h, \mu}, \mu)[v_h] = \jformd(v_h) + 2 \kformd(v_h, u_{h, \mu}) &&\text{for all } v_h \in V_h,
\label{eq:dual_solution_h}
\end{align}
and hence $\resd^\du(u_{h, \mu}, p_{h, \mu})[v_h] = 0$ for all $v_h \in V_h$, $\mu \in \Params$. Similarly, the \emph{discrete primal sensitivity equations} for solving for $d_{\nu} u_{h, \mu} \in V_h$ as well as
\emph{discrete dual sensitivity equations} for solving for $d_{\nu} p_{h, \mu} \in V_h$ at any direction $\nu \in \mathbb{R}^P$
follow analogously to Propositions \ref{prop:solution_dmu_eta} and \ref{prop:dual_solution_dmu_eta}.
Furthermore, $\Jhat$ is approximated by the \emph{discrete reduced functional}
\begin{align}
\Jhat_h(\mu) := \J(u_{h, \mu}, \mu) = \mathcal L(u_{h,\mu},\mu,p_h) && \text{for all } p_h\in V_h,
\label{eq:Jhat_h}
\end{align}
where $u_{h, \mu} \in V_h$ is the solution of \eqref{eq:state_h} and we formulate the discrete optimization problem 
\begin{align}
\min_{\mu \in \Params} \Jhat_h(\mu).
\tag{$\hat{\textnormal{P}}_h$}\label{Phat_h}
\end{align}
Further, $\bar\mu_h$ denotes a locally optimal solution to \eqref{Phat_h} satisfying the first- and second-order optimality conditions. 
\begin{remark}
	Since $u_{h,\mu}$ and $p_{h,\mu}$ belong to the same space $V_h$, Propositions~{\rm{\ref{prop:first_order_opt_cond}-{\rm\ref{prop:grad_Jhat}},{\rm\ref{prop:hessian_Jhat}}-\ref{prop:second_order}}} from Section~{\rm{\ref{sec:problem_fom}}} hold for the FOM as well, with all quantities replaced by their discrete counterparts.
\end{remark}
As usual in the context of RB methods, we eliminate the issue of ``truth'' by assuming that the high dimensional space $V_h$ is accurate enough to approximate the true solution.

\begin{assumption}[This is the ``truth'']
	\label{asmpt:truth}
	We assume that the primal discretization error $\|u_\mu - u_{h, \mu}\|$, the dual error $\|p_\mu - p_{h, \mu}\|$, 
	the primal sensitivity errors $\|d_{\mu_i} u_\mu - d_{\mu_i} u_{h, \mu}\|$ and the dual sensitivity errors $\|d_{\mu_i} p_\mu - d_{\mu_i} p_{h, \mu}\|$
	are negligible for all $\mu \in \Params$, $1 \leq i \leq P$.
\end{assumption}

To define suitable ROMs, in what follows, we assume that we have computed problem adapted RB spaces $V_\red^\pr, V_\red^\du \subset V_h$, 
the construction of which is detailed in Section~\ref{sec:construct_RB}. We stress here that $V_\red^\pr$ and $V_\red^\du$ might not coincide, this will imply further discussions of the RB approximation of the optimality system \eqref{eq:optimality_conditions}.

\subsection{ROM for the optimality system -- Standard approach}
\label{sec:standard_approach}

Given a RB space $V_\red^\pr \subset V_h$ of low dimension $n := \dim V_\red^\pr$ and dual RB space $V_\red^\du \subset V_h$ of low dimension $m := \dim V_\red^\du$, we obtain the RB approximation of state and adjoint equations as follows:
\begin{subequations}
	\label{eq:optimality_conditionsRB}
	\begin{itemize}
		\item RB approximation for \eqref{eq:optimality_conditions:u}: For each $\mu \in \Params$ the primal variable $u_{\red, \mu} \in V_\red^\pr$ 
		of the \emph{RB approximate primal equation} is defined through
		\begin{align}
		\bformd(u_{\red, \mu}, v_\red) = \lformd(v_\red) &\qquad \text{for all } v_\red \in V_\red^\pr.
		\label{eq:state_red}
		\end{align}
		\item RB approximation for \eqref{eq:optimality_conditions:p}: For each $\mu \in \Params$, $u_{\red, \mu} \in V_\red^\pr$ the dual/adjoint variable $p_{\red, \mu} \in V_\red^\du$ satisfies the \emph{RB approximate dual equation} through
		\begin{align}
		\bformd(q_\red, p_{\red, \mu}) = \partial_u \J(u_{\red, \mu}, \mu)[q_\red] = \jformd(q_\red) + 2 \kformd(q_\red, u_{\red, \mu}) &&\text{for all } q_\red \in V_\red^\du.
		\label{eq:dual_solution_red}
		\end{align}
	\end{itemize}
\end{subequations}
Analogously to Propostion~\ref{prop:solution_dmu_eta}, we define the \emph{RB solution map} $\mathcal{S}_\red:\Params \to V_\red^\pr$ by $\mu \mapsto u_{\red, \mu}$ and analogously to Propostion~\ref{prop:dual_solution_dmu_eta} the \emph{RB dual solution map} $\mathcal{A}_\red:\Params \to V_\red^\du$ by $\mu \mapsto p_{\red, \mu}$, where $u_{\red, \mu}$ and $p_{\red, \mu}$ denote the primal and dual reduced solutions of \eqref{eq:state_red} and \eqref{eq:dual_solution_red}, respectively.

To approximate \eqref{Phat_h}, we introduce the \emph{RB reduced functional} by
\begin{align}
\Jnoncor_\red(\mu) := \J(u_{\red, \mu}, \mu)=\J(\mathcal S_r(\mu), \mu),&&
\text{where $u_{\red, \mu} \in V_\red^\pr$ is the solution of \eqref{eq:state_red}}
\label{eq:Jhat_red}
\end{align}
instead of $\Jhat_h$ and the 
problem of finding a locally optimal solution $\bar \mu_\red$ of
\begin{align}
\min_{\mu \in \Params} \Jnoncor_\red(\mu).
\label{eq:Phat_red_uncorrected}
\end{align}
Now, a solution to the optimality system \eqref{eq:optimality_conditions} is approximated by the RB triple $(u_{\red,\bar\mu_\red},\bar\mu_\red,p_{\red,\bar\mu_\red})$.

As proposed in \cite{QGVW2017}, for computing an approximation of the gradient of $\Jnoncor_\red$, the gradient from Propostion~\ref{prop:grad_Jhat} can be utilized by replacing $u_{\mu}$ and $p_{\mu}$ with their RB counterparts. However, it can not be guaranteed in general that the computed gradient is the actual gradient of $\Jnoncor_\red$, if $V^\pr_\red$ and $V^\du_\red$ are chosen to be different. To see this, we consider first the Lagrangian and note that, for $1\leq i\leq P$ and all $p\in V$, it holds
\begin{equation}
	\Jnoncor_\red(\mu)  = \mathcal{L}(u_{\red,\mu},\mu,p), \qquad
	\label{Gradient_lagrangian}
	 \big(\nabla_\mu \Jnoncor_\red(\mu)\big)_i  =  \partial_u \mathcal{L}(u_{\red,\mu},\mu,p)[d_{\mu_i}u_{\red,\mu}] 
	 + d_{\mu_i} \mathcal{L}(u_{\red,\mu},\mu,p).	
\end{equation}
Now, following \cite{QGVW2017}, we define the \emph{inexact gradient} $\noncorgrad_\mu \Jnoncor_\red:\Params \to \R^P$ by 
\begin{equation}
\label{naive:red_grad}
\big(\noncorgrad_\mu \Jnoncor_\red(\mu)\big)_i := d_{\mu_i}\J(u_{\red, \mu}, \mu) + d_{\mu_i} r_\mu^\pr(u_{\red, \mu})[p_{\red, \mu}] = d_{\mu_i} \mathcal{L}(u_{\red,\mu},\mu,p_{\red,\mu})
\end{equation}
for all $1 \leq i \leq P$ and $\mu \in \Params$, where $u_{\red, \mu} \in V_\red^\pr$ and $p_{\red, \mu} \in V_\red^\du$ denote the primal and approximate dual reduced solutions of \eqref{eq:state_red} and \eqref{eq:dual_solution_red}, respectively. With the superscript $\sim$ we stress that $\noncorgrad_\mu \Jnoncor_\red(\mu)$ is not the actual gradient of $\Jnoncor_\red$, but its approximation. Choosing $p= p_{\red,\mu}\in V^\du_\red$ in \eqref{Gradient_lagrangian} and considering \eqref{naive:red_grad} lead to
\begin{align*}
\big(\nabla_\mu \Jnoncor_\red(\mu)\big)_i & = \partial_u \mathcal{L}(u_{\red,\mu},\mu,p_{\red,\mu})[d_{\mu_i}u_{\red,\mu}] + \noncorgrad_\mu \Jnoncor_\red(\mu).
\end{align*}
Note that, in general, it does not hold that $\partial_u\mathcal L(u_{\red,\mu},\mu, p_{\red,\mu})= 0$, since \eqref{eq:dual_solution_red} is not the dual equation with respect to the optimization problem \eqref{eq:Phat_red_uncorrected}, cf.~\cite[Section~1.6.4]{HPUU2009}, which would only be true if $V^\du_\red\subseteq V^\pr_\red$.
Thus, \eqref{naive:red_grad} defines only an approximation of the true gradient of $\Jnoncor_\red$ with the choice made in \cite{QGVW2017}. This introduces an additional approximation error in reconstructing the solution of the optimality system \eqref{eq:optimality_conditions}, which is well visible in our numerical experiments (see Section~\ref{sec:mmexc_example}): the standard RB approach leads to a significant lack in accuracy, requiring additional steps to enrich the RB space and cover this gap. We therefore propose to add a correction term to $\Jnoncor_\red$ based on the previous remarks.

\subsection{ROM for the optimality system -- NCD-corrected approach} 
\label{sec:ncd_approach}
Following the primal-dual RB approach for linear output functionals \cite[Section~2.4]{HAA2017}, it is more suitable to add a correction term to the output functional for which improved error estimates are available. We seek to minimize the Lagrangian corresponding to problem \eqref{P}. A similar approach, in the context of adaptive finite elements, can be found in \cite{BKR2000,Rannacher2006}.
We utilize \eqref{eq:dual_solution_red} to extend the primal-dual RB approach of \cite[Section~2.4]{HAA2017} to quadratic output functionals and define the \emph{NCD-corrected RB reduced functional}
\begin{align}
\cJhatn(\mu) := \mathcal{L}(u_{\red,\mu},\mu,p_{\red,\mu}) = \Jnoncor_\red(\mu) + r_\mu^\pr(u_{\red,\mu})[p_{\red,\mu}]
\label{eq:Jhat_red_corected}
\end{align} 
with $u_{\red, \mu} \in V_\red^\pr$ and $p_{\red, \mu} \in V_\red^\du$ the solutions of \eqref{eq:state_red} and \eqref{eq:dual_solution_red}, respectively.
Note that $\cJhatn$ coincides with the functional $\Jnoncor_\red$ in \eqref{eq:Jhat_red} if $V_\red^\du = V_\red^\pr$.
We then consider the \emph{RB reduced optimization problem} of finding a locally optimal solution $\bar \mu_\red$ of
\begin{align}
\min_{\mu \in \Params} \cJhatn(\mu).
\tag{$\hat{\textnormal{P}}_\red$}\label{Phat_\red}
\end{align}

Computing the actual gradient of $\Jhat_\red$ results in the next proposition, proved following \cite[Section~1.6.2]{HPUU2009}. 
\begin{proposition}[Gradient of the NCD-corrected RB reduced functional]
	\label{prop:true_corrected_reduced_gradient_adj}
	The $i$-th component of the true gradient of $\cJhatn$ is given by
	\begin{align*}
	\big(\nabla_\mu \cJhatn(\mu)\big)_i & = d_{\mu_i}\J(u_{\red,\mu},\mu) + d_{\mu_i}r_\mu^\pr(u_{\red,\mu})[p_{\red,\mu}+w_{\red,\mu}] - d_{\mu_i}r^\du_\mu (u_{\red,\mu},p_{\red,\mu})[z_{\red,\mu}] 
	\end{align*}
	where $u_{\red, \mu} \in V_\red^\pr$ and $p_{\red, \mu} \in V_\red^\du$ denote the RB approximate primal and dual solutions of \eqref{eq:state_red} and \eqref{eq:dual_solution_red}, $z_{\red,\mu} \in V_\red^\du$ solves
	\begin{equation}
	\label{A2:z_eq}
	a_\mu(z_{\red,\mu},q) = -r_\mu^\pr(u_{\red,\mu})[q] \quad \forall q\in V^\du_r
	\end{equation}
	and $w_{\red,\mu} \in V_\red^\pr$ solves
	\begin{equation}
	\label{A2:w_eq}
	a_\mu(v,w_{\red,\mu}) = r_\mu^\du(u_{\red,\mu},p_{\red,\mu})[v]-2k_\mu(z_{\red,\mu},v), \quad \forall v\in V^\pr_r.
	\end{equation}
\end{proposition}

\subsection{A posteriori error analysis} 
\label{sec:a_post_error_estimates}
A posteriori error estimates are required for controlling the accuracy of the reduced order model. 
In addition, we also use them for the error aware TR method (which is explained in Section~\ref{sec:TR}). 
We derive a posteriori error estimates for all reduced terms that we need for the TR method. Moreover, we suggest further advances for the reduction of sensitivities and
gradients.  
From a model reduction perspective, these error estimates need to be computed efficiently such that the time for the evaluation for many parameters can be neglected. 
Note that Assumption~\ref{asmpt:parameter_separable} is crucial for this, since it allows to precompute most of the required terms.  
For any functional $l \in V_h'$ or bilinear form $a: V_h \times V_h \to \R$, we denote their dual or operator norms $\|l\|$ and $\|a\|$ by the continuity constants $\cont{l}$ and $\cont{a}$, respectively. The same consideration applies for the norm $\|\cdot\|$ in $V_h'$ of the residuals.
For $\mu \in \Params$, we denote the coercivity constant of $\bformd$ w.r.t.~the $V_h$-norm by $\underline{\bformd} > 0$. 

\subsubsection{Standard RB estimates for the optimality system} 
\label{sec:estimation_for_the_standard_and_corrected_reduced_order_model}

We start with the residual based a posteriori error estimation for the primal variable, which is a standard result from RB theory and has extensively been used in the literature. For a proof, we refer to \cite{rozza2007}.
\begin{proposition}[Upper bound on the primal model reduction error]	\label{prop:primal_rom_error}
	For $\mu \in \Params$ let $u_{h, \mu} \in V_h$ be the solution of \eqref{eq:state_h} and $u_{\red, \mu} \in V_\red^\pr$
	the solution of \eqref{eq:state_red}. Then it holds
	\begin{align}
	\|u_{h, \mu} - u_{\red, \mu}\| \leq \Delta_\pr(\mu)  := \underline{\bformd}^{-1}\, \|\resd^\pr(u_{\red, \mu})\|.
	\notag
	\end{align}
\end{proposition}
For the reduced dual problem, a similar idea can be used to derive the following estimation, accounting for the fact that $p_{\red, \mu}$ is not a Galerkin projection of 
$p_{h, \mu}$. For a proof, we refer to \cite[Lemma 3]{QGVW2017}.
\begin{proposition}[Upper bound on the dual model reduction error]
	\label{prop:dual_rom_error}
	For $\mu \in \Params$, let $p_{h, \mu} \in V_h$ be the solution of \eqref{eq:dual_solution_h} and $p_{\red, \mu} \in V_\red^\du$
	the solution of \eqref{eq:dual_solution_red}. Then it holds
	\begin{align}
	\|p_{h, \mu} - p_{\red, \mu}\| &\leq \Delta_\du(\mu) := \underline{\bformd}^{-1}\big(2 \cont{\kformd}\;\Delta_\pr(\mu) + \|\resd^\du(u_{\red, \mu}, p_{\red, \mu})\|\Big).
	\notag
	\end{align}
\end{proposition}

In the next proposition we state the result of the standard approach from \cite[Theorem~4]{QGVW2017}. Furthermore we show an improved version by using, in contrast to \cite{QGVW2017}, the NCD-corrected reduced functional, which results in an optimal higher order a posteriori upper bound without lower order terms. 

\begin{proposition}[Upper bound on the model reduction error of the reduced output]
	\label{prop:Jhat_error} 
	\hfill
	\begin{enumerate}[(i)]
		\item [\emph{(i)}] With the notation from above, we have for the standard RB reduced cost functional
		\begin{align}
		|\Jhat_h(\mu) - &\Jnoncor_r(\mu)| \leq \Delta_{\Jnoncor_\red}(\mu) :=  \Delta_\pr(\mu) \|\resd^\du(u_{\red, \mu}, p_{\red,\mu})\| + \Delta_\pr(\mu)^2 \cont{\kformd} + \big|r_{\mu}^\pr(u_{\red, \mu})[p_{\red,\mu}]\big|.
		\notag
		\end{align}		
		\item [\emph{(ii)}] Furthermore, we have for the NCD-corrected RB reduced cost functional (or equivalently for the Lagrangian for any $p\in V_h$)
		\begin{align}
		|\Jhat_h(\mu) - &\cJhatn(\mu)| = |\mathcal{L}(u_{h,\mu},\mu,p)-\mathcal{L}(u_{\red,\mu},\mu,p)| \leq \Delta_{\cJhatn}(\mu) :=  \Delta_\pr(\mu) \|\resd^\du(u_{\red, \mu}, p_{\red,\mu})\| + \Delta_\pr(\mu)^2 \cont{\kformd}.
		\notag
		\end{align}
	\end{enumerate}
\end{proposition}
\begin{proof}
  We refer to \cite[Theorem~4]{QGVW2017} for a proof of $(i)$.
  Regarding $(ii)$, using
      the shorthand $e_{h, \mu}^\pr := u_{h, \mu} - u_{\red, \mu}$ and $a_\mu(e_{h, \mu}^\pr,p_{\red,\mu}) = r_{\mu}^\pr(u_{\red, \mu})[p_{\red,\mu}]$ lead us to 
		\begin{align*}
		|\Jhat_h(\mu) - \cJhatn(\mu)| &= |j_{h, \mu}(e_{h, \mu}^\pr) + \kformd(u_{h, \mu}, u_{h, \mu}) - \kformd(u_{\red, \mu}, u_{\red, \mu}) - a_\mu(e_{h, \mu}^\pr,p_{\red,\mu})| \notag\\
		&= |r_{\mu}^\du(u_{\red, \mu}, p_{\red,\mu})[e_{h, \mu}^\pr] - 2\kformd(u_{\red, \mu},e_{h, \mu}^\pr)
		+ \kformd(u_{h, \mu}, u_{h, \mu}) - \kformd(u_{\red, \mu}, u_{\red, \mu})|  \notag \\
		&\leq \|r_{\mu}^\du(u_{\red, \mu}, p_{\red,\mu})\|\; \|e_{h, \mu}^\pr\| +\cont{\kformd} \; \|e_{h, \mu}^\pr\|^2,
		\notag
		\end{align*}
		where we used the definition of the dual residual in the second equality and Cauchy-Schwarz for the inequality. 
		The assertion follows by using Propostion~\ref{prop:primal_rom_error}. 
\end{proof}
\begin{remark}\label{continuity_of_estimator}
	The estimator $\Delta_{\cJhatn}(\mu)$ is continuous w.r.t.~$\mu$,
	since the Riesz-representative of the residual is continuous. 
\end{remark}
For the inexact and NCD-corrected gradient, we derive the following a posteriori estimators. 
\begin{proposition}[Upper bound on the model reduction error of the gradient of reduced output] 
	\label{prop:grad_Jhat_error}
	\hfill
	\begin{enumerate}[(i)]
		\item [\emph{(i)}] For the inexact gradient $\noncorgrad_\mu \Jnoncor_\red(\mu)$ from the standard-RB approach \eqref{naive:red_grad}, we have  
		\begin{align}
		\big\|\nabla_\mu \Jhat_h(\mu) - \noncorgrad_\mu \Jnoncor_\red(\mu)\big\|_2 &\leq \Delta_{\noncorgrad \Jnoncor_\red}(\mu) = \big\|\underline{\Delta_{\noncorgrad \Jnoncor_\red}(\mu)}\big\|_2 \quad\quad\text{with}
		\notag\\
		\big(\underline{\Delta_{\noncorgrad \Jnoncor_\red}(\mu)}\big)_i  
		:= 2\Delta_\pr(\mu) \|u_{\red, \mu}\|\; \cont{d_{\mu_i} \kformd}
		&+ \Delta_\pr(\mu)\big(\cont{d_{\mu_i} \jformd} + \cont{d_{\mu_i} \bformd}\; \|p_{\red, \mu}\|\big)
		\notag\\
      + \Delta_\du(\mu)\big(\cont{d_{\mu_i} \lformd} &+ \cont{d_{\mu_i} \bformd}\; \|u_{\red, \mu}\|\big)
		+ \Delta_\pr(\mu)\; \Delta_\du(\mu)\; \cont{d_{\mu_i} \bformd}
		+ (\Delta_\pr)^2(\mu)\; \cont{d_{\mu_i} \kformd}.
		\notag
		\end{align}
		\item [\emph{(ii)}] For the gradient $\nabla_{\mu} \cJhatn(\mu)$ of the NCD-corrected reduced functional, computed with the adjoint approach 
		from Definition~{\rm{\ref{prop:true_corrected_reduced_gradient_adj}}}, we have 
		\begin{align*}
		&\big\|\nabla_\mu \Jhat_h(\mu) - \nabla_{\mu} \cJhatn(\mu)\big\|_2 \leq \Delta^*_{\nabla \Jhat_\red}(\mu) 
    = \big\|\underline{\Delta^*_{\nabla \Jhat_\red}(\mu)}\big\|_2 \quad\quad\text{with} \\
      \big(\underline{\Delta^*_{\nabla \Jhat_\red}(\mu)}\big)_i &:=
		2\Delta_\pr(\mu) \|u_{\red, \mu}\|\; \cont{d_{\mu_i} \kformd}
		+ \Delta_\pr(\mu)\big(\cont{d_{\mu_i} \jformd} + \cont{d_{\mu_i} \bformd}\; \|p_{\red, \mu}\|\big)\\
		&+ \Delta_\du(\mu)\big(\cont{d_{\mu_i} \lformd} + \cont{d_{\mu_i} \bformd}\; \|u_{\red, \mu}\|\big)
		+ \Delta_\pr(\mu)\; \Delta_\du(\mu)\; \cont{d_{\mu_i} \bformd}
		+ (\Delta_\pr)^2(\mu)\; \cont{d_{\mu_i} \kformd} \\
		&+ (\cont{d_{\mu_i}l_{\mu}} + \cont{d_{\mu_i}a_{\mu}} \| u_{\red,\mu} \|) \underline{\bformd}^{-1} \big( \|r_\mu^\du(u_{\red,\mu},p_{\red,\mu})\| 
		+ 2\cont{k_\mu} \underline{\bformd}^{-1} \|r_\mu^\pr(u_{\red,\mu})\| \big) \\
		&+ \underline{\bformd}^{-1} \|r_\mu^\pr(u_{\red,\mu})\| \big( \cont{d_{\mu_i}j}+ 2\cont{d_{\mu_i}k} \|u_{\red,\mu}\|
		+ \cont{d_{\mu_i}a} \|p_{\red,\mu}\| \big).
		\end{align*}
	\end{enumerate}
\end{proposition}
\begin{proof}
(i) For $\Delta_{\noncorgrad \Jnoncor_\red}(\mu)$, we have 
		\begin{align*}
		\big(\nabla_\mu \Jhat_h(\mu) - \noncorgrad_{\mu} \cJhatn(\mu)\big)_i
		&=d_{\mu_i}\J(u_{h, \mu}, \mu) - d_{\mu_i}\J(u_{\red, \mu}, \mu)  + d_{\mu_i}r_\mu^\pr(u_{h,\mu})[p_{h,\mu}] - d_{\mu_i}r_\mu^\pr(u_{\red,\mu})[p_{\red,\mu}].
		\end{align*}
		Regarding the first contribution, we obtain with $\|u_{h, \mu}\|_h \leq \|e_{h, \mu}^\pr\|_h + \|u_{\red, \mu}\|_h$
		\begin{align*}
		\big|d_{\mu_i}\J(u_{h, \mu}, \mu) - d_{\mu_i}\J(u_{\red, \mu}, \mu)\big| &=
		| d_{\mu_i}j_{h, \mu}(e_{h, \mu}^\pr) +  d_{\mu_i}k_{\mu}(e_{h, \mu}^\pr, u_{\red, \mu}) +  d_{\mu_i}k_{\mu}(u_{h, \mu}, e_{h, \mu}^\pr)|
		\\
		&\leq \Delta_\pr(\mu)\Big( \cont{d_{\mu_i} \jformd} + \cont{d_{\mu_i} \kformd}\big(2 \|u_{\red, \mu}\| + \Delta_\pr(\mu)\big)\Big).
		\end{align*} 
		For the other contributions we refer to \cite[Theorem~5]{QGVW2017}.
		
		(ii) For the adjoint estimator $\Delta^*_{\nabla_\mu \Jhat_\red}$, we have 
		\begin{align*}
		\partial_\mu \cJhatn(\mu)\cdot\nu & = \partial_\mu \J(u_{\red,\mu},\mu)\cdot\nu + \partial_\mu r_\mu^\pr(u_{\red,\mu})[p_{\red,\mu}+w_{\red,\mu}]\cdot\nu
		- \partial_\mu r_\mu^\du(u_{\red,\mu},p_{\red,\mu})[z_{\red,\mu}]\cdot\nu
		\end{align*}
		and thus 
		\begin{align*}
		\big(\nabla_\mu \Jhat_h(\mu) - \nabla_{\mu} \cJhatn(\mu)\big)_i
		&=d_{\mu_i}\J(u_{h, \mu}, \mu) - d_{\mu_i}\J(u_{\red, \mu}, \mu)  + d_{\mu_i}r_\mu^\pr(u_{h,\mu})[p_{h,\mu}] - d_{\mu_i}r_\mu^\pr(u_{\red,\mu})[p_{\red,\mu}] \\
		& \qquad \qquad - d_{\mu_i} r_\mu^\pr(u_{\red,\mu})[w_{\red,\mu}] - d_{\mu_i} r_\mu^\du(u_{\red,\mu},p_{\red,\mu})[z_{\red,\mu}]
		\end{align*}
		The first line is equal to the estimator $\Delta_{\noncorgrad \Jnoncor_\red}(\mu)$,
		the first term of the second line can be estimated by 
		\begin{align*}
		d_{\mu_i} r_\mu^\pr(u_{\red,\mu})[w_{\red,\mu}] &\leq \cont{d_{\mu_i}l_{\mu}}\|w_{\mu}\| + \cont{d_{\mu_i}a_{\mu}} \| u_{\red,\mu} \| \|w_{\mu}\| 
		\end{align*}
		The second term can analogously be estimated by
		\begin{align*}
		d_{\mu_i} r_\mu^\du(u_{\red,\mu},p_{\red,\mu})[z_{\mu}] \leq \cont{d_{\mu_i}j} \|z_{\mu}\| + 2\cont{d_{\mu_i}k} \| z_{\mu} \| \|u_{\red,\mu}\|
		+ \cont{d_{\mu_i}a} \|z_{\mu}\| \|p_{\red,\mu}\|.
		\end{align*}
		We also have 
		\begin{align*}
		\underline{\bformd} \|w_{\mu}\|^2 &\leq \bformd(w_{\mu},w_{\mu}) = r_\mu^\du(u_{\red,\mu},p_{\red,\mu})[w_{\mu}]-2k_\mu(z_{\mu},w_{\mu}) 
		\leq \|r_\mu^\du(u_{\red,\mu},p_{\red,\mu})\| \|w_{\mu}\| + 2\cont{k_\mu} \|z_{\mu}\| \|w_{\mu}\|
		\end{align*}
		which gives
		\begin{equation*}
		\|w_{\mu}\| \leq \underline{\bformd}^{-1} \left( \|r_\mu^\du(u_{\red,\mu},p_{\red,\mu})\| + 2\cont{k_\mu} \|z_{\mu}\| \right).
		\end{equation*}
		For $z_\mu$ we estimate
		\begin{align*}
		\underline{\bformd} \|z_{\mu}\|^2 &\leq \bformd(z_{\mu},z_{\mu}) = -r_\mu^\pr(u_{\red,\mu})[z_\mu]\leq \|r_\mu^\pr(u_{\red,\mu})\| \|z_{\mu}\|.
		\end{align*}
		Summing all together gives the assertion.
\end{proof}

In a view of Section~\ref{sec:ncd_approach}, we emphasize that the estimator for the NCD-corrected gradient does not show a better approximation
of the FOM gradient since more terms are added to the standard estimate. Propositon \ref{prop:Jhat_error}.(ii) suggests that there exist an estimator of higher order which we derive in the following section.  
\subsubsection{Sensitivity based approximation and estimation} 
\label{sec:sens_estimation}
We elaborate a better estimator for the NCD-corrected gradient by using sensitivities of the reduced primal and dual
solutions. In addition, approximated sensitivities that are computed from the FOM sensitivities suggest an even better approximation of the FOM gradient.   

We define the derivatives of the primal and dual solution maps associated with \eqref{eq:optimality_conditionsRB} in direction $\nu \in \mathbb{R}^P$ as the solutions
$d_{\nu} u_{\red, \mu} \in V_\red^\pr$ and $d_{\nu} p_{\red, \mu} \in V_\red^\du$ of
\begin{align}
a_\mu(d_{\nu} u_{\red, \mu}, v_\red) &= \partial_{\mu} r_\mu^\pr(u_{\red, \mu})[v_\red]\cdot\nu &&\text{for all } v_\red \in V_\red^\pr \text{ and}\label{eq:true_primal_reduced_sensitivity}\\
a_\mu(q_\red, d_{\nu} p_{\red, \mu}) &= -\partial_\mu a_\mu(q_\red, p_{r, \mu})\cdot\nu + d_\mu\partial_u \J(u_{\red, \mu}, \mu)[q_\red]\cdot\nu\notag\\
&= \partial_\mu r_\mu^\du(u_{\red, \mu}, p_{\red, \mu})[q_\red]\cdot\nu + 2k_\mu(q_\red, d_{\nu} u_{\red, \mu}) &&\text{for all } q_\red \in V_\red^\du,
\label{eq:true_dual_reduced_sensitivity}
\end{align}
respectively, analogously to Propositions \ref{prop:solution_dmu_eta} and \ref{prop:dual_solution_dmu_eta},
where the last equality holds for quadratic functionals as in \eqref{P.argmin}. 
With these sensitivities we can compute the same gradient of the NCD-corrected RB reduced functional from Propostion~\ref{prop:true_corrected_reduced_gradient_adj} in a different manner.

\begin{proposition}[Gradient of the NCD-corrected RB reduced functional -- Sensitivity approach]
	\label{prop:true_corrected_reduced_gradient}
	The $i$-th component of the true gradient of $\cJhatn$, $\nabla_\mu\cJhatn:\Params \to \R^P$, is given by
	\begin{align*}
	\big(\nabla_\mu \cJhatn(\mu)\big)_i & = d_{\mu_i}\J(u_{\red, \mu}, \mu) + d_{\mu_i} r_\mu^\pr(u_{\red, \mu})[p_{\red, \mu}] 
	+ r_\mu^\pr(u_{\red, \mu})[d_{\mu_i}p_{\red, \mu}]
	+ r_\mu^\du(u_{\red, \mu},p_{\red, \mu})[d_{\mu_i} u_{\red, \mu}] 
	\end{align*}
	for all $1 \leq i \leq P$ and $\mu \in \Params$, where $u_{\red, \mu} \in V_\red^\pr$ and $p_{\red, \mu} \in V_\red^\du$ solve \eqref{eq:optimality_conditionsRB}, $d_{\mu_i}u_{\red, \mu} \in V_\red^\pr$ and $d_{\mu_i}p_{\red, \mu} \in V_\red^\du$ denote
	the derivatives of RB primal and dual solution maps from \eqref{eq:true_primal_reduced_sensitivity} and \eqref{eq:true_dual_reduced_sensitivity}. 
\end{proposition}
\begin{proof} 
	It follows from the chain rule and Remark~\ref{prop:gateaux_wrt_V}; see \cite[Section~1.6.1]{HPUU2009}.
\end{proof}

Note that the sensitivity based gradient is mathematically equivalent to the one in Propostion~\ref{prop:true_corrected_reduced_gradient_adj}, but the second only requires to solve \eqref{A2:z_eq} and \eqref{A2:w_eq} once,
because they can be reused for every component $\mu_i$, whereas the computation of the gradient in Propostion~\ref{prop:true_corrected_reduced_gradient} requires to solve \eqref{eq:true_primal_reduced_sensitivity}
and \eqref{eq:true_dual_reduced_sensitivity} for each $1\leq i\leq P$; cf.~\cite{HPUU2009}.

In terms of numerical approximation w.r.t.~the FOM functional, we note that, e.g., a solution $d_{\mu_i} u_{\red, \mu} \in V_\red^\pr$ of \eqref{eq:true_primal_reduced_sensitivity} does not necessarily need to be a good approximation of the FOM version $d_{\mu_i} u_{h, \mu} \in V_h$ even though $u_{h, \mu}$ is contained in $V_\red^\pr$ since the high dimensional sensitivities are not generally contained in the respective reduced space (c.f. Propostion~\ref{prop:primal_solution_dmui_error}).

To remedy this we could compute the FOM sensitivities for all canonical directions and either include them in the respective primal and dual space (thus forming Taylor RB spaces) or distribute all directional sensitivities to 
problem adapted RB spaces for the primal and dual sensitivities w.r.t.~all canonical directions: $V_\red^{\pr,d_{\mu_i}}, V_\red^{\du,d_{\mu_i}} \subset V_h$.
Thus, we again commit a variational crime.

\begin{definition}[Approximate partial derivatives of the RB primal and dual solution maps]
	\label{prop:red_sens}
	Considering the reduced primal and dual solution maps $\Params \to V_\red^\pr$, $\mu \mapsto u_{\red, \mu}$ and $\Params \to V_\red^\du$, $\mu \mapsto p_{\red, \mu}$,
	respectively, where $u_{\red, \mu}$ and $p_{\red, \mu}$ are the solutions of \eqref{eq:state_red} and \eqref{eq:dual_solution_red},
	we define their approximate partial derivatives w.r.t.~the $i$th component of $\mu$ by $\dred{\mu_i} u_{\red, \mu} \in V_\red^{\pr,d_{\mu_i}}$
	and $\dred{\mu_i} p_{\red, \mu} \in V_\red^{\du,d_{\mu_i}}$, respectively, as solutions of the sensitivity equations
	\begin{align}
	\bformd(\dred{\mu_i} u_{\red, \mu}, v_\red) &= \partial_\mu \resd^\pr(u_{\red, \mu})[v_\red] \cdot e_i&&\text{for all } v_\red \in V_\red^{\pr,d_{\mu_i}},
	\label{eq:red_sens_pr} \\
	\bformd(q_\red, \dred{\mu_i} p_{\red, \mu}) &= \partial_\mu \resd^\du(u_{\red, \mu}, p_{\red, \mu})[q_\red] \cdot e_i 
	+ 2\kformd(q_\red, \dred{\mu_i} u_{\red, \mu})&&\text{for all }q_\red \in V_\red^{\du,d_{\mu_i}}.
	\label{eq:red_sens_du}
	\end{align}
	Similarly, we denote the approximate partial derivatives in direction $\nu \in \mathbb{R}^P$ by $\dred{\nu} u_{\red, \mu}$ and $\dred{\nu} p_{\red, \mu}$,
	respectively, defined by substituting $e_i$ with $\nu$ above.
\end{definition}
Following Propositions \ref{prop:solution_dmu_eta} and \ref{prop:dual_solution_dmu_eta} we would obtain 
$\dred{\mu_i} u_{\red, \mu} = d_{\mu_i} u_{\red, \mu}$, if $V_\red^{\pr,d_{\mu_i}} =V_\red^{\pr}$ and 
$\dred{\mu_i} p_{\red, \mu} = d_{\mu_i} p_{\red, \mu}$, if $V_\red^{\du,d_{\mu_i}} =V_\red^{\du}$.
Moreover, the approximate partial derivatives depend on the choice of the corresponding reduced approximation spaces. 

\begin{definition}[Approximate gradient of the NCD-corrected RB reduced functional]
	\label{prop:red_approx_grad}
	We define the approximate gradient $\gradred{\mu}\cJhatn:\Params \to \R^P$ of $\cJhatn$ by
	\begin{align}
	\big(\gradred{\mu} \cJhatn(\mu)\big)_i :=  d_{\mu_i}\J(u_{\red,\mu}, \mu) &+ d_{\mu_i}r_\mu^\pr(u_{\red,\mu})[p_{\red,\mu}] 
	+ r_\mu^\pr(u_{\red, \mu})[\dred{\mu_i}p_{\red, \mu}] + r_\mu^\du(u_{\red, \mu},p_{\red, \mu})[\dred{\mu_i} u_{\red, \mu}], 
	\end{align}
	for $1 \leq i \leq P$, where $u_{\red,\mu} \in V_\red^\pr$, $p_{\red,\mu} \in V_\red^\du$ denote the reduced primal and dual solutions and 
	$\dred{\mu_i} u_{\red, \mu} \in V_\red^{\pr,d_{\mu_i}}$ and $\dred{\mu_i} p_{\red, \mu} \in V_\red^{\du,d_{\mu_i}}$ denote the solutions of \eqref{eq:red_sens_pr}
	and \eqref{eq:red_sens_du}. 
\end{definition}

Both gradients from Definition~\ref{prop:true_corrected_reduced_gradient} and Propostion~\ref{prop:red_approx_grad} yield higher order estimate.
To show this, we first derive  error estimates for the reduction error of the reduced sensitivities from \eqref{eq:true_primal_reduced_sensitivity} and
\eqref{eq:true_dual_reduced_sensitivity} as well as for \eqref{eq:red_sens_pr} and \eqref{eq:red_sens_du}. 
For $v_h \in V_h$, the residuals of the equation in Propostion~\ref{prop:solution_dmu_eta} and Propostion~\ref{prop:dual_solution_dmu_eta} for the canonical directions are respectively given by
\begin{align}
\resd^{\pr,d_{\mu_i}}(u_{h, \mu}, d_{\mu_i} u_{h, \mu})[v_h] &:= d_{\mu_i} \resd^\pr(u_{h, \mu})[v_h] - \bformd(d_{\mu_i} u_{h, \mu}, v_h),
\label{sens_res_pr}\\
\resd^{\du,d_{\mu_i}}(u_{h, \mu}, p_{h, \mu}, d_{\mu_i} u_{h, \mu}, d_{\mu_i} p_{h, \mu})[v_h] 
&:=  d_{\mu_i} \resd^\du(u_{h, \mu}, p_{h, \mu})[q_\red] + 2\kformd(q_\red, d_{\mu_i} u_{h, \mu}) - \bformd(q_h, d_{\mu_i} p_{h, \mu}) .
\label{sens_res_du}
\end{align}
\begin{proposition}[Residual based upper bound on the model reduction error of the sensitivity of the primal solution map]
	\label{prop:primal_solution_dmui_error}
	For $\mu \in \Params$ and $1 \leq i \leq P$, let $d_{\mu_i}u_{h, \mu} \in V_h$ be
	the solution of the discrete version of \eqref{eq:primal_sens} and $d_{\mu_i}u_{\red, \mu} \in V_\red^{\pr,d_{\mu_i}}$
	be the solution of \eqref{eq:true_primal_reduced_sensitivity}. We then have
	\begin{align}
	\|d_{\mu_i}u_{h, \mu} - d_{\mu_i}u_{\red, \mu}\| &\leq \Delta_{d_{\mu_i}\pr}(\mu) := \underline{\bformd}^{-1}\Big(\cont{d_{\mu_i} \bformd} \Delta_\pr(\mu) + \|\resd^{\pr,d_{\mu_i}}(u_{\red, \mu}, d_{\mu_i}u_{\red, \mu})\| \Big).
	\notag
	\end{align}
\end{proposition}
\begin{proof}
	Using the shorthand $d_{\mu_i} e_{h, \mu}^\pr := d_{\mu_i}u_{h, \mu} - d_{\mu_i}u_{\red, \mu}$, we obtain
	\begin{align*}
	\underline{\bformd}\, \|d_{\mu_i} e_{h, \mu}^\pr\|^2 &\leq \bformd(d_{\mu_i} e_{h, \mu}^\pr, d_{\mu_i} e_{h, \mu}^\pr) 
	= \bformd(d_{\mu_i}u_{h, \mu}, d_{\mu_i} e_{h, \mu}^\pr) -\, \bformd(d_{\mu_i}u_{\red, \mu}, d_{\mu_i} e_{h, \mu}^\pr) \\
	&= d_{\mu_i} \resd^\pr(u_{h, \mu})[d_{\mu_i} e_{h, \mu}^\pr] - \bformd(d_{\mu_i}u_{\red, \mu}, d_{\mu_i} e_{h, \mu}^\pr) \\
	&= d_{\mu_i} \resd^\pr(u_{h, \mu})[d_{\mu_i} e_{h, \mu}^\pr] - d_{\mu_i} \resd^\pr(u_{\red, \mu})[d_{\mu_i} e_{h, \mu}^\pr] + d_{\mu_i} \resd^\pr(u_{\red, \mu})[d_{\mu_i} e_{h, \mu}^\pr] - \bformd(d_{\mu_i}u_{\red, \mu}, d_{\mu_i} e_{h, \mu}^\pr) \\	
	&= - d_{\mu_i} \bformd(e_{h, \mu}^\pr, d_{\mu_i} e_{h, \mu}^\pr) + \resd^{\pr,d_{\mu_i}}(u_{\red, \mu}, d_{\mu_i}u_{\red, \mu})[d_{\mu_i} e_{\red, \mu}^\pr] \\
	&\leq \cont{d_{\mu_i} \bformd} \;\|e_{h, \mu}^\pr\| \;\|d_{\mu_i} e_{h, \mu}^\pr\|  + \|\resd^{\pr,d_{\mu_i}}(u_{\red, \mu}, d_{\mu_i}u_{\red, \mu})\|\; \|d_{\mu_i} e_{h, \mu}^\pr\|
	\end{align*}
	using the coercivity of $\bformd$ in the first inequality, the definition $d_{\mu_i} e_{h, \mu}^\pr$ in the first equality,
	Propostion~\ref{prop:solution_dmu_eta} applied to $u_{h, \mu}$ in the second equality,
	the definition of the discrete sensitivity primal residual \eqref{sens_res_pr} in the third equality and the continuity of $d_{\mu_i} \bformd$ in the last inequality.
\end{proof}
We emphasize that the same result can be shown for $\dred{\mu_i} u_{\red, \mu}$ by replacing $d_{\mu_i} u_{\red, \mu}$ and using the equation \eqref{eq:red_sens_pr}
instead of \eqref{eq:true_primal_reduced_sensitivity}. We call the resulting error estimator $\Delta_{\dred{\mu_i}\pr}(\mu)$.

\begin{proposition}[Residual based upper bound on the model reduction error of the sensitivity of the dual solution map]
	\label{prop:dual_solution_dmui_error}
	For $\mu \in \Params$ and $1 \leq i \leq P$, let $d_{\mu_i}p_{h, \mu} \in V_h$ be
	the solution of the discrete version of \eqref{eq:dual_sens} and $d_{\mu_i}p_{\red, \mu} \in V_\red^{\pr,d_{\mu_i}}$
	be the solution of \eqref{eq:true_dual_reduced_sensitivity}. We then obtain
	\begin{align*}
	\|d_{\mu_i}p_{h, \mu} - d_{\mu_i}p_{\red, \mu}\| \leq \Delta_{ d_{\mu_i}\du}(\mu)\hspace*{5cm}\text{with}\\
	\Delta_{ d_{\mu_i}\du}(\mu) := \underline{\bformd}^{-1}\Big(
	2 \cont{d_{\mu_i} \kformd} \; \Delta_\pr(\mu) +  \cont{d_{\mu_i} \bformd} \; \Delta_\du(\mu) + 2 \cont{\kformd}  \; \Delta_{d_{\mu_i}pr}(\mu) 
	+ \| \resd^{\du,d_{\mu_i}}(u_{\red, \mu}, p_{\red, \mu}, d_{\mu_i}u_{\red, \mu}, d_{\mu_i}p_{\red, \mu}) \| \Big).
	\end{align*}
\end{proposition}
\begin{proof}
	Using the shorthand $d_{\mu_i} e_{h, \mu}^\du := d_{\mu_i}p_{h, \mu} - d_{\mu_i}p_{\red, \mu}$ and $e_{h, \mu}^\du := p_{h, \mu} - p_{\red, \mu}$, we obtain
	\begin{align*}
	\underline{\bformd}\;&\|d_{\mu_i} e_{h, \mu}^\du\|^2 \leq \bformd(d_{\mu_i} e_{h, \mu}^\du, d_{\mu_i} e_{h, \mu}^\du)
	= \underbrace{\bformd(d_{\mu_i} e_{h, \mu}^\du, d_{\mu_i}p_{h, \mu})}_{= d_{\mu_i} \resd^\du(u_{h, \mu}, p_{h, \mu})[d_{\mu_i} e_{h, \mu}^\du] 
		+ 2\kformd(d_{\mu_i} e_{h, \mu}^\du, d_{\mu_i}u_{h, \mu})} - \bformd(d_{\mu_i} e_{h, \mu}^\du, d_{\mu_i}p_{\red, \mu})
	\\
	&= d_{\mu_i} \resd^\du(u_{h, \mu}, p_{h, \mu})[d_{\mu_i} e_{h, \mu}^\du] + 2\kformd(d_{\mu_i} e_{h, \mu}^\du, d_{\mu_i}u_{h, \mu})
	- d_{\mu_i} \resd^\du(u_{\red, \mu}, p_{\red, \mu})[d_{\mu_i} e_{h, \mu}^\du] \\
	&\qquad - 2\kformd(d_{\mu_i} e_{h, \mu}^\du, d_{\mu_i}u_{\red, \mu})
	+ d_{\mu_i} \resd^\du(u_{\red, \mu}, p_{\red, \mu})[d_{\mu_i} e_{h, \mu}^\du] + 2\kformd(d_{\mu_i} e_{h, \mu}^\du, d_{\mu_i}u_{\red, \mu}) 
	 - \bformd(d_{\mu_i} e_{h, \mu}^\du, d_{\mu_i}p_{\red, \mu}) \\
	&= d_{\mu_i} \jformd(d_{\mu_i} e_{h, \mu}^\du) + 2 d_{\mu_i}\kformd(d_{\mu_i} e_{h, \mu}^\du, u_{h, \mu}) - d_{\mu_i}\bformd(d_{\mu_i} e_{h, \mu}^\du, p_h)
	 - d_{\mu_i}\jformd(d_{\mu_i} e_{h, \mu}^\du) + 2 d_{\mu_i}\kformd(d_{\mu_i} e_{h, \mu}^\du, u_{\red, \mu})\\
	&\qquad 
	- d_{\mu_i}\bformd(d_{\mu_i} e_{h, \mu}^\du, p_\red) +\! 2\kformd(d_{\mu_i} e_{h, \mu}^\du, d_{\mu_i}u_{h, \mu}) \!-\! 2\kformd(d_{\mu_i} e_{h, \mu}^\du, d_{\mu_i}u_{\red, \mu}) \\
	&\qquad + \resd^{\du,d_{\mu_i}}(u_{\red, \mu}, p_{\red, \mu}, d_{\mu_i}u_{\red, \mu}, d_{\mu_i}p_{\red, \mu})[d_{\mu_i} e_{h, \mu}^\du]	\\
	&= 2 d_{\mu_i}\kformd(d_{\mu_i} e_{h, \mu}^\du, e_{h, \mu}^\pr) - d_{\mu_i}\bformd(d_{\mu_i} e_{h, \mu}^\du, e_{h, \mu}^\du)
	 + 2\kformd(d_{\mu_i} e_{h, \mu}^\du, d_{\mu_i} e_{h, \mu}^\pr) \\
	&\qquad
	+ \resd^{\du,d_{\mu_i}}(u_{\red, \mu}, p_{\red, \mu}, d_{\mu_i}u_{\red, \mu}, d_{\mu_i}p_{\red, \mu})[d_{\mu_i} e_{h, \mu}^\du]\\
	&\leq \big(2 \cont{d_{\mu_i} \kformd} \; \|e_{h, \mu}^\pr\| +  \cont{d_{\mu_i} \bformd} \; \|e_{h, \mu}^\du\|  \big)\|d_{\mu_i} e_{h, \mu}^\du\| \\
	&\qquad + 2 \cont{\kformd}  \; \|d_{\mu_i} e_{h, \mu}^\pr\| \; \|d_{\mu_i} e_{h, \mu}^\du\|
	+ \| \resd^{\du,d_{\mu_i}}(u_{\red, \mu}, p_{\red, \mu}, d_{\mu_i}u_{\red, \mu}, d_{\mu_i}p_{\red, \mu}) \| \; \|d_{\mu_i} e_{h, \mu}^\du\|
	\end{align*}
	using the coercivity of $\bformd$ in the first inequality, the definition of $d_{\mu_i} e_{h, \mu}^\du$ in the first equality, Propostion~\ref{prop:dual_solution_dmu_eta} applied to $p_{h, \mu}$ in the second equality, the definition of the dual residual in \eqref{eq:dual_residual} in the third equality and continuity of all parts in the last inequality.
\end{proof}
Again, the same result holds for $\dred{\mu_i} p_{\red, \mu}$ if we replace $d_{\mu_i} p_{\red, \mu}$ and use \eqref{eq:red_sens_du} instead of
\eqref{eq:true_dual_reduced_sensitivity}. The resulting error estimator is then called $\Delta_{\dred{\mu_i}\du}(\mu)$.

Using the residual based a posteriori error estimates for the primal sensitivities, we are able to state two a posteriori error bounds on the model reduction error of
the true gradient and the approximated gradient of the NCD-corrected functional.  

\begin{proposition}[Upper bound on the model reduction error of the gradient of the reduced output -- sensitivity approach] 
	\label{prop:grad_Jhat_error_sens}
	\hfill
	\begin{enumerate}[(i)]
		\item [\emph{(i)}] For the gradient $\nabla_{\mu} \cJhatn(\mu)$ of the NCD-corrected RB reduced functional, computed with sensitivities
		according to Proposition~{\rm{\ref{prop:true_corrected_reduced_gradient}}}, we have
		\begin{align*}
      \big\|\nabla_\mu \Jhat_h(\mu) - \nabla_{\mu} \cJhatn(\mu)\big\|_2 &\leq \Delta_{\nabla\Jhat_\red}(\mu) 
    = \big\|\underline{\Delta_{\nabla\Jhat_\red}(\mu)}\big\|_2 \quad\quad\text{with} \\
      \big(\underline{\Delta_{\nabla\Jhat_\red}(\mu)}\big)_i := \cont{d_{\mu_i} \kformd} \, \left(\Delta_\pr(\mu)\right)^2 &+ \cont{\bformd} \, 
		\Delta_{d_{\mu_i}pr}(\mu) \,\Delta_\du(\mu)  + \| \resd^{\du,d_{\mu_i}}(u_{\red, \mu}, p_{\red, \mu}, d_{\mu_i}u_{\red, \mu}, d_{\mu_i}p_{\red, \mu}) \| \,
		\Delta_\pr(\mu).
		\end{align*}
		\item [\emph{(ii)}] Furthermore, we have for the approximate gradient from Definition~{\rm{\ref{prop:red_approx_grad}}}
		\begin{align*}
		\big\|\nabla_\mu \Jhat_h(\mu) - \gradred{\mu} \cJhatn(\mu)\big\|_2 &\leq \Delta_{\widetilde\nabla\Jhat_\red}(\mu) 
    = \big\|\underline{\Delta_{\widetilde\nabla\Jhat_\red}(\mu)}\big\|_2 \quad\quad\text{with} \\
    \big(\underline{\Delta_{\widetilde\nabla\Jhat_\red}(\mu)}\big)_i := \cont{d_{\mu_i} \kformd} \, \left(\Delta_\pr(\mu)\right)^2 &+ \cont{\bformd} \, 
		\Delta_{\dred{\mu_i}\pr}(\mu) \,\Delta_\du(\mu) + \| \resd^{\du,d_{\mu_i}}(u_{\red, \mu}, &p_{\red, \mu}, \dred{\mu_i}u_{\red, \mu}, \dred{\mu_i}p_{\red, \mu}) \| \,
		\Delta_\pr(\mu).
		\end{align*}
	\end{enumerate}
\end{proposition}
\begin{proof}
(i) To prove the first assertion, we use $r_\mu^\pr(u_{h, \mu})[d_{\mu_i}p_{\red, \mu}] = 0$ and $r_\mu^\du(u_{h, \mu},p_{h, \mu})[d_{\mu_i} u_{\red, \mu}]=0$ to obtain
		\begin{align*}
		\big(\nabla_\mu \Jhat_h(\mu) &- \nabla_\mu \cJhatn(\mu)\big)_i 
		=d_{\mu_i}\J(u_{h, \mu}, \mu) - d_{\mu_i}\J(u_{\red, \mu}, \mu) + d_{\mu_i}r_\mu^\pr(u_{h,\mu})[p_{h,\mu}] - d_{\mu_i}r_\mu^\pr(u_{\red,\mu})[p_{\red,\mu}]  \\
		&= \partial_\mu \jformd(e_{h, \mu}^\pr) + \partial_\mu \kformd(u_{h, \mu}, u_{h, \mu}) - \partial_\mu \kformd(u_{\red, \mu},u_{\red, \mu})  
		 + d_{\mu_i}r_\mu^\pr(u_{h,\mu})[p_{h,\mu}] - d_{\mu_i}r_\mu^\pr(u_{\red,\mu})[p_{\red,\mu}] \\
		&\qquad + \underset{=(*)}{\underbrace{r_\mu^\pr(e_{h, \mu}^\pr)[d_{\mu_i}p_{\red, \mu}]}} 
		+ \underset{=(**)}{\underbrace{r_\mu^\du(e_{h, \mu}^\pr,e_{h, \mu}^\du)[d_{\mu_i} u_{\red, \mu}]}}.
		\end{align*}
		For the last two residual terms we have
		\begin{align*}
		(*) &= \lformd(d_{\mu_i}p_{\red, \mu}) - \lformd(d_{\mu_i}p_{\red, \mu}) 
		- \bformd(e_{h, \mu}^\pr,d_{\mu_i}p_{\red, \mu}) \\ 
		&= - \bformd(e_{h, \mu}^\pr,d_{\mu_i}p_{\red, \mu}) 
		+ d_{\mu_i} r_\mu^\du(u_{\red, \mu},p_{\red, \mu})[e_{h, \mu}^\pr] 
		+ 2\kformd(d_{\mu_i}u_{\red, \mu}, e_{h, \mu}^\pr) 
		- d_{\mu_i} r_\mu^\du(u_{\red, \mu},p_{\red, \mu})[e_{h, \mu}^\pr] 
		- 2\kformd(d_{\mu_i}u_{\red, \mu}, e_{h, \mu}^\pr) \\
		&= \resd^{\du,d_{\mu_i}}(u_{\red, \mu}, p_{\red, \mu}, d_{\mu_i}u_{\red, \mu}, d_{\mu_i}p_{\red, \mu})[e_{h, \mu}^\pr] 
		- d_{\mu_i} r_\mu^\du(u_{\red, \mu},p_{\red, \mu})[e_{h, \mu}^\pr] 
		- 2\kformd(d_{\mu_i}u_{\red, \mu}, e_{h, \mu}^\pr) 
		\end{align*}
		and
		\begin{align*}
		(**) &= \jformd(d_{\mu_i}u_{\red, \mu}) - \jformd(d_{\mu_i}u_{\red, \mu}) 
		+ 2\kformd(d_{\mu_i}u_{\red, \mu}, e_{h, \mu}^\pr) - \bformd(d_{\mu_i}u_{\red, \mu}, e_{h, \mu}^\du). 
		\end{align*}
		Thus, by summing both terms we have
		\begin{align*}
		(*) + (**)= 
		\resd^{\du,d_{\mu_i}}(u_{\red, \mu}, p_{\red, \mu},d_{\mu_i}u_{\red, \mu}, d_{\mu_i}p_{\red, \mu})[e_{h, \mu}^\pr] 
		- \underset{=(***)}{\underbrace{d_{\mu_i} r_\mu^\du(u_{\red, \mu},p_{\red, \mu})[e_{h, \mu}^\pr]}} - \bformd(d_{\mu_i}u_{\red, \mu}, e_{h, \mu}^\du) 
		\end{align*}
		and for $(***)$ it holds
		\begin{align*}
		(***) = d_{\mu_i}\jformd(e_{h, \mu}^\pr) + 2 d_{\mu_i}\kformd(e_{h, \mu}^\pr, u_{\red,\mu}) - d_{\mu_i}\bformd(e_{h, \mu}^\pr,p_{\red,\mu}).
		\end{align*}
		Combining $(*)$, $(**)$ and $(***)$ with the previous result, we have
		\begin{align*}
		\big(\nabla_\mu \Jhat_h(\mu) &- \nabla_\mu \cJhatn(\mu)\big)_i 
		 = \partial_\mu \jformd(e_{h, \mu}^\pr) + \partial_\mu \kformd(u_{h, \mu}, u_{h, \mu}) - \partial_\mu \kformd(u_{\red, \mu},u_{\red, \mu})  \\
		&\qquad - d_{\mu_i}\jformd(e_{h, \mu}^\pr) - 2 d_{\mu_i}\kformd(e_{h, \mu}^\pr, u_{\red,\mu}) + d_{\mu_i}\bformd(e_{h, \mu}^\pr,p_{\red,\mu}) + d_{\mu_i}r_\mu^\pr(u_{h,\mu})[p_{h,\mu}]  \\
		&\qquad - d_{\mu_i}r_\mu^\pr(u_{\red,\mu})[p_{\red,\mu}] 
		 +  \resd^{\du,d_{\mu_i}}(u_{\red, \mu}, p_{\red, \mu},d_{\mu_i}u_{\red, \mu}, d_{\mu_i}p_{\red, \mu})[e_{h, \mu}^\pr] 
		- \bformd(d_{\mu_i}u_{\red, \mu}, e_{h, \mu}^\du)  \\
		& = d_{\mu_i} \kformd(e_{h, \mu}^\pr, e_{h, \mu}^\pr) 
		+  \resd^{\du,d_{\mu_i}}(u_{\red, \mu}, p_{\red, \mu},d_{\mu_i}u_{\red, \mu}, d_{\mu_i}p_{\red, \mu})[e_{h, \mu}^\pr] \\
		& \qquad + \underset{=(****)}{\underbrace{d_{\mu_i}r_\mu^\pr(u_{h,\mu})[p_{h,\mu}] - d_{\mu_i}r_\mu^\pr(u_{\red,\mu})[p_{\red,\mu}] 
				+  d_{\mu_i}\bformd(e_{h, \mu}^\pr,p_{\red,\mu}) -  \bformd(d_{\mu_i}u_{\red, \mu}, e_{h, \mu}^\du)}}.
		\end{align*}
		Further, we have
		\begin{align*}
		d_{\mu_i}r_\mu^\pr(u_{h,\mu})[p_{h,\mu}] &- d_{\mu_i}r_\mu^\pr(u_{\red,\mu})[p_{\red,\mu}] = d_{\mu_i} \lformd(e_{h, \mu}^\du) - 
		d_{\mu_i} \bformd(u_{h, \mu}, p_{h, \mu}) + d_{\mu_i} \bformd(u_{\red, \mu}, p_{\red, \mu}) \\
		& =  \bformd(d_{\mu_i}u_h, e_{h, \mu}^\du) + d_{\mu_i} \bformd(u_h, e_{h, \mu}^\du) 
		 - d_{\mu_i} \bformd(u_{h, \mu}, p_{h, \mu}) + d_{\mu_i} \bformd(u_{\red, \mu}, p_{\red, \mu}),
		\end{align*}
		where we used the discretized version of \eqref{eq:primal_sens} in the second equality. Inserting this into $(*)$ gives
		\begin{align*}
		(***\,*) =&  \bformd(d_{\mu_i}u_h, e_{h, \mu}^\du) -  \bformd(d_{\mu_i}u_{\red, \mu}, e_{h, \mu}^\du) \\ & +  
		\underset{=0}{\underbrace{d_{\mu_i} \bformd(u_h, e_{h, \mu}^\du) - d_{\mu_i} \bformd(u_{h, \mu}, p_{h, \mu}) 
				+ d_{\mu_i} \bformd(u_{\red, \mu}, p_{\red, \mu}) 
				+d_{\mu_i}\bformd(e_{h, \mu}^\pr,p_{\red,\mu}) }} 
		= \bformd(d_{\mu_i} e_{h, \mu}^\pr, e_{h, \mu}^\du).
		\end{align*}
		In total, we have 
		\begin{align*}
		& &\big(\nabla_\mu \Jhat_h(\mu) - \nabla_\mu \cJhatn(\mu)\big)_i 
		 = d_{\mu_i} \kformd(e_{h, \mu}^\pr, e_{h, \mu}^\pr) 
		+ \bformd(d_{\mu_i} e_{h, \mu}^\pr, e_{h, \mu}^\du) +  \resd^{\du,d_{\mu_i}}(u_{\red, \mu}, p_{\red, \mu},d_{\mu_i}u_{\red, \mu}, d_{\mu_i}p_{\red, \mu})[e_{h, \mu}^\pr] 
		\end{align*}
		which proofs the assertion. \\
(ii) The estimate follows analogously to (i), by replacing $d_{\mu_i} u_{\red, \mu}$ and $d_{\mu_i} p_{\red, \mu}$ with $\dred{\mu_i} u_{\red, \mu}$
		and $\dred{\mu_i} p_{\red, \mu}$, respectively.
\end{proof}

To conclude, $\Delta_{\widetilde\nabla\Jhat_\red}(\mu)$ and $\Delta_{\nabla\Jhat_\red}(\mu)$ both decay with second order (cf.~Section~\ref{sec:estimator_study}).
We also point out, that $\Delta_{\widetilde\nabla\Jhat_\red}(\mu)$ is an improved estimator which can be used to replace the poor estimator
$\Delta^{\red,*}_{\nabla_\mu \Jhat_\red}(\mu)$.
However both higher order estimators $\Delta_{\widetilde\nabla\Jhat_\red}(\mu)$ and $\Delta_{\nabla\Jhat_\red}(\mu)$
come with the price of computing the dual norm of the sensitivity residuals in \eqref{sens_res_pr} and \eqref{sens_res_du} for each direction
which aggravates the computational complexity. 

\section{The Trust-Region Method and adaptive enrichment strategies}
\label{sec:TRRB_and_adaptiveenrichment}
To solve problem \eqref{P} we apply the TR method, which iteratively computes a first-order critical point of \eqref{P}. At each iteration $k\geq 0$, we consider a so-called model function $m^{(k)}$, which is a cheaply computable approximation of the quadratic cost functional $\J$ in a neighbourhood of the parameter $\mu^{(k)}$, i.e., the Trust-Region. Therefore, for $k\geq 0$, given a TR radius $\delta^{(k)}$, we consider the TR minimization sub-problem
\begin{equation}
\label{TRsubprob}
\min_{s\in \mathbb{R}^P} m^{(k)}(s) \, \text{ subject to } \|s\|_2 \leq \delta^{(k)},\, \widetilde{\mu}:= \mu^{(k)}+s \in\Params \text{ and } r_{\tilde{\mu}}^\pr(u_{\tilde{\mu}})[v]= 0 \, \text{ for all }  v\in V.
\end{equation}
Under suitable assumptions on $m^{(k)}$, problem \eqref{TRsubprob} admits a unique solution $\bar s^{(k)}$, which is used to compute the next iterate $\mu^{(k+1)} = \mu^{(k)} + \bar s^{(k)}$.
\subsection{The Trust-Region Reduced Basis Method}
\label{sec:TR}
Slightly different from \cite{AFS00,QGVW2017}, we choose as model function the NCD-corrected RB reduced functional $\cJhatn^{(k)}$ defined in (\ref{eq:Jhat_red_corected}), i.e.~$m^{(k)}(\cdot)= \cJhatn^{(k)}(\mu^{(k)}+\cdot)$ for $k\geq 0$, where the super-index $(k)$ indicates that we use different RB spaces $V_\red^{*, (k)}$ in each iteration. As indicated in Proposition~\ref{prop:Jhat_error} and shown in our numerical experiments below, $\cJhatn^{(k)}$ converges to $\Jhat$ with higher order 
in comparison to the standard RB reduced functional from \eqref{eq:Jhat_red}, which has been considered in \cite{QGVW2017}.
We initialize the RB spaces using the initial guess $\mu^{(0)}$, i.e.~setting $V^{\pr,0}_\red = \left\{u_{h,\mu^{(0)}}\right\}$ and $V^{\du,0}_\red = \left\{p_{h,\mu^{(0)}}\right\}$. At every iteration $k$ we may -- depending on the a posteriori estimates -- enrich the obtained space using the computed parameter $\mu^{(k+1)}$; for further details see Section~\ref{sec:construct_RB}. Possible sufficient and necessary conditions for convergence, dependent on the approximate generalized Cauchy point (AGC) $\mu^{(k)}_\text{\rm{AGC}}$ (see Definition~\ref{Def:AGC}), are given in \cite{QGVW2017}. In contrast to \cite{QGVW2017}, we consider additional bilateral parameter constraints in \eqref{TRsubprob}. In particular, the presence of these inequality constraints requires a review of the proof of convergence for the TR-RB algorithm. In \cite{QGVW2017}, the convergence is based on the results contained in \cite{YM2013}, where the authors consider an equality-constrained optimization problem. We state first how our method differs from the one in \cite{QGVW2017}, then we prove the convergence of this modified algorithm. According to \cite{QGVW2017}, the inexact RB version of problem \eqref{TRsubprob} is 
\begin{equation}
\label{TRRBsubprob}
\min_{\widetilde{\mu}\in\Params} \cJhatn^{(k)}(\widetilde{\mu}) \text{ s.t. }  \frac{\Delta_{\cJhatn}(\widetilde{\mu})}{\cJhatn^{(k)}(\widetilde{\mu})}\leq \delta^{(k)},
\end{equation}
where $\widetilde{\mu}:= \mu^{(k)}+s$, the equality constraint $r_{\tilde\mu}^\pr(u_{\tilde{\mu}})[v]= 0$ is hidden in the definition of $\cJhatn$ and the inequality ones are concealed in the request $\widetilde{\mu}\in \Params$. As also remarked in \cite{QGVW2017}, the projected BFGS method \cite{Kel99}, which we use in order to solve \eqref{TRRBsubprob}, computes the ACG point $\mu^{(k)}_\text{\rm{AGC}}$ in the first iterate and generates a sequence $\{\mu^{(k,\ell)}\}_{\ell=1}^L$ where $L$ is the last BFGS iteration. In what follows, $\mu^{(k,1)}:= \mu^{(k)}_\text{\rm{AGC}}$ and the TR iterate $\mu^{(k+1)}:=\mu^{(k,L)}$. Throughout the paper the index $k$ refers to the current outer TR iteration, $\ell$ refers instead to the inner BFGS iteration. Note that $L$ may be different for each iteration $k$, but we will indicate it only when strictly necessary in order to simplify the notation. To describe the projected BFGS method in details, we define
\begin{align}
\label{eq:General_Opt_Step_point}
\mu^{(k,\ell)}(j):= \Proj_\mathcal{P}(\mu^{(k,\ell)} + \kappa^j d^{(k,\ell)}) \in\Params && \text{for } j\geq 0,
\end{align}
where $\kappa\in(0,1)$, $d^{(k,\ell)}\in\mathbb{R}^P$ is the chosen descent direction at the iteration $(k,\ell)$ and the projection operator $\Proj_\Params: \mathbb{R}^P\rightarrow \Params$ is defined as
\begin{align*}
(\Proj_\Params(\mu))_i:= \left\{ \begin{array}{ll}
(\mu_\mathsf{a})_i & \text{if } (\mu)_i\leq (\mu_\mathsf{a})_i, \\
(\mu_\mathsf{b})_i & \text{if } (\mu)_i\geq (\mu_\mathsf{b})_i, \\
(\mu)_i & \text{otherwise,}
\end{array} \right. 
&& \text{for } i=1,\ldots,P.
\end{align*}
Note that the operator $\Proj_\Params$ is Lipschitz continuous with constant one; cf.~\cite{Kel99}. For computing the descent direction $d^{(k,\ell)}$ we follow the projected BFGS algorithm reported in \cite[Section~5.5.3]{Kel99}. Furthermore, we enforce respectively an Armijo-type condition and the additional TR constraint on $\cJhatn^{(k)}$
\begin{align}
\label{Armijo}\cJhatn^{(k)}(\mu^{(k,\ell)}(j)) - \cJhatn^{(k)}(\mu^{(k,\ell)}) & \leq  -\frac{\kappa_{\mathsf{arm}}}{\kappa^j} \big\| \mu^{(k,\ell)}(j)-\mu^{(k,\ell)}\big\|^2_2, \\
\label{TR_radius_condition} \frac{\Delta_{\cJhatn}(\mu^{(k,\ell)}(j))}{\cJhatn^{(k)}(\mu^{(k,\ell)}(j))} & \leq \delta^{(k)},
\end{align}
by selecting $\mu^{(k,\ell+1)} = \mu^{(k,\ell)}(j^{(k,\ell)})$ for $\ell\geq 1$, where $j^{(k,\ell)}<\infty$ is the smallest index for which the conditions \eqref{Armijo}-\eqref{TR_radius_condition} hold for some $\kappa_{\mathsf{arm}}\in(0,\frac{1}{2})$, generally $\kappa_{\mathsf{arm}}=10^{-4}$; cf.~\cite{QGVW2017}. Moreover, we use as termination criteria for the optimization sub-problem\newline
\begin{subequations}\label{Termination_crit_subproblem}
	\noindent\begin{minipage}{0.6\textwidth}
		\begin{equation}
		\label{FOC_subproblem}
		\big\|\mu^{(k,\ell)}-\Proj_\Params(\mu^{(k,\ell)}-\nabla_\mu \Jhat_\red^{(k)}(\mu^{(k,\ell)}))\big\|_2\leq \tau_\text{\rm{sub}}
		\hspace{2em} \text{or}
		\end{equation}
	\end{minipage}%
	\begin{minipage}{0.4\textwidth}
		\vspace{0.2em}
		\begin{equation}
		\label{Cut_of_TR}
		\beta_2\delta^{(k)} \leq \frac{\Delta_{\Jhat_\red^{(k)}}(\mu)}{\Jhat_\red^{(k)}(\mu)} \leq \delta^{(k)},
		\end{equation}
	\end{minipage}\vskip1em
\end{subequations}
\noindent where $\tau_\text{\rm{sub}}\in(0,1)$ is a predefined tolerance and $\beta_2\in(0,1)$, generally close to one. Condition \eqref{Cut_of_TR} is used to prevent that the optimizer spends much time close to the boundary of the Trust-Region, where the model is poor in approximation; cf.~\cite{QGVW2017}. Note that, without the projection operator $\Proj_\Params$, conditions \eqref{Armijo}-\eqref{Termination_crit_subproblem} coincide with the ones in \cite{QGVW2017}, apart from using the NCD-corrected RB reduced functional. Furthermore, in addition to \cite{QGVW2017}, we consider a condition which allows enlarging the TR radius. A drawback of the TR algorithm proposed in \cite{QGVW2017} is that the TR radius may be significantly shrunk at the beginning, i.e.~when the TR model is poor in approximation. Afterwards, even if the RB space is enriched, i.e.~the approximation of the TR model function is improved, the TR radius is kept small. Thus, one misses the local second-order rate of convergence of the BFGS method. More precisely, if $\mu^{(k,\ell)}$ is close to the locally optimal solution $\bar \mu^{(k)}$ of the TR sub-problem, we want to make full BFGS steps, which gives us faster convergence. The possibility to enlarge the TR radius at each iteration will also decrease the number of outer iterations needed to converge. As a condition for enlarging the radius we check whether the sufficient reduction predicted by the model function $\cJhatn^{(k)}$ is realized by the objective function, i.e.~we check if
\begin{equation}
\label{TR_act_decrease}
\varrho^{(k)}:= \frac{\Jhat_h(\mu^{(k)})-\Jhat_h(\mu^{(k+1)})}{\cJhatn^{(k)}(\mu^{(k)})-\cJhatn^{(k)}(\mu^{(k+1)})}  \geq  \eta_\varrho
\end{equation}
for a tolerance $\eta_\varrho \in [3/4,1)$. Condition \eqref{TR_act_decrease} seems costly because of the evaluation of the FOM cost functional $\Jhat_h$, but, after the
enrichment of the RB space, the quantities in the numerator of \eqref{TR_act_decrease} are cheaply accessible, since one has already
solved the FOM to generate the new snapshots for the RB space enrichment. Note that this also implies that we can cheaply evaluate the FOM gradient $\nabla_\mu \Jhat_h(\mu^{k+1})$ in case of an enrichment. This knowledge will be used for the stopping criterium in the outer loop of the algorithm. Finally, let us define the AGC point for our constrained case.
\begin{definition}[AGC point for simple bounds]
	\label{Def:AGC}
	At iteration $k$, we define the AGC point as
	\[
	\mu_\text{\rm{AGC}}^{(k)}:= \mu^{(k,0)}(j^{(k)}_c)=  \Proj_\mathcal{P}(\mu^{(k,0)} + \kappa^{j^{(k)}_c} d^{(k,0)}),
	\]
	where $\mu^{(k,0)}:= \mu^{(k)}$, $d^{(k,0)}:= -\nabla_\mu \cJhatn^{(k)}(\mu^{(k,0)})$ and $j^{(k)}_c$ is the smallest non-negative integer $j$ for which $\mu^{(k,0)}(j)$ satisfies \eqref{Armijo}-\eqref{TR_radius_condition} for $\ell=0$.
\end{definition}
We refer to Algorithm~\ref{Alg:TR-RBmethod} for the proposed TR-RB algorithm.
\begin{algorithm2e}[!h]
	\KwData{Initial TR radius $\delta^{(0)}$, TR shrinking factor $\beta_1\in (0,1)$, tolerance for enlarging the TR radius $\eta_\varrho\in [\frac{3}{4},1)$, initial parameter $\mu^{(0)}$, stopping tolerance for the sub-problem $\tau_\text{\rm{sub}}\ll 1$, stopping tolerance for the first-order critical condition $\tau_\text{\rm{FOC}}$ with $\tau_\text{\rm{sub}}\leq\tau_\text{\rm{FOC}}\ll 1$, safeguard for TR boundary $\beta_2\in(0,1)$.}
	
	Set $k=0$ and \textsf{Loop\_flag}$=$\textsf{True}\; 
	\While{ 
		{\normalfont\textsf{Loop\_flag}}}{
		 Compute $\mu^{(k+1)}$ as solution of \eqref{TRRBsubprob} with termination criteria \eqref{Termination_crit_subproblem}\label{TRRB-optstep}\;
		\uIf{\label{Suff_condition_TRRB}$\cJhatn^{(k)}(\mu^{(k+1)})+\Delta_{\Jhat_\red^{(k)}}(\mu^{(k+1)})<\cJhatn^{(k)}(\mu^{(k)}_\text{\rm{AGC}})$} {
			Accept $\mu^{(k+1)}$, 
				update the RB model at $\mu^{(k+1)}$ and compute $\varrho^{(k)}$ from \eqref{TR_act_decrease}\; 
				\eIf {$\varrho^{(k)}\geq\eta_\varrho$}  {
					Enlarge the TR radius $\delta^{(k+1)} = \beta_1^{-1}\delta^{(k)}$\;	
				}
			    {
					Set $\delta^{(k+1)}=\delta^{(k)}$\;
				}
				
		}
		\uElseIf {\label{Nec_condition_TRRB}$\cJhatn^{(k)}(\mu^{(k+1)})-\Delta_{\Jhat_\red^{(k)}}(\mu^{(k+1)})> \cJhatn^{(k)}(\mu^{(k)}_\text{\rm{AGC}})$} {
			Reject $\mu^{(k+1)}$, shrink the TR radius $\delta^{(k+1)} = \beta_1\delta^{(k)}$ and go to \ref{TRRB-optstep}\;
		}
		\Else {
				Update the RB model at $\mu^{(k+1)}$ and compute $\varrho^{(k)}$ from \eqref{TR_act_decrease}\; 
			\eIf {$\cJhatn^{(k+1)}(\mu^{(k+1)})\leq \cJhatn^{(k)}(\mu^{(k)}_\text{\rm{AGC}})$} {
				Accept $\mu^{(k+1)}$\;	
				\eIf {$\varrho^{(k)}\geq\eta_\varrho$} {
					Enlarge the TR radius $\delta^{(k+1)} = \beta_1^{-1}\delta^{(k)}$\;	
				}
				{
				Set $\delta^{(k+1)}=\delta^{(k)}$\;
				}		
			}{
				Reject $\mu^{(k+1)}$, shrink the TR radius $\delta^{(k+1)} = \beta_1\delta^{(k)}$ and go to \ref{TRRB-optstep}\;		
			}
		}
		\If { $\|\mu^{(k+1)}-P_\Params(\mu^{(k+1)}-\nabla_\mu \Jhat_h(\mu^{(k+1)}))\|_2\leq \tau_{\text{\rm{FOC}}}$ } {
			Set \textsf{Loop\_flag}$=$\textsf{False}\;
		} 
		Set $k=k+1$\;

	}
	\caption{TR-RB algorithm} 
	\label{Alg:TR-RBmethod}
\end{algorithm2e}

\subsection{Convergence analysis}
\label{sec:TR_convergence_analysis}
In order to guarantee the well-posedness (because of \eqref{TR_radius_condition}) and the convergence of the method, we make the following assumption
\begin{assumption}
\label{asmpt:bound_J}
The cost functional $\J(u,\mu)$ is strictly positive for all $u\in V$ and all parameters $\mu\in\Params$. 
\end{assumption}
Note that this assumption is not too restrictive, since the boundedness from below is a standard assumption in optimization to guarantee the existence of a solution for the minimization problem. If a global lower bound for the cost functional is also known, one can add a sufficiently large constant, without changing the position of its local minima and maxima. Another important request, pointed out in \cite{QGVW2017,YM2013}, is that an error-aware sufficient decrease condition
\begin{align}
\label{Suff_decrease_condition}
\cJhatn^{(k+1)}(\mu^{(k+1)})\leq \cJhatn^{(k)}(\mu_\text{\rm{AGC}}^{(k)}) && \text{ for all } k\in\mathbb{N}
\end{align}
is fulfilled at each iteration $k$ of the TR-RB algorithm. As in \cite{QGVW2017,YM2013}, we consider cheaply computable sufficient and necessary conditions for \eqref{Suff_decrease_condition} in Algorithm~\ref{Alg:TR-RBmethod} (Step~\ref{Suff_condition_TRRB} and Step~\ref{Nec_condition_TRRB}, respectively). The TR-RB algorithm rejects, then, any computed point which does not satisfy \eqref{Suff_decrease_condition}. One may be concerned of the fact that Algorithm~\ref{Alg:TR-RBmethod} may be trapped in an infinite loop where every computed point is rejected and the TR radius is shrunk all time. We point out that this never happened in our numerical tests. Anyway, we consider a safety termination criteria, which is triggered when the TR radius is smaller than the double machine precision. To prove convergence of Algorithm~\ref{Alg:TR-RBmethod}, in what follows, we then assume that this situation can not occur.
\begin{assumption}
\label{asmpt:rejection}
For each $k\geq 0$, there exists a radius $\widetilde{\delta}^{(k)}>\tau_{\rm{\text{mac}}}$ for which a solution of \eqref{TRRBsubprob} is such that $\eqref{Suff_decrease_condition}$ is verified, where $\tau_{\rm{\text{mac}}}= 2.22\cdot10^{-16}$ is the double machine precision.
\end{assumption}
\begin{lemma}
	\label{Lemma:AGC_search}
	Let Assumptions~{{\rm{\ref{asmpt:parameter_separable}}}--{\rm{\ref{asmpt:rejection}}}} hold true. The search of the AGC point defined in Definition~{\rm{\ref{Def:AGC}}} takes finitely many iterations at each step $k$ of the TR-RB Algorithm.
\end{lemma}
\begin{proof}
	We want to prove that there exists an index $j_c^{(k)}<\infty$ for each $k\geq 0$, for which $\mu_\text{\rm{AGC}}^{(k)}= \mu^{(k,0)}(j^{(k)}_c)$ satisfies
	\eqref{Armijo}-\eqref{TR_radius_condition} for $\ell=0$. From \cite[Theorem~5.4.5]{Kel99} (and the subsequent discussion) we conclude that for all $k\in\mathbb{N}$ there exists a strictly positive index $j^{(k)}_1\in\mathbb{N}$ such that $\mu^{(k,0)}(j)$ satisfies \eqref{Armijo} for $j\geq j^{(k)}_1$ and $\ell=0$. If $k=0$, by construction we have that $\Delta_{\Jhat_\red^{(0)}}(\mu^{(0)}) = 0$. Therefore, there exists a sufficiently large (but finite) index $j^{(0)}_2\in\mathbb{N}$ such that $\mu^{(0,0)}(j)$ satisfies \eqref{TR_radius_condition} for all $j\geq j^{(0)}_2$ and $\ell=0$. The reason relies on the 
	continuity w.r.t.~$\mu$ of the error estimator $\Delta_{\Jhat_\red^{(k)}}(\mu)$ (cf.~Remark~\ref{continuity_of_estimator}) and of the cost functional $\cJhatn^{(k)}(\mu)$ for all $k\in\mathbb{N}$. Hence there exists $j^{(0)}_c=\max(j^{(0)}_1,j^{(0)}_2)<\infty$, for which $\mu^{(0,0)}(j)$ satisfies \eqref{Armijo}-\eqref{TR_radius_condition} for $\ell=0$. If $k\geq 1$, 
	since the model has been enriched, i.e.~$\Delta_{\Jhat_\red^{(k)}}(\mu^{(k)}) = 0$, we can show the claim arguing as we did for $k=0$. 
	Note, in fact, that we increase the iteration counter only when $\mu^{(k)}$ is accepted at iteration $k-1$ and, thus, when the RB model is enriched at this parameter. 
\end{proof}
\begin{theorem}
\label{Thm:convergence_of_TR}
	Let the hypotheses of Lemma~{\rm{\ref{Lemma:AGC_search}}} be verified. 
	Then every accumulation point $\bar\mu$ of the sequence $\{\mu^{(k)}\}_{k\in\mathbb{N}}\subset \Params$ generated by the TR-RB algorithm 
	is an approximate first-order critical point for $\Jhat_h$ (up to the chosen tolerance $\tau_\text{\rm{sub}}$), i.e., it holds
	\begin{equation}
	\label{First-order_critical_condition}
	\|\bar \mu-P_\Params(\bar \mu-\nabla_\mu \Jhat_h(\bar \mu))\|_2 \leq \tau_\text{\rm{sub}}.
	\end{equation}
\end{theorem}
\begin{proof}
	The set $\Params\subset \mathbb{R}^P$ is compact. Therefore there exists a sequence of indices $\left\{k_i\right\}_{i\in\mathbb{N}}$ such that the sub-sequence $\{\mu^{(k_i)}\}_{i\in\mathbb{N}}$ converges to a point $\bar \mu\in \Params$.
	It remains to show that $\bar \mu$ is an approximate first-order critical point. At first, note that once the RB space is enriched at a point $\mu^{(k)}$, we have $\Delta_{\Jhat_\red^{(k)}}(\mu^{(k)})= 0$. Hence, also $q^{(k)}(\mu^{(k)}) = 0$ holds, where
	\begin{align*}
	q^{(k)}(\mu) := \frac{\Delta_{\Jhat_\red^{(k)}}(\mu)}{\Jhat_\red^{(k)}(\mu)} && \text{for all } k\in\mathbb{N},\, \mu\in\Params.
	\end{align*}
	Note that both the estimator $\Delta_{\Jhat_\red^{(k)}}$ and $q^{(k)}$ are uniformly continuous on $\Params$ for all $k\geq 0$. This follows directly from Remark~\ref{continuity_of_estimator} and the Heine-Cantor theorem. When the model is enriched at a parameter $\mu^{(k)}$, from the uniform continuity of $q^{(k)}$ it follows that for all $\varepsilon>0$ there exists an $\eta^{(k)}>0$ (depending on $\varepsilon$) such that
	$\|\mu^{(k)}-\mu\|_2< \eta^{(k)}$ implies
	\[
	\big|q^{(k)}(\mu)-\underset{= 0}{\underbrace{q^{(k)}(\mu^{(k)})}}\big| < \varepsilon.
	\]
    Furthermore, due to the convergence of the sub-sequence $\left\{\mu^{(k_i)}\right\}_{i\in\mathbb{N}}\subset\Params$, we have that there exists a sufficiently large constant $I>0$ and a constant $\gamma>0$ such that $\|\mu^{(k_i)}-\mu^{(k_{i+1})}\|_2< \gamma < \eta^{(k_i)}$ for all $i \geq I$. Then we have
    \begin{align}
    \label{unif_cont_estimator_step}
    q^{(k_i)}(\mu^{(k_{i+1})}) = \frac{\Delta_{\Jhat_\red^{(k_i)}}(\mu^{(k_{i+1})})}{\Jhat_\red^{(k_i)}(\mu^{(k_{i+1})})}<\varepsilon && \text{ for all } i\geq I.
    \end{align}
    We want to prove that $q^{(k_{i+1}-1)}(\mu^{(k_{i+1})})< \beta_2 \delta^{(k_{i+1}-1)}$ for all $i\geq I$,  such that the unique solution $\mu^{(k_{i+1})}$ to \eqref{TRRBsubprob} (for $k=k_{i+1}-1$) is not triggering the termination criteria \eqref{Cut_of_TR}. Note that $\varepsilon$ in \eqref{unif_cont_estimator_step} can be chosen appropriately (which implies a certain $\eta^{(k_i)}$ for all $i\in\mathbb{N}$ and thus a sufficiently large index $I$, of course). Since the RB space is enriched at each iteration of Algorithm~\ref{Alg:TR-RBmethod}, we especially have that $\Delta_{\Jhat_\red^{(k_{i+1}-1)}}(\mu^{(k_{i+1})})\leq \Delta_{\Jhat_\red^{(k_i)}}(\mu^{(k_{i+1})})$. Using \eqref{unif_cont_estimator_step}, we find that $\Delta_{\Jhat_\red^{(k_{i+1}-1)}}(\mu^{(k_{i+1})})\leq \Delta_{\Jhat_\red^{(k_i)}}(\mu^{(k_{i+1})})< \varepsilon \Jhat_\red^{(k_i)}(\mu^{(k_{i+1})})$. Hence, $q^{(k_{i+1}-1)}(\mu^{(k_{i+1})})< \beta_2 \delta^{(k_{i+1}-1)}$ holds for 
    \[
    \varepsilon = \beta_2 \delta^{(k_{i+1}-1)}\frac{\Jhat_\red^{(k_{i+1}-1)}(\mu^{(k_{i+1})})}{\Jhat_\red^{(k_i)}(\mu^{(k_{i+1})})}
    \] 
    and for all $i\geq I$.
	This shows that from a certain iteration $I$, we are far enough from the
	boundary of the Trust-Region for all $i\geq I$, so that \eqref{Cut_of_TR} does not affect the projected BFGS algorithm. Thus, \eqref{FOC_subproblem} must hold for $\mu^{(k_{i+1})}= \mu^{(k_{i+1}-1,L^{(k_{i+1}-1)})}$ for $i\geq I$. Hence, we have proved that each $\mu^{(k_{i+1})}$ is an approximate first-order critical point for $\cJhatn^{(k_{i+1}-1)}$ (up to the chosen tolerance $\tau_\text{\rm{sub}}$) for all $i\geq I$, which yields to
	\begin{align*}
		\|\mu^{(k_{i+1})}-P_\Params(\mu^{(k_{i+1})}-\nabla_\mu\cJhatn^{(k_{i+1}-1)}(\mu^{(k_{i+1})})\|_2 \leq \tau_\text{\rm{sub}}, && \text{for all } i\geq I.
	\end{align*}
	Moreover, taking into account the RB method properties and the fact that $V_h$ is a finite dimensional space, there exists a constant $I_\nabla>0$ sufficiently large, such that $\nabla_\mu \cJhatn^{(k_i)}(\mu)= \nabla_\mu \Jhat_h(\mu)+\epsilon^{(k_i)}$ for all $\mu$ in a neighborhood of $\bar\mu$ and for $i\geq I_{\nabla}$, with $\epsilon^{(k_i)}\to 0$ as $i\to\infty$. 
	Thus, exploiting the continuity of the projection operator and assuming $i\geq \max(I,I_\nabla)$, we have that
	\[
	\begin{aligned}
	\tau_\text{\rm{sub}} & \geq \|\mu^{(k_{i+1})}-P_\Params(\mu^{(k_{i+1})}-\nabla_\mu\cJhatn^{(k_{i+1}-1)}(\mu^{(k_{i+1})})\|_2 \\
	& = \|\mu^{(k_{i+1})}-P_\Params(\mu^{(k_{i+1})}-\nabla_\mu\Jhat_h(\mu^{(k_{i+1})})+\epsilon^{(k_{i+1}-1)})\|_2 \to \|\bar\mu - P_\Params(\bar\mu-\nabla_\mu\Jhat_h(\bar \mu))\|_2.
	\end{aligned}
	\]
	Hence, the accumulation point $\bar \mu$ is an approximate first-order critical point (up to the tolerance $\tau_\text{\rm{sub}}$).
\end{proof}
\begin{remark}
\label{rmk:local_minimum_TR}
What remains to prove is that $\bar \mu$ is a local minimum of $\Jhat_h$ (or rather a sufficiently close approximation of a local minimum). Exploiting the sufficient decrease condition, one can easily show by contradiction that $\bar \mu$ is not a maximum of $\Jhat_h$. It can still be a saddle point as well as a local minimum. In the numerical experiments, to verify that the computed point $\bar \mu$ is actually a local minimum, we employ the second-order sufficient optimality conditions after the algorithm terminates.
\end{remark}

\subsection{Construction of RB spaces}
\label{sec:construct_RB} 
In an enrichment step of the outer
loop of the TR-algorithm \ref{Alg:TR-RBmethod} for $\mu \in \Params$, we assume to have access to the primal and dual
solutions $u_{h, \mu}, p_{h, \mu} \in V_h$
and consider two strategies to enrich the RB spaces.

\begin{enumerate}[(a)]
	\item \emph{Lagrangian RB spaces}: \label{enrich:lag} We add each FOM solution to the RB space that is directly related to its respective reduced formulation, i.e.~for a
	given $\mu \in \Params$, we enrich by
	$
	V^{\pr,k}_{\red} = V^{\pr,k-1}_{\red} \cup \{u_{h,\mu}\},  
	V^{\du,k}_{\red} = V^{\du,k-1}_{\red} \cup \{p_{h,\mu}\}.
	$
	\item \emph{Single RB space}: We add all available information into a single RB space,
  i.e.~$V^{\pr, k}_{\red} = V^{\du, k}_{\red} = V^{\pr,k-1}_{\red} \cup \{u_{h,\mu}, p_{h,\mu}\}$. 
	According to Section~\ref{sec:mor} this results in  
	$\cJhatn(\mu)$ being equal to the standard RB reduced functional from \eqref{eq:Jhat_red}. 
	\label{enrich:single}
\end{enumerate} 
These strategies for constructing RB spaces have a significant impact on the performance and accuracy of the TR-RB method.
Note that offline computations for the construction of RB models scale quadratically with the number of
basis functions in the RB space.
Thus, Lagrange RB spaces in \eqref{enrich:lag} are computationally beneficial compared to \eqref{enrich:single} at a potential loss of accuracy
of the corresponding RB models (since less information is added). Moreover, different spaces as in \eqref{enrich:lag} destroy the duality of state and adjoint equations, cf.~Section~\ref{sec:standard_approach}. 

\subsection{Trust-Region variants based on adaptive enrichment strategies}
\label{sec:high_order_TR}
A major contribution of this article is to introduce and analyze variants of adaptive TR-RB methods with projected BFGS as sub-problem solver for efficiently computing a 
solution of the optimization problem \eqref{P}.
In terms of performance we need to account for all computational costs, including traditional offline and online costs of the algorithms.
The proposed methods mainly differ in terms of the model function and its gradient information. 
Following Section~\ref{sec:TR}, we propose a TR method which adaptively builds
an RB space along the path of optimization (see Algorithm \ref{Alg:TR-RBmethod}). From a MOR perspective this diminishes
the offline time of the ROM significantly since no global RB space (with respect to the parameter domain) has
to be built in advance. 
We enrich the model after the sub-problem \eqref{TRRBsubprob} of the TR method has been solved. 
We distinguish  three different approaches:
\begin{enumerate}[1.]
	\item \emph{standard approach:} Following Section~\ref{sec:standard_approach}, the standard approach for the functional 
	is to replace the FOM quantities by their respective ROM counterpart, i.e.~we consider the map $\mu \mapsto \Jnoncor_\red(\mu)$ from \eqref{eq:Jhat_red}. 
	Gradient information can be computed by reducing the corresponding FOM gradient which results in 
	$\noncorgrad_\mu \J(u_{\red, \mu}, \mu)$ from \eqref{naive:red_grad}. Consequently this approach does not allow for using a higher order estimate but 
	$\Delta_{\Jnoncor_\red}(\mu)$. 
	\item \emph{semi NCD-corrected approach:} A first correction strategy is to replace the functional by the NCD-corrected RB reduced functional $\cJhatn$
	from \eqref{eq:Jhat_red_corected} but stick with the inexact gradient of the standard approach.
  This allows to use the higher order estimator for the functional, i.e.
	$\Delta_{\cJhatn}(\mu)$. 
	\item \emph{NCD-corrected approach:} 
    We propose to consider the NCD-corrected RB reduced functional $\cJhatn$ from \eqref{eq:Jhat_red_corected} and its actual gradient according to 
	Propostion~\ref{prop:true_corrected_reduced_gradient_adj}.
  Note that we only need to solve two additional equations, independently of the dimension $P$ of the parameter space. 
\end{enumerate}
For the basis construction, we may use variants \eqref{enrich:lag} or \eqref{enrich:single} from Section~\ref{sec:construct_RB}. 
Note however, that by using 
the basis enrichment \eqref{enrich:single}, all approaches $1.$ - $3.$ are equivalent. 
Using variant \eqref{enrich:lag} with BFGS is inspired from \cite{QGVW2017}. However, our algorithms differ from the 
TR-RB approach in \cite{QGVW2017} since we are working with the NCD-corrected reduced cost functional (in 2) and its actual gradient (in 2 and 3). 
Note that the presence of inequality constraints, which are missing in \cite{QGVW2017}, implies a projection-based optimization algorithm.
In addition, we stress that, differently from \cite{QGVW2017}, we take advantage of the proposed condition for 
enlarging the TR radius and of a stopping criterium independent from the RB a posteriori estimates, as presented in Section~\ref{sec:TR}.

\begin{remark}
	Note that we do not use the sensitivity based quantities from Section~{\rm{\ref{sec:sens_estimation}}} although they suggest the highest
	numerical accuracy w.r.t.~the FOM optimality system.
  However, for the experiments in Section~\ref{sec:num_experiments}, 
  additional computational cost for computing
	FOM sensitivities will not pay off in the TR-RB algorithm, especially for high-dimensional parameter spaces.  
\end{remark}

\section{Numerical experiments}
\label{sec:num_experiments}
 
We present numerical experiments to demonstrate the adaptive TR-RB variants from Section~\ref{sec:high_order_TR} with both RB constructions from Section~\ref{sec:construct_RB} for quadratic objective functionals with elliptic PDE constraints as in \eqref{P}, and compare them to state-of-the art algorithms from the literature.
We also validate the higher-order a posteriori error estimates from Section \ref{sec:a_post_error_estimates} numerically.
We consider two setups: first, the elliptic thermal fin problem from \cite[Sec.~5.1.1]{QGVW2017} (where the correction term of the proposed NCD-corrected approach vanishes) in Section~\ref{sec:thermalfin}. Second, we consider a more challenging optimization problem in Section~\ref{sec:mmexc_example}, including a detailed analysis of the a posteriori error estimates from Section~\ref{sec:a_post_error_estimates}.
All simulations have been performed with a pure \texttt{Python} implementation based on the open source MOR library pyMOR \cite{milk2016pymor}, making use of pyMORs builtin vectorized \texttt{numpy/scipy}-based discretizer for the FOM and generic MOR algorithms for projection and orthonormalization (such as a stabilized Gram-Schmidt algorithm) to effortlessly obtain efficient ROMs.
The source code to reproduce all results (including detailed interactive \texttt{jupyter}-notebooks\footnote{Available at \url{https://github.com/TiKeil/NCD-corrected-TR-RB-approach-for-pde-opt}.}) is available at \cite{Code}.
All experiments are based on the same implementation (including a reimplementation of \cite{QGVW2017}) and were performed on the same machine multiple times to avoid caching or multi-query effects.
Timings may thus be used to compare and judge the computational efficiency of the different algorithms. 

We consider stationary heat transfer in a bounded connected spatial domain $\Omega \subset \R^2$ with polygonal boundary $\partial\Omega$ partitioned into a non-empty Robin boundary $\Gamma_\text{\rm{R}} \subset \partial\Omega$ and possibly empty distinct Neumann boundary $\Gamma_\text{\rm{N}} = \partial\Omega\backslash \Gamma_\text{\rm{R}}$, and unit outer normal $n: \partial\Omega \to \R^2$.
We consider the Hilbert space $V = H^1(\Omega) := \{v \in L^2(\Omega) \,|\, \nabla v \in L^2(\Omega) \}$ of weakly differentiable functions and, for an admissible parameter $\mu \in \Params$, we seek the temperature $u_\mu \in V$ as the solution of
\begin{align}
  -\nabla\cdot\big(\kappa_\mu \nabla u_\mu \big) &= f_\mu \ \text{in }\Omega, &
  \kappa_\mu \nabla u_\mu\cdot n &= c_\mu (u_\text{\rm{out}} - u_\mu) \ \text{on }\Gamma_\text{\rm{R}},
\label{eq:heat_equation}
  & \kappa_\mu\nabla u_\mu \cdot n &= g_\text{\rm{N}} \ \text{on }\Gamma_\text{\rm{N}}
\end{align}
in the weak sense, with the admissible parameter set, the spatial domain and its boundaries and the data functions $\kappa_\mu \in L^\infty(\Omega)$, $f_\mu \in L^2(\Omega)$, $c_\mu \in L^\infty(\Gamma_\text{\rm{R}})$ and $u_\text{\rm{out}} \in L^2(\Gamma_\text{\rm{R}})$ defined in the respective experiment.
The bilinear form $a$ and linear functional $l$ in \eqref{P.state} are thus given for all $\mu \in \Params$ and $v, w \in V$ by
\begin{align}
  a_\mu(v, w) := \int_\Omega\kappa_\mu \nabla v \cdot \nabla w \dx + \int_{\Gamma_\text{\rm{R}}}\hspace{-5pt}c_\mu\, vw\ds &&\text{and}&& l_\mu(v) := \int_\Omega f_\mu\,v\dx + \int_{\Gamma_\text{\rm{R}}}\hspace{-5pt}c_\mu\, u_\text{\rm{out}}v\ds + \int_{\Gamma_\text{\rm{N}}}\hspace{-5pt}g_\text{\rm{N}}v\ds.
\end{align}
For the FOM we fix a fine enough reference simplicial or cubic mesh and define $V_h \subset V$ as the respective space of continuous piecewise (bi-)linear Finite Elements.

Since the inner product and norm have a big influence on the computational efficiency of the a posteriori error estimates as well as their sharpness, we use the mesh-independent energy-product $(u, v) := a_{\check{\mu}}(u, v)$ for a fixed parameter $\check{\mu} \in \Params$, which is a product over $V$ due to the symmetry, continuity and coercivity of the bilinear form for each example below.
Owing to this choice of the product, we may use the $\min$-theta approach from \cite[Prop.~2.35]{HAA2017} to obtain lower bounds on coercivity constants and the $\max$-theta approach from \cite[Ex.~5.12]{HAA2017} to obtain upper bounds on continuity constants, each required for the a posteriori error estimates.
Compared to the more general Successive Constraint Method \cite{pat07}, this approach yields quite sharp estimates and is computationally more efficient, both offline and online.
Due to Assumption~\ref{asmpt:parameter_separable} and the bi-linearity of the objective functional, we may carry out the preassembly of all high-dimensional quantities after each enrichment, which is well-known for RB methods \cite[Sec.~2.5]{HAA2017}.
We would like to point out that while the more accurate and stable preassembly of the estimates from \cite{buhr14} is readily available in pyMOR, the slightly cheaper standard preassembly of the estimates was sufficient for our experiments.

For all experiments, we use an initial TR radius of $\delta^0 = 0.1$,
a TR shrinking factor $\beta_1=0.5$, an Armijo step-length $\kappa=0.5$,
a truncation of the TR boundary of $\beta_2 = 0.95$,
a tolerance for enlarging the TR radius of $\eta_\varrho = 0.75$,
a stopping tolerance for the TR sub-problems of $\tau_\text{\rm{sub}} = 10^{-8}$,
a maximum number of TR iteration $K = 40$, a maximum number of sub-problem iteration $K_{\text{\rm{sub}}}= 400$ and
a maximum number of Armijo iteration of $50$. We also point out that the stopping tolerance for the FOC condition 
$\tau_\text{\rm{FOC}}$ is specified in each experiment.  

\subsection{State of the art optimization methods}
\label{sec:state_of_the_art_methods}

\noindent
We compare our proposed methods to the following ones from the literature:

\textbf{Adaptive TR-RB with BFGS sub-problem solver and Lagrangian RBs \cite{QGVW2017}}:
	We consider the same method as in \cite{QGVW2017}, where the authors used the standard functional and gradient from Section~\ref{sec:standard_approach}. 
	Furthermore, no enlarging strategy has been used for the TR-radius and no projection for parameter constraints has been considered.
	Importantly, the authors did not take advantage of the fact, that the full order FOC condition in line 23 of Algorithm \ref{Alg:TR-RBmethod}
	is cheaply available after an enrichment step. Instead they used the reduced FOC condition plus the estimator for the gradient of the
	cost functional $\|\noncorgrad_\mu \Jnoncor_\red(\mu^{(k+1)}))\|_2 + \Delta_{\noncorgrad_\mu \Jnoncor_\red}(\mu) \leq \tau_\text{\rm{FOC}}$ in line 23. Note that this 
	approach has multiple drawbacks. First, the evaluation is more costly due to the estimator. Second, it is less accurate and third, it can
	prevent the TR-RB from converging in case the estimator is not able to be small enough (for instance governed by high constants or numerical
	issues in the estimator).

\textbf{FOM projected BFGS}:  We consider a standard projected BFGS method,
	which uses FOM evaluations of the forward model to compute the reduced cost functional and its gradient. 
	We restrict the maximum number of iterations by $400$.

\subsection{Model problem 1: Elliptic thermal fin model problem}
\label{sec:thermalfin}

We consider the six-dimensional elliptic thermal fin example from \cite[Sec.~5.1.1]{QGVW2017} and refer to Figure~\ref{fig:fin} for the problem definition.
The purpose of this experiment is to show the applicability of the proposed algorithms and to compare them to the one proposed in \cite{QGVW2017}.
For all runs we prescribe the same desired parameter $\mu^\text{d} \in \Params$ by randomly drawing $k_1, \dots, k_4$ strictly within $\Params$ and by setting $k_0 = 0.1$ and $\text{Bi} = 0.01$, to artificially mimic the situation where parameter constraints have to be tackled.
Defining $T^\text{d} := q(u_{\mu^\text{d}})$ where $u_{\mu^\text{d}} \in V$ is the solution of \eqref{P.state} associated with the desired parameter and where $q(v) := \int_{\Gamma_\text{\rm{N}}} v \ds$ for $v\in V$ denotes the mean temperature at the root of the fin, we consider a cost functional $\J(u, \mu) = \Theta(\mu) + j_\mu(u) + k_\mu(u, u)$ as in \eqref{P.argmin} with $\Theta(\mu) := (\|\mu^\text{d} - \mu\| / \|\mu^\text{d}\|)^2 + {T^\text{d}}^2 + 1$, $j_\mu(v) := - T^\text{d}\,q(v)$ and $k_\mu(v, w) := 1/2\,q(v)\,q(w)$.
We would like to point out that the authors in \cite{QGVW2017} dropped the ${T^\text{d}}^2 + 1$ term from the definition of $\Theta$, which we re-add to ensure Assumption~\ref{asmpt:bound_J}. This constant term does not change the position of local minima and the derivatives of the cost functional. However, this makes the Trust-Region radius shrink especially at the beginning, slowing down the TR-RB methods. This does not affect the comparison among the TR-RB methods, since all suffer from this issue. 
Note that for this particular example, the proposed NCD-correction term vanishes, see Remark~\ref{rem:fin}.
For the FOM, we generate an unstructured simplicial mesh using pyMORs \texttt{gmsh} (see \cite{gmsh}) bindings, resulting in $\dim V_h = 77537$.

\begin{SCfigure}[50][ht]
	\centering\footnotesize
	\begin{tikzpicture}[scale=1.1]
	\draw (-.5,0) -- (-.5,.75);
	\draw (-.5,.75) -- (-3.,.75);
	\node at (-1.8,0.86) {{\tiny $k_1$}};
	\node at (-1.8,1.86) {{\tiny $k_2$}};
	\node at (-1.8,2.86) {{\tiny $k_3$}};
	\node at (-1.8,3.86) {{\tiny $k_4$}};
	\node at (1.8,0.86) {{\tiny $k_1$}};
	\node at (1.8,1.86) {{\tiny $k_2$}};
	\node at (1.8,2.86) {{\tiny $k_3$}};
	\node at (1.8,3.86) {{\tiny $k_4$}};
	\node at (0,2) {{\tiny $k_0$}};
	\draw (-3.,.75) -- (-3.,1) -- (-.5,1);
	\draw (-.5,1) -- (-.5,1.75) -- (-3.,1.75) -- (-3.,2) -- (-.5,2);
	\draw (-.5,2) -- (-.5,2.75) -- (-3.,2.75) -- (-3.,3) -- (-.5,3);
	\draw (-.5,3) -- (-.5,3.75) -- (-3.,3.75) -- (-3.,4) -- (3,4);
	\draw (.5,0) -- (.5,.75) -- (3.,.75) -- (3.,1) -- (.5,1);
	\draw (.5,1) -- (.5,1.75) -- (3.,1.75) -- (3.,2) -- (.5,2);		
	\draw (.5,2) -- (.5,2.75) -- (3.,2.75) -- (3.,3) -- (.5,3);		
	\draw (.5,3) -- (.5,3.75) -- (3.,3.75) -- (3.,4);		
	\draw (-.5,0) -- (.5,0) node [midway,above] {$\Gamma_\text{\rm{N}}$};
	\draw[densely dotted] (-.5,.75) -- (-.5,1);
	\draw[densely dotted] (-.5,1.75) -- (-.5,2);
	\draw[densely dotted] (-.5,2.75) -- (-.5,3);
	\draw[densely dotted] (-.5,3.75) -- (-.5,4);
	\draw[densely dotted] (.5,.75) -- (.5,1);
	\draw[densely dotted] (.5,1.75) -- (.5,2);
	\draw[densely dotted] (.5,2.75) -- (.5,3);
	\draw[densely dotted] (.5,3.75) -- (.5,4);
	\draw[->|] (3.2,0.45) -- (3.2,0.75);
	\draw[->|] (3.2,1.3) -- (3.2,1.);
  \draw[<->|] (0.5,0.45) -- (3,0.45) node[midway, fill=white] {$L$};
	\draw (3,0.75) -- (3,1) node[midway,right] {\small $t$};
	\end{tikzpicture} 
  \caption{%
  	{\small
    Problem definition of the thermal fin example from Section~\ref{sec:thermalfin}.
    Depicted is the spatial domain $\Omega$ (with $L=2.5$ and $t=0.25$) with Neumann boundary at the bottom with $|\Gamma_{\text{N}}| = 1$ and Robin boundary $\Gamma_\text{\rm{R}} := \partial\Omega \backslash \Gamma_\text{\rm{N}}$, as well as the values $k_0, \dots, k_4 > 0$ of the diffusion $k_\mu$, which is piecewise constant in the respective indicated part of the domain.
    The other data functions in \eqref{eq:heat_equation} are given by $f_\mu = 0$, $g_\text{\rm{N}} = -1$, $u_\text{\rm{out}} = 0$ and $c_\mu = \text{Bi} \in \R$, the Biot number.
    We allow to vary the six parameters $(k_0, \dots, k_4, \text{Bi})$ and define the set of admissible parameters as $[0.1, 10]^5 \times [0.01, 1] \subset \R^P$ with $P = 6$.
    We choose $\check{\mu} = (1, 1, 1, 1, 1, 0.1)$ for the energy product.}
  }
	\label{fig:fin}
\end{SCfigure}
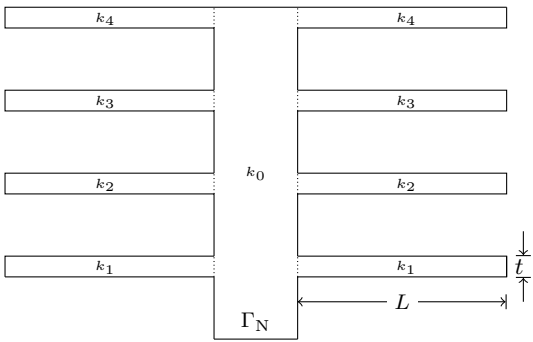

Starting with ten different randomly drawn initial parameters $\mu^{(0)}$, we measure the total computational runtime, the number of TR iterations $k$ and the error in the optimal parameter for all combinations of adaptive TR algorithms from Section~\ref{sec:TRRB_and_adaptiveenrichment} and choice of RB spaces from Section~\ref{sec:construct_RB}, as well as for the state of the art methods from the literature from Section~\ref{sec:state_of_the_art_methods}.

\begin{table}
	\centering\footnotesize
	\begin{tabular}{l|cc|c|cc}
    & av.~(min/max) runtime[s]
      & speed-up
        & av.~(min/max) iter.
          & rel.~error
            & FOC cond.\\
  \hline
  FOM proj.~BFGS
    & 967.86 (176.69/3401.06)
      & --
        & 111.20 (25/400)
          & $3.13\cdot 10^{-3}$
          & $1.19\cdot 10^{-2}$\\
  TR-RB from \cite{QGVW2017}
    & 68.06 (43.28/88.21)
      & 10.40
        & 7.20 (8/13)
          & $1.34\cdot 10^{-6}$
            & $4.31\cdot 10^{-5}$\\
  1(a) TR-RB with $V_\red^\pr \neq V_\red^\du$
    & 44.56 (34.22/74.96)
      & 21.72
        & 8.80 (8/11)
          & $3.08\cdot 10^{-6}$
            & $4.64\cdot 10^{-5}$\\
  1(b) TR-RB with $V_\red^\pr = V_\red^\du$
    & 43.86 (34.09/74.35)
      & 22.07
        & 8.70 (8/10)
          & $3.37\cdot 10^{-6}$
          & $6.40\cdot 10^{-5}$
	\end{tabular}
	\vspace{0.15cm}
	\caption{%
    {\small Performance and accuracy of selected algorithms for the example from Section~\ref{sec:thermalfin} for ten optimization runs with randomly initial guesses $\mu^{(0)}$: averaged, minimum and maximum total computational time (column 2) and speed-up compared to the FOM variant (column 3); average, minimum and maximum number of iterations $k$ required until convergence (column 4), average relative error in the parameter (column 5) and average FOC condition (column 6).}
  }
  \label{table:fin}
\end{table}

\begin{SCfigure}
	\centering\footnotesize
	\input{Pictures/mu_error_FIN.tex}
	\caption{%
    {\small Error decay and performance of selected algorithms for the example from Section~\ref{sec:thermalfin} for a single optimization run with random initial guess $\mu^{(0)}$ for $\tau_\text{\rm{FOC}} = 5\cdot 10^{-4}$: for each algorithm each marker corresponds to one (outer) iteration of the optimization method and indicates the absolute error in the current parameter, measured against the known desired optimum $\bar{\mu} = \mu^\text{d}$.
    In all except the FOM variant, the ROM is enriched in each iteration corresponding to Algorithm \ref{Alg:TR-RBmethod}, depending on the variant in question.}
  }
	\label{fig:fin_timing_mu_d}
\end{SCfigure}
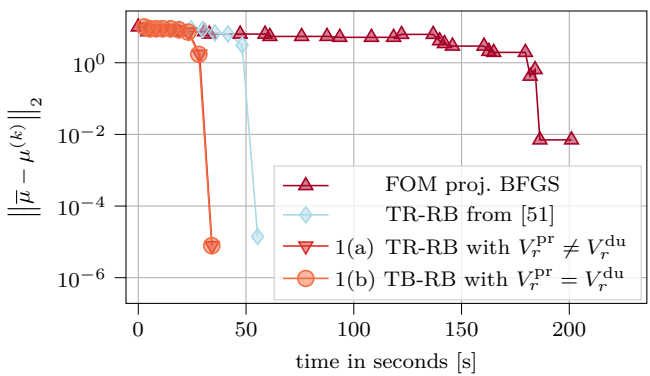

All considered optimization methods converged (up to a tolerance), but we restrict the presentation to the most informative ones (all results can be found in the accompanying code).
As we observe from Table~\ref{table:fin}, the ROM based adaptive TR-RB algorithms vastly outperform the FOM variant, noting that the computational time of the ROM variants includes all offline and online computations.
Figure~\ref{fig:fin_timing_mu_d} details the decay of the error decay in the optimal parameter during the optimization for a selected random initial guess.
We observe that the choice of the RB enrichment does not impact the performance of the algorithm for this example too much, see Remark~\ref{rem:fin}.
Also methods 2(\ref{enrich:lag}) and 3(\ref{enrich:lag}) show a comparable computational speed (not shown).
We also observe that the method from \cite{QGVW2017} requires more time and more iterations on average, 
variants 1 are still faster due to the enlarging of the TR radius and of the use of a termination criterium which does not depend on a posteriori estimates, which may force additional TR iterations.

\begin{remark}[Vanishing NCD-correction for the fin problem]
  \label{rem:fin}
	It is important to notice that this model problem is not suitable to fully demonstrate the capabilities of the NCD-corrected approach. The reason is
	that the choice of the functional is a misfit on only the root edge of the thermal fin, plus a Tikhonov regularization term.
  Since the root of the thermal fin is also the source of
	the primal problem, the dual solutions $p_{\red, \mu}$ of the reduced dual equation
	\eqref{eq:dual_solution_red} are thus linearly dependent on the respective primal solutions $u_{\red, \mu}$ and the correction term $r_\mu^\pr(u_{\red,\mu})[p_{\red,\mu}]$ for the NCD-corrected RB reduced functional from \eqref{eq:Jhat_red_corected} vanishes.
	In general, for quadratic objective functionals, this is not the case and all variants with correction terms thus waste unnecessary computational time.
\end{remark}

\subsection{Model problem 2: stationary heat distribution in a building}
\label{sec:mmexc_example}

For these experiments we consider as objective functional a weighted $L^2$-misfit on a domain of interest $D \subseteq \Omega$ and a weighted Tikhonov term comparable to design optimization, optimal control or inverse problems, i.e.
\begin{align} 
\mathcal{J}(v, \mu) = \frac{\sigma_D}{2} \int_{D}^{} (v - u^{\text{d}})^2 + \frac{1}{2} \sum^{M}_{i=1} \sigma_i (\mu_i-\mu^{\text{d}}_i)^2 + 1,
\label{eq:mmexc_example_J}
\end{align} 
with given desired state $u^{\text{d}} \in V$ and parameter $\mu^{\text{d}} \in \Params$ and weights $\sigma_D, \sigma_i$ specified further below. With respect to \eqref{P.argmin}, we thus have
$\Theta(\mu) = \frac{1}{2} \sum^{M}_{i=1} \sigma_i (\mu_i-\mu^{\text{d}}_i)^2 + \frac{\sigma_D}{2} \int_{D}^{} u^{\text{d}} u^{\text{d}} + 1$,
$j_{\mu}(u) = -\sigma_D \int_{D}^{} u^{\text{d}}u$
and
$k_{\mu}(u,v) = \frac{\sigma_D}{2} \int_{D}^{} uv$.
\begin{SCfigure}[][h]
	\centering\footnotesize
	\begin{subfigure}[c]{0.60\textwidth}
		\includegraphics[width=\textwidth]{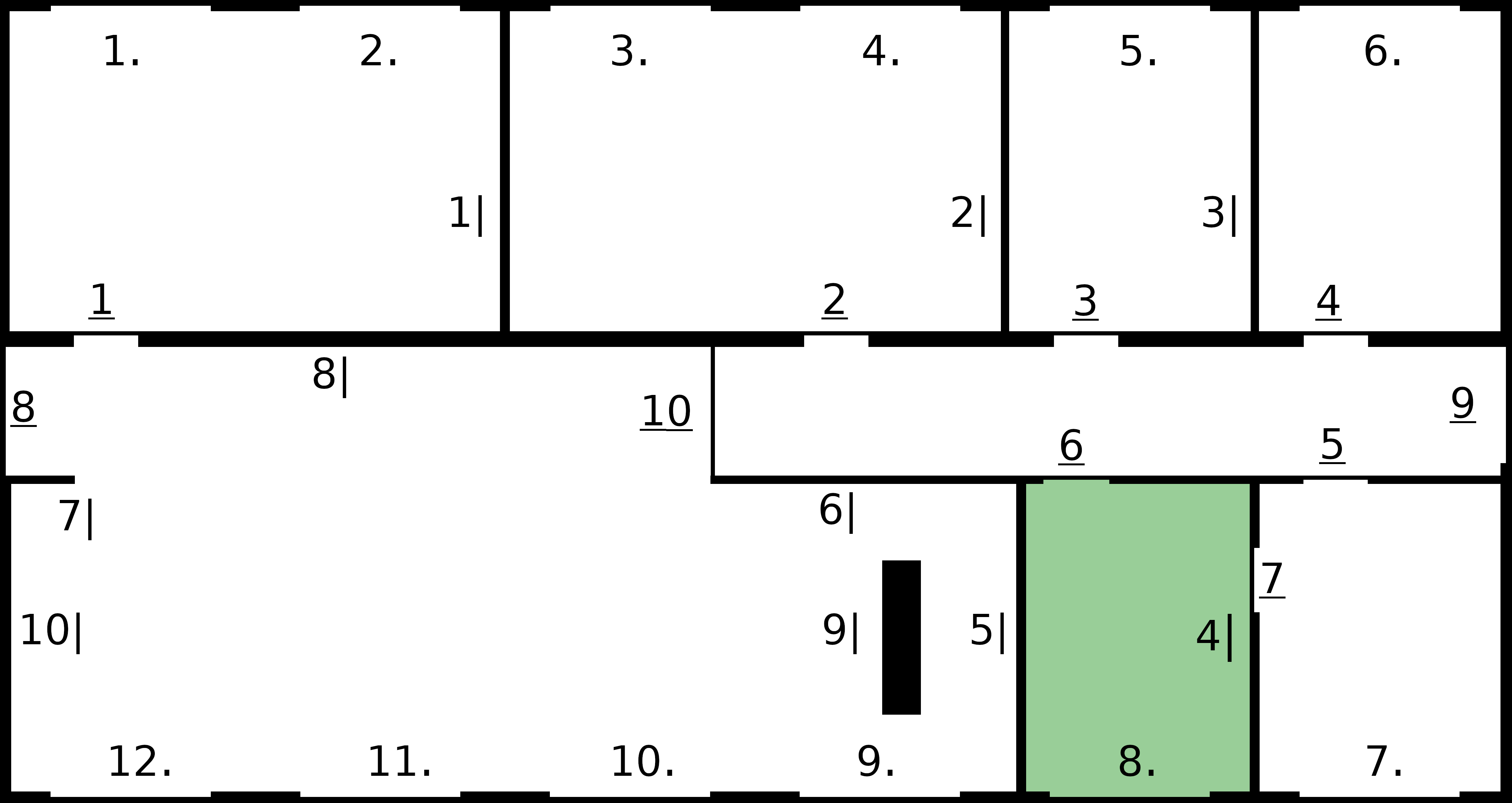}
	\end{subfigure}
	\centering
	\vspace{2em}
	\caption{{\small Definition~of Model problem 2: with $\Omega := [0,2] \times [0,1] \subset \mathbb{R}^2$.
			Numbers indicate affine components, where $i.$ is a window, $\underbar{i}$ are doors, and $i|$ are walls.
			The $i$-th heater is located under the $i$-th window. With respect to \eqref{eq:heat_equation}, we consider $\Gamma_\text{\rm{R}}:= \partial \Omega$, where
			$c_{\mu}$ contains outside wall $10|$, outside doors $\underbar{8}$ and $\underbar{9}$ and all windows. All other diffusion components enter the coefficient $\kappa_\mu$,
			whereas the heaters enter into the source term $f_\mu$. Furthermore, we set $u_{\text{out}}=5$ and
			the green region illustrates the domain of interest $D$.}}
	\label{ex1:blueprint}
\end{SCfigure}

Motivated by ensuring a desired temperature in a single room of a building floor,
we consider
blueprints
with windows, heaters, doors and walls, yielding parameterized diffusion, forces and boundary values
as sketched in Figure~\ref{ex1:blueprint}.\footnote{See \url{https://github.com/TiKeil/NCD-corrected-TR-RB-approach-for-pde-opt} for the definition of the data functions.} 
For simplicity we omit a realistic modeling of temperature and restrict
ourselves to academic numbers of the diffusion and heat source quantities.
We seek to ensure a desired temperature $u^\text{d}=18$ and set $\mu^\text{d}_i = 0$.
For the FOM discretization we choose a cubic mesh which resolves all features of the data functions derived from Figure~\ref{ex1:blueprint}, resulting in $\dim V_h = 80601$ degrees of freedom. We consider a ten-dimensional parameter example with three wall sets $\{1|,2|,3|,8|\}$, $\{4|,5|,6|,7|\}$ and $\{9|\}$ and
seven heater sets, $\{1,2\}$,
$\{3,4\}$ and $\{5\}$, $\{6\}$, $\{7\}$, $\{8\}$ and $\{9,10,11,12\}$ (each set governed by a single parameter component).
The set of admissible parameters is given by $\Params= [0.025,0.1]^3\times[0,100]^7$ and we choose
$\sigma_D= 100$ and $(\sigma_i)_{1\leq i\leq 10} = (10\sigma_w,5\sigma_w,\sigma_w,2\sigma_h,$
$2\sigma_h,\sigma_h,\sigma_h,\sigma_h,\sigma_h,4\sigma_h)$
in \eqref{eq:mmexc_example_J},
with $\sigma_w = 0.05$ and $\sigma_h= 0.001$.
The choice of $\sigma_i$ is related to the measure of the walls and how many heaters are considered in each group.
The other components of the data functions are fixed and thus not directly involved in the optimization process.
Briefly, the diffusion coefficient of air and inside doors is set to $0.5$, of the outside wall to $0.001$,
of outside doors $\underline{8}$ and $\underline{9}$ to $0.01$ and of windows to $0.025$.
For the energy product, we choose $\check{\mu}= (0.05,0.05,0.05,10,10,10,10,10,10,10)$.

We use this setup to inspect different TR-RB algorithms in Section \ref{sec:mmexc_opt_results}, but also to study the a posteriori error estimates from Section \ref{sec:a_post_error_estimates} in the following section.

\subsubsection{Numerical validation of the a posteriori error estimates}
\label{sec:estimator_study}

To study the performance of the a posteriori error estimates proposed in Section \ref{sec:a_post_error_estimates},
we neglect the outer-loop optimization and simply use a goal oriented adaptive greedy algorithm \cite{HDO2011} with basis extension (a) from Section \ref{sec:construct_RB}
to generate a ROM which ensures that the worst relative estimated error for the reduced functional and its gradient over the adaptively generated training set and a randomly chosen
validation set is below a prescribed tolerance of $\tau_{\text{\rm{FOC}}}= 5 \cdot 10^{-4}$. 
In particular we first ensure $\Delta_{\hat{J}_\red}(\mu)/\Jnoncor_\red(\mu) < \tau_{\rm{FOC}}$ for $\Delta_{\hat{J}_\red}$ from Proposition \ref{prop:Jhat_error}.i and continue with $\Delta_{\noncorgrad \Jnoncor_\red}(\mu)/\|\noncorgrad \Jnoncor_\red(\mu)\|_2 < \tau_{\noncorgrad \Jnoncor}$ for $\Delta_{\noncorgrad \Jnoncor_\red}$ from Proposition \ref{prop:grad_Jhat_error}.i, cf. \cite[Algorithm 2]{QGVW2017}. 
Let us mention that the goal for $\Delta_{\hat{J}_\red}$ is fulfilled after $24$ basis enrichments and we have $\Delta_{\noncorgrad \Jnoncor_\red}(\mu)/\|\noncorgrad \Jnoncor_\red(\mu)\| < 4.84$ after $56$ basis enrichments, where we artificially stop the algorithm
since the associated computational effort is already roughly 17 hours, demonstrating the need for the proposed adaptive TR-RB algorithm studied in the next section.

\begin{figure}[ht]
  \centering%
  \footnotesize%
	\input{Pictures/estimator_study.tex}
	\caption{%
    {\small Evolution of the true and estimated model reduction error (top) in the reduced functional and its approximations (A) and the gradient of the reduced functional and its approximations (B), as well as error estimator efficiencies (bottom), during adaptive greedy basis generation for the experiment from Section \ref{sec:estimator_study}.
    Top: depicted is the $L^\infty(\Params_\textnormal{val})$-error for a validation set $\Params_\textnormal{val} \subset \Params$ of $100$ randomly selected parameters, i.e.~$|\hat{J}_h - \Jnoncor_\red|$ corresponds to $\max_{\mu \in \Params_\textnormal{val}} |\hat{J}_h(\mu) - \Jnoncor_\red(\mu)|$, $\Delta_{\cJhatn}$ corresponds to $\max_{\mu \in \Params_\textnormal{val}}\Delta_{\cJhatn}(\mu)$, $\|\nabla \hat{\mathcal{J}}_h - \nabla\cJhatn\|_2$ corresponds to $\max_{\mu \in \Params_\textnormal{val}}\|\nabla \hat{\mathcal{J}}_h(\mu) - \nabla\cJhatn(\mu)\|_2$, and so forth.
    Bottom: depicted is the worst efficiency of the respective error estimate (higher: better), i.e.~``$\Delta_{\hat{J}_\red}$ eff.'' corresponds to $\min_{\mu \in \Params_\textnormal{val}} |\hat{J}_h(\mu) - \Jnoncor_\red(\mu)|\,/\,\Delta_{\hat{J}_\red}(\mu)$, and so forth.}
  }
	\label{fig:estimator_study}
\end{figure}
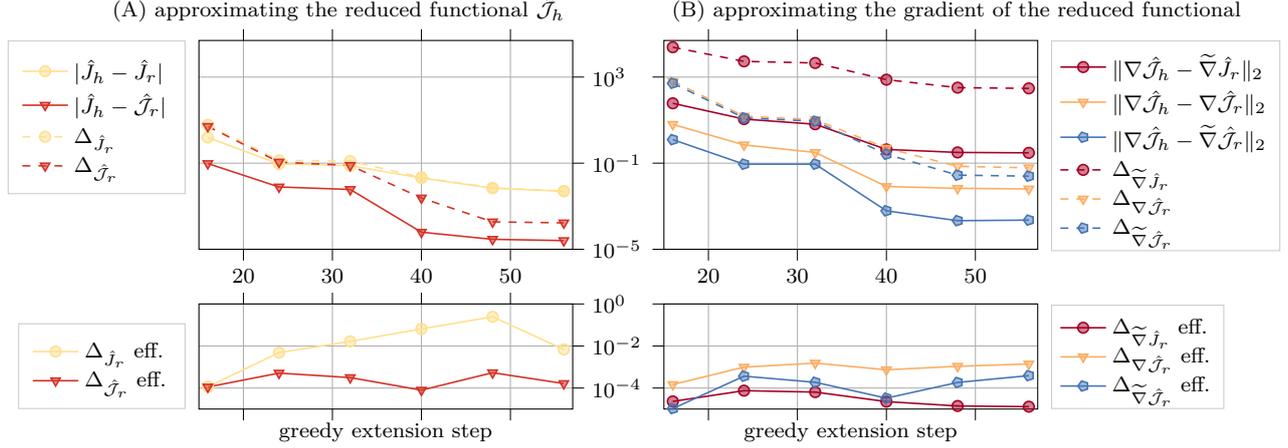

As we observe from Figure \ref{fig:estimator_study}, the error of the NCD-corrected terms is of several orders of magnitude smaller than the corresponding terms of the standard approach.
It can also be seen that the (computationally more costly) sensitivity bases quantities, i.e.~$\noncorgrad\cJhatn$, show the best error.
However, all estimators for the corrected and sensitivity based quantities show a worse effectivity, hinting that there is still room for improvement.

\subsubsection{Optimization results}
\label{sec:mmexc_opt_results}

Similar to Section \ref{sec:thermalfin}, starting with ten different randomly drawn initial parameters $\mu^{(0)}$, we measure the total computational runtime, the number of TR iterations $k$ and the error in the optimal parameter for all combinations of adaptive TR algorithms from Section~\ref{sec:TRRB_and_adaptiveenrichment} and choice of RB spaces from Section~\ref{sec:construct_RB}, as well as for the state of the art methods from the literature from Section~\ref{sec:state_of_the_art_methods}.

All algorithms converged
(up to a tolerance) to the same point $\bar\mu$ and it was verified a posteriori that this point is a local minimum of $\Jhat$,
i.e.~it satisfies the second-order sufficient optimality conditions. 
The value of $\bar\mu$ in order to compute the relative error was calculated with the
FOM projected Newton method for a FOC condition tolerance of $10^{-12}$ and,
thanks to the choice of the cost functional weights,
the target $u^\text{d}$ is approximate by $\bar u$ with a relative error of $1.7\cdot 10^{-6}$ in $D$.
We consider the same setup for two different stopping tolerances $\tau_{\text{\rm{FOC}}}= 5\cdot 10^{-4}$ and $\tau_{\text{\rm{FOC}}}= 10^{-6}$ to demonstrate that the performance (both in terms of time and convergence) of the methods vastly depends on the choice of $\tau_{\text{\rm{FOC}}}$.
\begin{table}
	\centering\footnotesize
	\begin{subtable}{\textwidth}
	\centering
	\begin{tabular}{l|cc|c|cc}
    \textsc{\textbf{(A) Result for $\tau_{\text{\rm{FOC}}}= 5\cdot 10^{-4}$}}
    & av.~(min/max) runtime[s]
      & speed-up
        & av.~(min/max) iter.
          & rel.~error
            & FOC cond.\\
  \hline
  FOM proj.~BFGS
    & 332.57 (196.51/591.85)
      & --
        & 44.30 (30/60)
          & $1.40\cdot 10^{-3}$
          & $1.80\cdot 10^{-4}$\\
  TR-RB from \cite{QGVW2017}
    & 117.87 (70.29/166.31)
      & 2.82
        & 10.10 (6/14)
          & $5.46\cdot 10^{-4}$
            & $1.41\cdot 10^{-4}$\\
  1(a) TR-RB with $V_\red^\pr \neq V_\red^\du$
    & 91.50 (47.07/230.29)
      & 3.63
        & 8.30 (5/10)
          & $2.01\cdot 10^{-3}$
            & $2.04\cdot 10^{-4}$\\
  1(b) TR-RB with $V_\red^\pr = V_\red^\du$
    & 78.65 (54.69/114.36)
      & 4.23
        & 6.90 (5/9)
          & $2.53\cdot 10^{-4}$
          & $8.23\cdot 10^{-5}$ \\
  2(a) TR-RB semi NCD-corrected
    & 79.47 (63.38/94.28)
      & 4.18
        & 8.50 (7/10)
          & $5.98\cdot 10^{-5}$
          & $1.02\cdot 10^{-4}$ \\
  3(a) TR-RB NCD-corrected
    & 71.84 (50.38/87.16)
      & 4.63
        & 7.40 (5/9)
          & $1.09\cdot 10^{-3}$
          & $6.12\cdot 10^{-5}$ \\
	\end{tabular}
	\vspace{0.15cm}
	\end{subtable}
	\vspace{0.15cm}
	\begin{subtable}{\textwidth}
	\centering
		\begin{tabular}{l|cc|c|cc}
      \textsc{\textbf{(B) Result for $\tau_{\text{\rm{FOC}}}= 10^{-6}$}}
	    & av.~(min/max) runtime[s]
	      & speed-up
	        & av.~(min/max) iter.
	          & rel.~error
	            & FOC cond.\\
	  \hline
	  FOM proj.~BFGS
	    & 409.28 (317.25/637.55)
	      & --
	        & 57.00 (49/71)
	          & $2.82\cdot 10^{-6}$
	          & $3.35\cdot 10^{-7}$\\
	  TR-RB from \cite{QGVW2017}
	    & 614.81 (566.66/671.97)
	      & 0.66
	        & 40.00 (40/40)
	          & $8.46\cdot 10^{-7}$
	            & $8.44\cdot 10^{-8}$\\
	  1(a) TR-RB with $V_\red^\pr \neq V_\red^\du$
	    & 165.48 (92.26/417.24)
	      & 2.47
	        & 15.30 (10/40)
	          & $3.29\cdot 10^{-6}$
	            & $5.43\cdot 10^{-7}$\\
	  1(b) TR-RB with $V_\red^\pr = V_\red^\du$
	    & 86.39 (62.68/124.43)
	      & 4.74
	        & 7.80 (6/10)
	          & $3.52\cdot 10^{-6}$
	          & $3.03\cdot 10^{-7}$ \\
	  2(a) TR-RB semi NCD-corrected
	    & 90.37 (80.97/102.60)
	      & 4.53
	        & 9.80 (9/11)
	          & $8.12\cdot 10^{-7}$
	          & $2.26\cdot 10^{-7}$ \\
	  3(a) TR-RB NCD-corrected
	    & 88.24 (58.18/108.90)
	      & 4.64
	        & 8.90 (6/10)
	          & $2.65\cdot 10^{-6}$
	          & $2.73\cdot 10^{-7}$ \\
		\end{tabular}
		\vspace{0.15cm}
	\end{subtable}
	\caption{%
    {\small Performance and accuracy of selected algorithms for two choices of $\tau_\text{\rm{FOC}}$ for the example from Sec.~\ref{sec:mmexc_opt_results} for ten optimization runs with random initial guess, compare Table \ref{table:fin}.}
  }
  \label{table:MP2_com}
\end{table}

From Table~\ref{table:MP2_com}, we observe that all proposed TR-RB methods speed up the FOM projected BFGS method with
the NCD-corrected approach outperforming the others,
since the gradient used is the true one of the model function $\Jhat_\red$.
Moreover, independently of the model function, the algorithm from \cite{QGVW2017} is much slower,
demonstrating the positive impact of the suggested improvements on
enlarging the TR radius and on the termination criterium based on cheaply available FOM information
(instead of relying on an a posteriori estimate),
also visible in the
number of outer TR iterations.
Comparing our proposed TR variants in terms of iterations,
it is more beneficial to consider a single RB space, i.e.~$V^\pr_\red = V^\du_\red$.
While enrichment (a) is more costly and the time-to-ROM-solution is slightly larger, the richer space seems to allow for better approximations of $\Jhat_h$.

All methods approximate the optimal parameter $\bar\mu$ with a small relative error and reach the desired tolerance for the FOC condition.
However, in view of the resulting relative error in Table~\ref{table:MP2_com} and Figure~\ref{Fig:_FOC_cond}, we observe that the choice $\tau_\text{\rm{FOC}}= 5\cdot10^{-4}$ is not sufficiently small for this model problem. In fact, we observe for most of the variants, that this choice for the tolerance $\tau_\text{\rm{FOC}}$ does not guarantee an adequately low relative error in approximating $\bar\mu$ and
 affects the timings by stopping the method too early. 
We conclude that the choice $\tau_{\text{\rm{FOC}}}= 10^{-6}$ instead results in a valid optimum of all variants (up to a tolerance of $10^{-6}$).
Importantly, for this choice of $\tau_\text{\rm{FOC}}$, we point out that the variant from \cite{QGVW2017} only stopped because
we restricted the maximum number of iterations to $40$, although the FOC condition dropped under the depicted tolerance of $10^{-6}$. 
 \begin{figure}[h]
	\centering\footnotesize
	\hspace{-0.6cm}
	\begin{subfigure}{0.45\textwidth}
		\input{Pictures/mu_error_EXC_10_.tex} 	
	\end{subfigure}
	\hspace{0.6cm}
	\begin{subfigure}{0.45\textwidth}
		\input{Pictures/mu_error_EXC_10_1e-6.tex}
	\end{subfigure}
	\caption{{\small Error decay and performance of selected algorithms for two choices of $\tau_{\text{\rm{FOC}}}$ for the example from Section~\ref{sec:mmexc_opt_results} for a single optimization run with random initial guess, compare Figure \ref{fig:fin_timing_mu_d}.}
	}
	\label{Fig:_FOC_cond}
\end{figure}
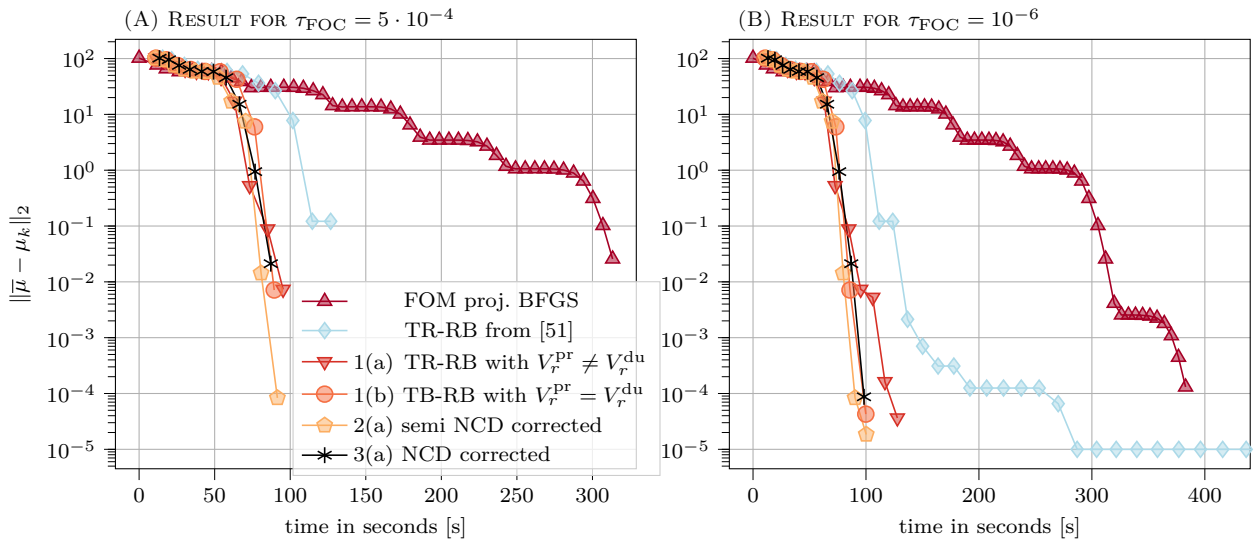
This is caused by the fact that in \cite{QGVW2017} the a posteriori estimate,
which is summed to the FOC condition, cannot get numerically small enough, showing the limit of the proposed stopping criterium in \cite{QGVW2017}. From Figure~\ref{Fig:_FOC_cond}(B) we conclude that the NCD-corrected approaches 2(\ref{enrich:lag}) and 3(\ref{enrich:lag}) outperform the standard ROM variant 1(\ref{enrich:lag}), which also
reached the maximum number of iterations for one of the ten samples. 
Consequently, the NCD-correction entirely resolves the issue of the variational crime (introduced by splitting the reduced spaces), since it shows roughly the same performance as
variant 1(\ref{enrich:single}). However, looking at the minimum and maximum number of computational time in Table~\ref{table:MP2_com}, variant 3(\ref{enrich:lag}) shows a less volatile and more robust
behavior. 

\section{Conclusion}
In this work we proposed and analyzed several variants of new adaptive Trust-Region Reduced Basis methods for parameterized partial differential equations.
First, we proved convergence of the modified algorithm in case of additional bilateral constraints on the parameter set, making this method more appealing for real-world applications.
Second, the use of a NCD-corrected RB reduced functional improves the RB approximation compared to the standard approach,
and enables the possibility of using an exact gradient in the case of separate RB spaces (each variant accompanied by rigorous a posteriori error estimates).
This approach turns out to be the most reliable in terms of computational time and accuracy,
outperforming the existing TR-RB method.
Furthermore, the proposed cheaply-computable criteria for enlarging the TR radius and for terminating the iterations ensure a faster convergence.
In future works we are interested in considering the projected Newton method to replace the projected BFGS method used in this contribution.
This leads to additional effort on developing a posteriori estimates for the RB approximation of the hessian and of the optimal parameter.
In addition, we are interested in combining the proposed TR-RB algorithm with localized RB methods for large-scale applications.

{\small
	\bibliographystyle{abbrv}
	\bibliography{bibliography}
}

\end{document}

%% file: Pictures/mu_error_FIN.tex
\begin{tikzpicture}[]

\definecolor{color0}{rgb}{0.65,0,0.15}
\definecolor{color1}{rgb}{0.84,0.19,0.15}
\definecolor{color2}{rgb}{0.96,0.43,0.26}
\definecolor{color3}{rgb}{0.99,0.68,0.38}
\definecolor{color4}{rgb}{1,0.88,0.56}
\definecolor{color5}{rgb}{0.67,0.85,0.91}

\begin{axis}[
legend cell align={left},
legend style={fill opacity=0.8, draw opacity=1, text opacity=1, at={(0.99,0.01)}, anchor=south east, draw=white!80!black},
log basis y={10},
tick align=outside,
tick pos=left,
x grid style={white!69.0196078431373!black},
xlabel={time in seconds [s]},
xmajorgrids,
xmin=-5.7143321871758, xmax=236.000975930691,
xtick style={color=black},
y grid style={white!69.0196078431373!black},
ylabel={$\big\| \overline{\mu}-\mu^{(k)} \big\|_2$},
ymajorgrids,
ymin=1.91901353329293e-07, ymax=27.8370378432078,
ymode=log,
ytick style={color=black},
y=0.475cm
]
\addplot [semithick, color0, mark=triangle*, mark size=3, mark options={solid, fill opacity=0.5}]
table {%
0 9.99473140636549
4.34884524345398 7.35446882820324
18.6473112106323 7.3541512685285
30.8121237754822 7.35749530443356
32.9880452156067 6.31413726386274
47.2815909385681 6.31413448410398
58.816871881485 6.31207321123728
61.0466980934143 5.42208140924561
75.8011300563812 5.42204355619545
87.570513010025 5.43839223519159
93.4977016448975 5.12334963173444
108.184760808945 5.12324262569418
118.464211463928 5.12391766323573
122.161113739014 6.18793044571934
136.850122451782 6.18787453046038
139.825917005539 4.10890037737276
142.081712961197 3.43536524159338
145.785378932953 2.91800394897423
160.438888788223 2.9179971551271
162.668388605118 2.07914534719756
164.919844388962 1.947091236104
179.623940706253 1.9470783493143
181.859387397766 0.422463567698884
184.128856897354 0.643786083343712
186.372574567795 0.00702386653971826
201.036522626877 0.00702382140842531
};
\addlegendentry{\hspace{16pt} FOM proj.~BFGS}
\addplot [semithick, color5, mark=diamond*, mark size=3, mark options={solid, fill opacity=0.5}]
table {%
3.60644197463989 9.99473140636549
6.18412208557129 8.91952684680866
10.1918063163757 8.91275320910397
14.3757772445679 8.91255198939559
19.1285471916199 8.9086116684482
24.5155129432678 8.90683550331313
29.9072082042694 8.45403937547982
35.5541481971741 6.77948120396427
41.6052765846252 6.3782435509005
48.167160987854 3.11373071399215
55.3844101428986 1.40622934454799e-05
};
\addlegendentry{\hspace{19pt}TR-RB from \cite{QGVW2017}}
\addplot [semithick, color1, mark=triangle*, mark size=3, mark options={solid,rotate=180, fill opacity=0.5}]
table {%
2.86481094360352 9.99473140636549
5.50435853004456 8.91952684680866
8.40358209609985 8.91275320910397
11.4713573455811 8.91255198939559
15.107709646225 8.90773079070491
19.0532963275909 8.34476797197126
23.4172186851501 7.18301515662573
28.4373028278351 1.71823629414235
34.2188301086426 7.84148139372534e-06
};
\addlegendentry{1(a)\hspace{1pt} TR-RB with $V_\red^\pr \neq V_\red^\du$}
\addplot [semithick, color2, mark=*, mark size=3, mark options={solid, fill opacity=0.5}]
table {%
2.82310795783997 9.99473140636549
5.37446165084839 8.91952684680867
8.26291704177856 8.91275320910398
11.3642964363098 8.9125519893956
14.8835608959198 8.90773079070492
18.759311914444 8.34476797199037
23.1022231578827 7.18301515637139
28.0834274291992 1.71823629453389
34.0882914066315 7.84148159296135e-06
};
\addlegendentry{1(b) TB-RB with $V_\red^\pr = V_\red^\du$}
\end{axis}

\end{tikzpicture}

%% file: Pictures/estimator_study.tex
\begin{tikzpicture}

\definecolor{color0}{rgb}{0.65,0,0.15}
\definecolor{color1}{rgb}{0.84,0.19,0.15}
\definecolor{color2}{rgb}{0.96,0.43,0.26}
\definecolor{color3}{rgb}{0.99,0.68,0.38}
\definecolor{color4}{rgb}{1,0.88,0.56}
\definecolor{color5}{rgb}{0.67,0.85,0.91}
\definecolor{color6}{rgb}{0.27,0.46,0.71}
\definecolor{color7}{rgb}{0.19,0.21,0.58}

\begin{axis}[
  name=top_left,
width=6.5cm,
height=4.35cm,
legend cell align={left},
legend style={fill opacity=0.8, draw opacity=1, text opacity=1, at={(1.2,0)}, anchor=south, draw=white!80!black},
log basis y={10},
tick align=outside,
ytick pos=right,
x grid style={white!69.0196078431373!black},
xmajorgrids,
xmin=15, xmax=57,
xtick style={color=black},
y grid style={white!69.0196078431373!black},
ymajorgrids,
ymin=1e-05, ymax=50000,
ymode=log,
ytick style={color=black}
]
\addplot [semithick, color4, mark=*, mark size=2, mark options={solid, fill opacity=0.5}]
table {%
16 1.5602058
24 0.0913711
32 0.0766135
40 0.0203013
48 0.0069036
56 0.0049428
};
\addlegendentry{error J}
\addplot [semithick, color1, mark=triangle*, mark size=2, mark options={solid, rotate=180, fill opacity=0.5}]
table {%
16 0.0960233
24 0.0078391
32 0.0060158
40 0.0000622
48 0.0000290
56 0.0000253
};
\addlegendentry{error J-corr}
\addplot [semithick, color4, mark=*, mark size=2, dashed, mark options={solid, fill opacity=0.5}]
table {%
16 5.9913766
24 0.1355419
32 0.1245623
40 0.0218672
48 0.0070929
56 0.0050998
};
\addlegendentry{estimator J}
\addplot [semithick, color1, mark=triangle*, mark size=2, dashed, mark options={solid, rotate=180, fill opacity=0.5}]
table {%
16 5.1036908
24 0.1120268
32 0.0782738
40 0.0024091
48 0.0001877
56 0.0001684
};
\addlegendentry{estimator J corr}
\legend{};
\end{axis}

\begin{axis}[
  name=top_right,
  at=(top_left.east),
  anchor=west,
  xshift=1.2cm,
width=6.5cm,
height=4.35cm,
legend cell align={left},
legend style={fill opacity=0.8, draw opacity=1, text opacity=1, at={(1.2,0)}, anchor=south, draw=white!80!black},
log basis y={10},
tick align=outside,
x grid style={white!69.0196078431373!black},
xmajorgrids,
xmin=15, xmax=57,
xtick style={color=black},
y grid style={white!69.0196078431373!black},
ymajorgrids,
ymin=1e-05, ymax=50000,
ymode=log,
yticklabels={,,},
ytick pos=left,
ytick style={color=black}
]
\addplot [semithick, color0, mark=*, mark size=2, mark options={solid, fill opacity=0.5}]
table {%
16 60.3425755
24 11.1470314
32 6.4303011
40 0.4347226
48 0.3154153
56 0.3015742
};
\addlegendentry{error DJ}
\addplot [semithick, color3, mark=triangle*, mark size=2, mark options={solid, rotate=180, fill opacity=0.5}]
table {%
16 6.3335463
24 0.7063809
32 0.3138649
40 0.0082424
48 0.0067837
56 0.0063186
};
\addlegendentry{error DJ-corr}
\addplot [semithick, color6, mark=pentagon*, mark size=2, mark options={solid, rotate=90, fill opacity=0.5}]
table {%
16 1.2382556
24 0.0897214
32 0.0900121
40 0.0006126
48 0.0002131
56 0.0002262
};
\addlegendentry{error DJ-approx}
\addplot [semithick, color0, mark=*, mark size=2, dashed, mark options={solid, fill opacity=0.5}]
table {%
16 24291.8313696
24 5358.8100309
32 4503.6642106
40 756.5755578
48 321.9380287
56 297.4647676
};
\addlegendentry{estimator DJ}
\addplot [semithick, color3, mark=triangle*, mark size=2, dashed, mark options={solid, rotate=180, fill opacity=0.5}]
table {%
16 636.4101662
24 15.3648749
32 10.5050809
40 0.4986274
48 0.0712033
56 0.0608080
};
\addlegendentry{estimator DJ-corr}
\addplot [semithick, color6, mark=pentagon*, mark size=2, dashed, mark options={solid, rotate=90, fill opacity=0.5}]
table {%
16 513.9431314
24 12.7039995
32 8.8771813
40 0.2603803
48 0.0277693
56 0.0249135
};
\addlegendentry{estimator DJ-approx}
\legend{};
\end{axis}

\begin{axis}[
  name=bottom_left,
  at=(top_left.south east),
  anchor=north east,
  yshift=-0.725cm,
width=6.5cm,
height=2.975cm,
legend cell align={left},
legend style={fill opacity=0.8, draw opacity=1, text opacity=1, at={(1.2,0)}, anchor=south, draw=white!80!black},
log basis y={10},
tick align=outside,
ytick pos=right,
x grid style={white!69.0196078431373!black},
xmajorgrids,
xmin=15, xmax=57,
xtick style={color=black},
xticklabels={,,,},
y grid style={white!69.0196078431373!black},
ymajorgrids,
ymin=1e-05, ymax=1,
ymode=log,
ytick style={color=black}
]
\addplot [semithick, color4, mark=*, mark size=2, mark options={solid, fill opacity=0.5}]
table {%
16 0.00011436177044914534
24 0.004825509741466907
32 0.01655434134961476
40 0.0647896076381
48 0.245075189558
56 0.00681446824893
};
\addplot [semithick, color1, mark=triangle*, mark size=2, mark options={solid, rotate=180, fill opacity=0.5}]
table {%
16 0.00011197998752053878
24 0.0005131116316516528
32 0.00031366189555426645
40 7.94490291925e-05
48 0.000535968476775
56 0.000163565209881
};
\legend{};
\end{axis}

\begin{axis}[
  name=bottom_right,
  at=(top_right.south west),
  anchor=north west,
  yshift=-0.725cm,
width=6.5cm,
height=2.975cm,
legend cell align={left},
legend style={fill opacity=0.8, draw opacity=1, text opacity=1, at={(1.2,0)}, anchor=south, draw=white!80!black},
log basis y={10},
tick align=outside,
tick pos=left,
x grid style={white!69.0196078431373!black},
xmajorgrids,
xmin=15, xmax=57,
xtick style={color=black},
xticklabels={,,,},
y grid style={white!69.0196078431373!black},
ymajorgrids,
yticklabels={,,},
ytick pos=left,
ymin=1e-05, ymax=1,
ymode=log,
ytick style={color=black}
]
\addplot [semithick, color0, mark=*, mark size=2, mark options={solid, fill opacity=0.5}]
table {%
16 2.3101324284381392e-05
24 7.500823933255525e-05
32 6.408900297593858e-05
40 2.25887468754e-05
48 1.40071072069e-05
56 1.29833821353e-05
};
\addplot [semithick, color3, mark=triangle*, mark size=2, mark options={solid, rotate=180, fill opacity=0.5}]
table {%
16 0.0001466538071330275
24 0.0009869744430820331
32 0.0015309285077020757
40 0.000739227006107
48 0.00107934795818
56 0.00139588710316
};
\addplot [semithick, color6, mark=pentagon*, mark size=2, mark options={solid, rotate=90, fill opacity=0.5}]
table {%
16 1.0535255069876036e-05
24 0.0003666168336441947
32 0.0001877566150757363
40 3.29876687934e-05
48 0.000184414981217
56 0.0003891282978
};
\legend{};
\end{axis}

\begin{customlegend}[legend cell align={left}, legend style={fill opacity=0.8, draw opacity=1, text opacity=1,
    at=(top_left.north west),
    anchor=north east,
    xshift=-5pt,
    inner sep=5pt,
draw=white!80!black},
	legend entries={
    $|\hat{J}_h - \Jnoncor_\red|$,
    $|\hat{J}_h - \cJhatn|$,
    $\Delta_{\Jnoncor_\red}$, 
    $\Delta_{\cJhatn}$, 
  }]
\addlegendimage{semithick, color4, mark=*, mark size=2, mark options={solid, fill opacity=0.5}}
\addlegendimage{semithick, color1, mark=triangle*, mark size=2, mark options={solid, rotate=180, fill opacity=0.5}}
\addlegendimage{semithick, color4, mark=*, mark size=2, dashed, mark options={solid, fill opacity=0.5}}
\addlegendimage{semithick, color1, mark=triangle*, mark size=2, dashed, mark options={solid, rotate=180, fill opacity=0.5}}
\end{customlegend}

\begin{customlegend}[legend cell align={left}, legend style={fill opacity=0.8, draw opacity=1, text opacity=1,
    at=(top_right.north east),
    anchor=north west,
    xshift=5pt,
    inner sep=3pt,
draw=white!80!black},
	legend entries={ 
  $\|\nabla \hat{\mathcal{J}}_h - \noncorgrad\Jnoncor_\red\|_2$,
  $\|\nabla \hat{\mathcal{J}}_h - \nabla\cJhatn\|_2$,
  $\|\nabla \hat{\mathcal{J}}_h - \noncorgrad\cJhatn\|_2$,
	$\Delta_{\noncorgrad\Jnoncor_\red}$,
  $\Delta_{\nabla\cJhatn}$, 
	$\Delta_{\noncorgrad\cJhatn}$,
	}]
\addlegendimage{semithick, color0, mark=*, mark size=2, mark options={solid, fill opacity=0.5}}
\addlegendimage{semithick, color3, mark=triangle*, mark size=2, mark options={solid, rotate=180, fill opacity=0.5}}
\addlegendimage{semithick, color6, mark=pentagon*, mark size=2, mark options={solid, rotate=90, fill opacity=0.5}}
\addlegendimage{semithick, color0, mark=*, mark size=2, dashed, mark options={solid, fill opacity=0.5}}
\addlegendimage{semithick, color3, mark=triangle*, mark size=2, dashed, mark options={solid, rotate=180, fill opacity=0.5}}
\addlegendimage{semithick, color6, mark=pentagon*, mark size=2, dashed, mark options={solid, rotate=90, fill opacity=0.5}}
\end{customlegend}

\begin{customlegend}[legend cell align={left}, legend style={fill opacity=0.8, draw opacity=1, text opacity=1,
    at=(bottom_left.south west),
    anchor=south east,
    xshift=-5pt,
    inner sep=5pt,
draw=white!80!black},
	legend entries={
    $\Delta_{\Jnoncor_\red}$ eff., 
    $\Delta_{\cJhatn}$ eff., 
  }]
\addlegendimage{semithick, color4, mark=*, mark size=2, mark options={solid, fill opacity=0.5}}
\addlegendimage{semithick, color1, mark=triangle*, mark size=2, mark options={solid, rotate=180, fill opacity=0.5}}
\end{customlegend}

\begin{customlegend}[legend cell align={left}, legend style={fill opacity=0.8, draw opacity=1, text opacity=1,
    at=(bottom_right.south east),
    anchor=south west,
    xshift=5pt,
    inner sep=3pt,
draw=white!80!black},
	legend entries={ 
	$\Delta_{\noncorgrad\Jnoncor_\red}$ eff.,
  $\Delta_{\nabla\cJhatn}$ eff., 
	$\Delta_{\noncorgrad\cJhatn}$ eff.,
	}]
\addlegendimage{semithick, color0, mark=*, mark size=2, mark options={solid, fill opacity=0.5}}
\addlegendimage{semithick, color3, mark=triangle*, mark size=2, mark options={solid, rotate=180, fill opacity=0.5}}
\addlegendimage{semithick, color6, mark=pentagon*, mark size=2, mark options={solid, rotate=90, fill opacity=0.5}}
\end{customlegend}

\node[anchor=north, yshift=-2 pt] at (bottom_left.south) {greedy extension step};
\node[anchor=north, yshift=-2 pt] at (bottom_right.south) {greedy extension step};
\node[anchor=south east, yshift=4pt] at (top_left.north east) {(A) approximating the reduced functional $\Jhat_h$};
\node[anchor=south west, yshift=4pt] at (top_right.north west) {(B) approximating the gradient of the reduced functional};

\end{tikzpicture}

%% file: Pictures/mu_error_EXC_10_.tex
\begin{tikzpicture}

\definecolor{color0}{rgb}{0.65,0,0.15}
\definecolor{color1}{rgb}{0.84,0.19,0.15}
\definecolor{color2}{rgb}{0.96,0.43,0.26}
\definecolor{color3}{rgb}{0.99,0.68,0.38}
\definecolor{color4}{rgb}{1,0.88,0.56}
\definecolor{color5}{rgb}{0.67,0.85,0.91}

\begin{axis}[
  name=left,
legend cell align={left},
legend style={fill opacity=0.8, draw opacity=1, text opacity=1, at={(0.34,0.21)}, anchor=west, draw=white!80!black},
log basis y={10},
tick align=outside,
tick pos=left,
x grid style={white!69.0196078431373!black},
xlabel={time in seconds [s]},
xmajorgrids,
xmin=-15.6670947432518, xmax=329.008989608288,
xtick style={color=black},
y grid style={white!69.0196078431373!black},
ylabel={\(\displaystyle \| \overline{\mu}-\mu_k \|_2\)},
ymajorgrids,
ymin=4.4843436808366e-06, ymax=220.993014467563,
ymode=log,
ytick style={color=black}
]
\addplot [semithick, color0, mark=triangle*, mark size=3, mark options={solid, fill opacity=0.5}]
table {%
0 101.694749908925
11.8720343112946 76.6344328826546
17.9447810649872 65.5067669990911
26.5323231220245 58.1295253774286
35.2170124053955 59.0115951280537
41.6089336872101 58.4939381433727
48.0273129940033 58.3395819689597
54.4964914321899 57.4951884391855
61.2268347740173 44.9083655387143
67.8027763366699 38.7433342505646
74.2297432422638 30.1132884189381
80.7131202220917 30.541943990057
87.1376538276672 30.5797403707051
95.7394621372223 30.6877367833582
102.287242174149 30.5139572509563
108.70614361763 29.3771285500316
115.19947886467 26.5388565600953
121.608815670013 22.560037310177
128.0400390625 14.3479859523988
134.460395097733 13.5804698329301
140.89999961853 13.6677252430272
147.32558631897 13.6854911537257
153.684593200684 13.6702499690576
160.108421325684 13.5504640651469
166.547456979752 12.5237935862747
172.86968755722 10.1390172370278
179.272742271423 6.47893809307185
186.017577886581 3.91162947760179
192.293509244919 3.46480482852387
198.609087705612 3.46586360401667
205.003822803497 3.46694944167388
211.310054540634 3.46012010437098
217.55388212204 3.40591532107205
223.888615369797 3.22469574742039
230.194384098053 2.70155535774846
236.482502937317 1.84053812825313
242.779143333435 1.17557279387446
249.087074041367 1.05901236143825
255.489965438843 1.06020275336439
261.867829084396 1.0611800350202
268.231870174408 1.0603964339395
274.532961368561 1.05006611015456
280.875968456268 1.00745053958016
287.519969940186 0.890082498254083
293.972450733185 0.639372370454038
300.295924663544 0.311869643001277
306.908550024033 0.102122548506474
313.341894865036 0.0256872520461831
};
\addlegendentry{\hspace{16pt} FOM proj.~BFGS}
\addplot [semithick, color5, mark=diamond*, mark size=3, mark options={solid, fill opacity=0.5}]
table {%
15.5775761604309 101.694749914018
21.5570247173309 95.2544133844278
29.9755818843842 75.4099308159367
38.8098907470703 64.0679273362064
48.2749555110931 58.9002429218621
58.2099120616913 58.7665520191065
68.4788293838501 53.1967067086869
79.0315234661102 36.1345643809206
90.1136507987976 26.5812213583053
101.746352672577 7.77184075151648
114.605060577393 0.121698707208084
126.764373779297 0.12139085075229
};
\addlegendentry{\hspace{19pt}TR-RB from \cite{QGVW2017}}
\addplot [semithick, color1, mark=triangle*, mark size=3, mark options={solid,rotate=180, fill opacity=0.5}]
table {%
11.1392743587494 101.694749914018
17.11465883255 95.2544133844278
23.5293519496918 75.4099308159367
30.4334824085236 64.0679273362064
37.9646141529083 58.8254086036674
45.8417568206787 57.9178213150213
54.239414691925 45.7674046218088
63.0553381443024 17.1659141983604
73.1298186779022 0.528120998554676
85.0742800235748 0.088572996002305
95.2746829986572 0.00733824299393689
};
\addlegendentry{1(a)\hspace{1pt} TR-RB with $V_\red^\pr \neq V_\red^\du$}
\addplot [semithick, color2, mark=*, mark size=3, mark options={solid, fill opacity=0.5}]
table {%
11.18288397789 101.694749914018
17.8009340763092 95.0951073301147
25.4488184452057 75.0102498060299
34.0624237060547 63.7333999506352
43.5729718208313 58.8433158089068
53.5242161750793 58.0198982797859
64.5801253318787 42.5520357574727
76.3733286857605 5.97831025727406
89.4082858562469 0.00714132377950299
};
\addlegendentry{1(b) TB-RB with $V_\red^\pr = V_\red^\du$}
\addplot [semithick, color3, mark=pentagon*, mark size=3, mark options={solid, fill opacity=0.5}]
table {%
10.8595275878906 101.694749914018
16.7021052837372 95.2544133844278
22.9807689189911 75.4099308159367
29.6670250892639 64.0679273362064
36.9445209503174 58.8254086036674
44.4493291378021 57.9178213150213
52.4963920116425 45.7674046218088
61.2372426986694 17.0605235487557
70.4567382335663 7.47008221969308
80.7195756435394 0.0142931366807365
91.5153188705444 8.33143432198402e-05
};
\addlegendentry{2(a) semi NCD corrected}
\addplot [semithick, black, mark=asterisk, mark size=3, mark options={solid, fill opacity=0.5}]
table {%
13.4768707752228 101.694749914018
19.7069566249847 95.2544133844209
26.4656765460968 75.3587869929344
33.602260351181 63.9788199949804
41.4095425605774 58.859685583101
49.2445447444916 58.2492967427756
57.5780880451202 45.4725937425856
66.3905770778656 15.1974212649069
76.7422344684601 0.942934319070116
87.1576807498932 0.021167321405947
};
\addlegendentry{3(a) NCD corrected}
\end{axis}

\node[anchor=south west] at (left.north west) {\textsc{(A) Result for $\tau_{\text{\rm{FOC}}}=5\cdot 10^{-4}$}};
\end{tikzpicture}

%% file: Pictures/mu_error_EXC_10_1e-6.tex
\begin{tikzpicture}

\definecolor{color0}{rgb}{0.65,0,0.15}
\definecolor{color1}{rgb}{0.84,0.19,0.15}
\definecolor{color2}{rgb}{0.96,0.43,0.26}
\definecolor{color3}{rgb}{0.99,0.68,0.38}
\definecolor{color4}{rgb}{1,0.88,0.56}
\definecolor{color5}{rgb}{0.67,0.85,0.91}

\begin{axis}[
  name=right,
legend cell align={left},
legend style={fill opacity=0.8, draw opacity=1, text opacity=1, at={(1,0.5)}, anchor=west, draw=white!80!black},
log basis y={10},
tick align=outside,
tick pos=left,
x grid style={white!69.0196078431373!black},
xlabel={time in seconds [s]},
xmajorgrids,
xmin=-20.960649907589, xmax=440.173648059368,
xtick style={color=black},
y grid style={white!69.0196078431373!black},
ymajorgrids,
ymin=4.4843436808366e-06, ymax=220.993014467563,
ymode=log,
ytick style={color=black}
]
\addplot [semithick, color0, mark=triangle*, mark size=3, mark options={solid, fill opacity=0.5}]
table {%
0 101.694749908925
11.8227753639221 76.6344328826546
18.0333218574524 65.5067669990911
26.6583814620972 58.1295253774286
35.0693249702454 59.0115951280537
41.360090970993 58.4939381433727
47.6699612140656 58.3395819689597
54.0571613311768 57.4951884391855
60.475839138031 44.9083655387143
66.8279640674591 38.7433342505646
73.1730172634125 30.1132884189381
79.8124206066132 30.541943990057
86.1816501617432 30.5797403707051
94.744637966156 30.6877367833582
101.086725711823 30.5139572509563
107.417257547379 29.3771285500316
113.785773038864 26.5388565600953
120.147310733795 22.560037310177
126.465299129486 14.3479859523988
132.743663549423 13.5804698329301
139.105927705765 13.6677252430272
145.466085195541 13.6854911537257
151.809690237045 13.6702499690576
158.101904153824 13.5504640651469
164.427534341812 12.5237935862747
170.725522756577 10.1390172370278
177.048982620239 6.47893809307185
183.360107898712 3.91162947760179
189.69978260994 3.46480482852387
196.036733865738 3.46586360401667
202.56974029541 3.46694944167388
208.870253324509 3.46012010437098
215.07564496994 3.40591532107205
221.397597789764 3.22469574742039
227.699336051941 2.70155535774846
233.960841655731 1.84053812825313
240.292472839355 1.17557279387446
246.58556842804 1.05901236143825
252.882767438889 1.06020275336439
259.096759080887 1.0611800350202
265.451264858246 1.0603964339395
271.741575241089 1.05006611015456
278.136191129684 1.00745053958016
284.483489274979 0.890082498254083
290.997215986252 0.639372370454038
297.312686681747 0.311869643001277
305.364391565323 0.102122548506474
312.132185220718 0.0256872520461831
319.721675395966 0.00413078425544472
326.131166934967 0.00255412308842596
332.34970498085 0.00254610141212139
338.603756427765 0.0025429009369023
344.915750026703 0.0025127226206382
351.409331083298 0.0024398138461593
357.721709489822 0.00223763288468792
364.101721763611 0.00179700448435439
370.322831630707 0.00108508397827876
376.664766073227 0.000445536313619211
382.943562269211 0.000130715792809606
};
\addlegendentry{5. FOM proj. BFGS}
\addplot [semithick, color5, mark=diamond*, mark size=3, mark options={solid, fill opacity=0.5}]
table {%
15.0834636688232 101.694749914018
20.8926212787628 95.2544133844278
29.018415927887 75.4099308159367
37.7143726348877 64.0679273362064
46.7216105461121 58.9002429218621
56.0600709915161 58.7665520191065
66.2690076828003 53.1967067086869
76.5717370510101 36.1345643809206
87.8098843097687 26.5812213583053
99.1848306655884 7.77184075151648
111.681534767151 0.121698707208084
123.707145929337 0.12139085075229
136.807378292084 0.00213458006709451
150.148610115051 0.00070265838616608
163.837367534637 0.000310413647634123
177.788628339767 0.000310251084753871
192.031569719315 0.000125322245049535
206.855760335922 0.000125272954182688
221.766639947891 0.000125185538394569
237.465733289719 0.000125184171346979
253.313305139542 0.00012501858548298
270.196664571762 6.56715239340949e-05
286.975438117981 9.98241579955393e-06
304.209729194641 9.98109460785684e-06
321.546380996704 9.98107397048428e-06
339.791335582733 9.98107396716918e-06
358.050210237503 9.98090889164456e-06
376.922388076782 9.98090760207332e-06
396.134050607681 9.98090756246715e-06
415.789769649506 9.98082502525526e-06
436.588827371597 9.98082486641714e-06
455.848903417587 9.98082486641714e-06
475.412249565125 9.98082486641714e-06
496.140789747238 9.98082486641714e-06
514.43416762352 9.98082486641714e-06
532.975608110428 9.98082486641714e-06
551.482770204544 9.98082486641714e-06
571.509459733963 9.98082486641714e-06
590.078435897827 9.98082486641714e-06
608.564557790756 9.98082486641714e-06
627.022351980209 9.98082486641714e-06
};
\addlegendentry{Qian et al. 2017}
\addplot [semithick, color1, mark=triangle*, mark size=3, mark options={solid,rotate=180, fill opacity=0.5}]
table {%
10.9836044311523 101.694749914018
17.0856072902679 95.2544133844278
23.7086544036865 75.4099308159367
30.5910837650299 64.0679273362064
38.4077038764954 58.8254086036674
46.041802406311 57.9178213150213
54.2439885139465 45.7674046218088
62.901976108551 17.1659141983604
72.6504902839661 0.528120998554676
84.7693104743958 0.088572996002305
95.3249473571777 0.00733824299393689
105.960053682327 0.00527301821282116
116.768208503723 0.000162636786332236
127.789884090424 3.64077041750881e-05
};
\addlegendentry{1(a) standard lag.}
\addplot [semithick, color2, mark=*, mark size=3, mark options={solid, fill opacity=0.5}]
table {%
10.6433067321777 101.694749914018
17.1916410923004 95.0951073301147
24.9593110084534 75.0102498060299
33.0756437778473 63.7333999506352
42.0009367465973 58.8433158089068
51.5410561561584 58.0198982797859
62.0700180530548 42.5520357574727
73.2681591510773 5.97831025727406
85.8639888763428 0.00714132377950299
99.8546583652496 4.2715167840283e-05
};
\addlegendentry{1(b) standard uni.}
\addplot [semithick, color3, mark=pentagon*, mark size=3, mark options={solid, fill opacity=0.5}]
table {%
11.3806500434875 101.694749914018
17.2724416255951 95.2544133844278
23.6917324066162 75.4099308159367
30.6141812801361 64.0679273362064
37.9904551506042 58.8254086036674
45.2538509368896 57.9178213150213
53.1298584938049 45.7674046218088
61.3701057434082 17.0605235487557
70.0314295291901 7.47008221969308
79.8150877952576 0.0142931366807365
90.3450720310211 8.33143432198402e-05
100.228097200394 1.84661127400357e-05
};
\addlegendentry{2(a) semi NCD lag.}
\addplot [semithick, black, mark=asterisk, mark size=3, mark options={solid, fill opacity=0.5}]
table {%
13.2413511276245 101.694749914018
19.3898150920868 95.2544133844209
26.1817166805267 75.3587869929344
33.0169637203217 63.9788199949804
40.6761891841888 58.859685583101
48.2365906238556 58.2492967427756
56.5340046882629 45.4725937425856
65.5767681598663 15.1974212649069
76.236855506897 0.942934319070116
86.953501701355 0.021167321405947
98.2464504241943 8.72644160303025e-05
};
\addlegendentry{3(a) NCD lag.}
\legend{};
\end{axis}

\node[anchor=south west] at (right.north west) {\textsc{(B) Result for $\tau_{\text{\rm{FOC}}}= 10^{-6}$}};
\end{tikzpicture}